\documentclass{scrartcl}

\usepackage{amsmath}
\usepackage{amsthm}
\usepackage{amssymb}
\usepackage{stmaryrd}

\usepackage[msc-links]{amsrefs}
\usepackage{mathrsfs}

\usepackage{hyperref}

\usepackage{tikz} 
\usepackage{tikz-cd}

\usepackage{bm}

\theoremstyle{plain}
\newtheorem{theorem}{Theorem}[section]
\newtheorem{lemma}[theorem]{Lemma}
\newtheorem{proposition}[theorem]{Proposition}
\newtheorem{corollary}[theorem]{Corollary}

\theoremstyle{definition}
\newtheorem{definition}[theorem]{Definition}

\newtheorem{remark}[theorem]{Remark}
\newtheorem{notation}[theorem]{Notation}
\newtheorem{claim}{\indent Claim}

\newtheorem{step}{\indent Step}

\newtheorem{convention}{\indent Convention}

\newcommand{\A}{\mathbb{A}}
\newcommand{\C}{\mathbb{C}}
\newcommand{\F}{\mathbb{F}}

\newcommand{\Q}{\mathbb{Q}}
\newcommand{\R}{\mathbb{R}}

\newcommand{\Z}{\mathbb{Z}}

\newcommand{\fraka}{\mathfrak{a}}

\newcommand{\frakp}{\mathfrak{p}}

\newcommand{\calB}{\mathcal{B}}

\newcommand{\calO}{\mathcal{O}}

\newcommand{\bfR}{\mathbf{R}}
\newcommand{\bfK}{\mathbf{K}}

\newcommand{\bfS}{\mathbf{S}}
\newcommand{\bfs}[1]{\mathbf{s}({#1})}
\newcommand{\bfi}[1]{\mathbf{i}({#1})}

\newcommand{\bfV}{\mathbf{V}}
\newcommand{\bfp}{\mathbf{p}}

\newcommand{\taumod}{\tau-\textbf{mod}}
\newcommand{\Fmod}{F-\textbf{mod}}

\newcommand{\Frob}{\mathrm{Frob}}

\newcommand{\Spec}{\mathrm{Spec}}
\DeclareMathOperator{\NP}{\mathrm{NP}}

\DeclareMathOperator{\pr}{pr}

\DeclareMathOperator{\Hom}{\mathrm{Hom}}
\DeclareMathOperator{\Div}{\mathrm{Div}}

\DeclareMathOperator{\Pic}{\mathrm{Pic}}

\newcommand{\Res}{\textup{Res}}

\newcommand{\tame}{\mathrm{t}}

\newcommand{\Addresses}{{
  \bigskip
  \footnotesize

  Joe Kramer-Miller, \textsc{Department of Mathematics, Lehigh University,
    Bethlehem, PA}\par\nopagebreak
  \textit{E-mail address}, Joe Kramer-Miller: \texttt{jjk221@lehigh.edu}

    \bigskip 
    
      James Upton, \textsc{Department of Mathematics, University of California, Santa Cruz,
    Santa Cruz, CA}\par\nopagebreak
  \textit{E-mail address}, James Upton: \texttt{jtupton@ucsc.edu}

}}

\title{Zeros of the Goss zeta function}
\author{Joe Kramer-Miller and James Upton}
\date{\today}

\begin{document}

\maketitle

\abstract{Let $X$ be a smooth proper curve over a finite field and let $\infty \in X$ be a closed point. Let $A$ be the ring of functions on $X - \infty$. The Goss zeta function $\zeta_A$ of $A$ is an equicharacteristic analogue of the Riemann zeta function. In this article we study the zeros of $\zeta_A$ under the generic condition that $X$ is ordinary. We prove an analogue of the Riemann hypothesis, which verifies a corrected version of a conjecture of Goss. We also show that the zeros of $\zeta_A$ at negative even integers are `simple' and that $\zeta_A$ is nonzero at negative odd integers. This answers questions posed by Goss and Thakur. Both of these results were previously only known under the restrictive hypothesis that $A$ has class number one. Finally, we prove versions of these results for $v$-adic interpolations of the Goss zeta function.  }
\tableofcontents

\section{Introduction}

	The Riemann zeta function $\zeta(s) = \sum n^{-s}$ is an object of fundamental importance in number theory. The question of the distribution of its zeros is the content of the famous Riemann hypothesis, which is arguably the most important question in classical number theory. Given the strong analogy between number fields and function fields, it is natural to seek analogues of $\zeta(s)$ and the Riemann hypothesis in positive characteristic.
	
	Let $X$ be a smooth, projective, geometrically connected curve over a finite field $\F_q$ of characteristic $p$ with genus $g$. Let $\infty$ be a closed point in $X$, which we view as the `infinite place', and let $d$ be the degree of $\infty$. We let $A$ denote the coordinate ring of the affine curve $X - \infty$. Let $K(X)$ denote the function field of $X$ and let $K_\infty$ denote the completion of $K(X)$ at the infinite place. We let $\C_\infty$ be the completion of an algebraic closure of $K_\infty$.  The analogy between function-field arithmetic and classical arithmetic can be summarized as follows:
	\begin{align*}
		A &\leftrightarrow \Z,	&	K(X) &\leftrightarrow	\Q,	&	K_\infty	&\leftrightarrow	\R,	&	\C_\infty	\leftrightarrow	\C.
	\end{align*}
    With this analogy in mind, David Goss defined a $\C_\infty$-valued zeta function that plays the role of $\zeta(s)$ (see \cite{Goss-book}*{\S 8} or \cite{Goss-L-series_edition}). The Goss zeta function shares many features with the Riemann zeta function, including the existence of ``trivial zeros'' (see below), transcendence of certain special values for $A=\F_q[\theta]$ (see \cite{Yu} and \cite{Chang-Yu-algebraic-relations-zetavalues}), and an analogue of the Herbrand-Ribet theorem \cite{Goss-Sinnott}. This paper concerns the distribution of its zeros, and in particular the question of a \emph{Riemann hypothesis} for the Goss zeta function.

	\subsection{Background on the zeta function}
	
		\subsubsection{Character groups}
	
			Let $\R^+$ denote the multiplicative group of positive real numbers. Every complex number $s \in \C$ defines a character of $\R^+$ by complex exponentiation $\alpha \mapsto \alpha^s$. This identifies the complex plane with the group $\Hom(\R^+,\C^\times)$ of continuous multiplicative characters of $\R^+$. We may regard the Riemann zeta function as a meromorphic function
			\begin{equation*}
				\zeta:	\Hom(\R^+,\C^\times) \to \C.
			\end{equation*}
		      Note that the absolute value on $\R$ provides a natural identification $\R^+ \cong \R^\times / \langle \pm 1\rangle$. This leads us to the following analogue in characteristic $p$:
			
			\begin{definition}
				The group of positive numbers of $K_\infty$ is the quotient $K_\infty^+ = K_\infty^\times/\mu_\infty$, where $\mu_\infty$ denotes the group of roots of unity in $K_\infty^\times$.
			\end{definition}
			
			The character theory of $K_\infty^+$ is complicated by the presence of the subgroup $U_\infty$ of $1$-units in $K_\infty^\times$. The character group of $U_\infty$ is enormous, since $U_\infty \cong \Z_p^\Z$. The Goss zeta function is defined on a subset of characters of $K_\infty^+$ having a particularly simple action on $U_\infty$: Fix a uniformizer $\pi \in K_\infty^+$. Then $\pi$ determines a factorization
			\begin{equation*}
				K_\infty^\times \cong \mu_\infty \times \pi^\Z \times U_\infty,
			\end{equation*}
			and we may identify $K_\infty^+$ with the subgroup $1 \times \pi^\Z \times U_\infty$. Accordingly, every $\alpha \in K_\infty^+$ factors uniquely
			\begin{equation}
				\alpha	=	\pi^{v_\infty(\alpha)} \cdot \langle \alpha \rangle_\pi,
			\end{equation}
			where $\langle \alpha \rangle_\pi \in U_\infty$. Let us write $\deg(\alpha) = \deg(\infty) \cdot v_\infty(\alpha)$.
			
			\begin{definition}
				The $\pi$-adic \emph{Goss plane} is the topological group $\bfS_\pi = \C_\infty^\times \times \Z_p$. For $s = (x,y) \in \bfS_\pi$, the action of $s$ on a positive number $\alpha \in K_\infty^+$ is given by
				\begin{equation*}
					\alpha^s	=	x^{\deg(\alpha)} \langle \alpha \rangle_\pi^y \in \C_\infty^\times.
				\end{equation*}
			\end{definition}
		
		\subsubsection{The zeta function}
		
			Let $\fraka \subseteq A$ be a non-zero ideal. If $\fraka$ is principal, then every generator of $\fraka$ determines the same positive number $\alpha \in K_\infty^+$ (analogously to the fact that every non-zero ideal $\fraka \subseteq \Z$ has a unique positive generator). Thus, we may regard $\bfS_\pi$ as a group of characters for the group of non-zero principal ideals of $A$. This action of $\bfS_\pi$ on principal ideals extends uniquely to an action on the group of \emph{all} non-zero fractional ideals of $A$ (see \S 2 below or \cite{Goss-book}[\S 8.2]).
			
			\begin{definition}
				The $\pi$-adic Goss zeta function of $A$ is the function on $\bfS_\pi$ defined by the equivalent formal expressions
				\begin{equation*}
					\zeta_{A,\pi}(s)	:=	\sum_\fraka \fraka^{-s}	=	\prod_\frakp	\frac{1}{1-\frakp^{-s}},
				\end{equation*}
				where the sum is taken over non-zero ideals $\fraka \subseteq A$, and the product is taken over non-zero prime ideals $\frakp \subset A$.
                This sum is known to be convergent for all $s \in \bfS_\pi$. Furthermore, if $y \in \Z_p$, then $\zeta_{A,\pi}(-,y)$ is analytic on $\C_\infty$.
			\end{definition}
			
			\begin{remark}
				In \S \ref{s:zeta} we present a ``coordinate-free'' definition of the zeta function denoted by $\zeta_A$. Its domain is a topological group, which we denote by $\bfS_A$. The choice of uniformizer $\pi$ provides an isomorphism $\bfS_A \cong \bfS_\pi$, which identifies $\zeta_A$ with $\zeta_{A,\pi}$.
			\end{remark}

		\subsubsection{A theory of the zeros}
		      The zeros of the Riemann zeta function $\zeta(s)$ can be put into two camps: the trivial zeros at negative integers and those that have (conjecturally) real part $\frac{1}{2}$. It is therefore natural to ask about the `real parts' of zeros of the Goss zeta function and to ask about vanishing at `special values'.
			\paragraph{The real part of zeros} 
            Recall that if $s \in \C$, then for any $\alpha \in \R^+$ we have $|\alpha^s| = |\alpha|^{\mathrm{Re}(s)}$. Similarly, if $s \in \bfS_\infty$ then there exists a unique real number $r(s) \in \R$ such that for any $\alpha \in K_\infty^+$, we have
			\begin{equation*}
				|\alpha^s| = |\alpha|^{r(s)}.
			\end{equation*}
			We regard the parameter $r(s)$ as the characteristic $p$ analogue of the real part of a complex number. In coordinates $s = (x,y)$, we have
			\begin{equation}\label{eq:real-part-to-valuation}
				r(s)	=	d v_\infty(x).
			\end{equation}
            Thus, for $y \in \Z_p$ it is natural to consider the multiset
		\begin{equation*}
			R_A(y)	=	\{r(s): \zeta_A(s) = 0 \text{ and }y(s) = y\},
		\end{equation*}
        where $y(s)$ denotes the $y$-coordinate of $s$. The set $R_A(y)$ 
		consists of ``real parts'' of the zeros of $\zeta_A$ with a given $y$-parameter counted with multiplicity. In light of \eqref{eq:real-part-to-valuation}, knowledge of $R_A(y)$ is equivalent to that of the \emph{Newton polygon} of the entire function $\zeta_{A,\pi}(-,y)$ on $\C_\infty^\times$. The absence of a functional equation for $\zeta_A$ makes it unclear what a characteristic $p$ Riemann hypothesis should look like.
		
		Progress in this direction began with Wan \cite{Wan-Riemann_hypothesis}, who computed the $R_A(y)$ for $A = \F_p[\theta]$, and showed that each ``real part'' occurs with multiplicity $1$. As a corollary to Wan's theorem, one sees that the zeros of $\zeta_{A,\pi}(-,y)$ are all simple and lie in $K_\infty$. In \cite{Goss-RH-2000} Goss conjectured that, barring finitely many exceptions, the elements of $R_A(y)$ should have multiplicity $1$. Evidence for this conjecture was given by work of Diaz-Vargas \cite{Diaz-Vargas-RH_for_prime_affine_line}, Poonen, and Sheats \cite{Sheats-Riemann_hypothesis}, who verify the conjecture for the affine line over \textit{any} finite field.
		
		Beyond the genus $0$ case, evidence is sparse:  B\"ockle \cite{Bockle} studies the $R_A(y)$ in the special case $A=\F_2[x,y]/(y^2 + y + x^3+x+1)$, which corresponds to a supersingular elliptic curve with class number one. He shows that this set agrees exactly with the affine line case, except that the smallest ``real part'' occurs with multiplicity $2$. Similar calculations were carried out in the thesis of Yujia Qui in the case where $A$ has class number one. Excluding the affine line, there are exactly eight such $A$ in all characteristics.

		\paragraph{Vanishing at special values} 
        We may also consider special values of $\zeta_{A,\pi}(s)$ at negative integers.
        The question of vanishing at odd integers was posed by Thakur in \cite{Thakur-characteristic_p_zeta-irregular_zeros}. In this article Thakur proved that the special values at odd negative integers are nonzero when $A$ has class number one. Beyond this, little is known. The story for
        even integers is more developed. David Goss discovered `trivial zeros' occurring at negative even integers. These trivial zeros arise from congruences between the zeta functions $\zeta_{A,\pi}(-,y)$ and certain $L$-series in characteristic $0$, which have an extra Euler factor at $\infty$. Thus, trivial zeros come from an extra factor at $\infty$, analogous to what happens with the Riemann zeta function. 
  
        Using a limiting argument with trivial zeros, Goss found that for a delicately constructed $y$, infinitely many elements of $R_A(y)$ occur with multiplicity $\geq d$ (the corresponding zeros of $\zeta_A$ are called \emph{near-trivial}). The existence of near-trivial zeros provides a counterexample to Goss's original conjecture on multiplicities occuring in $R_A(y)$. To complicate matters further, in \cite{Thakur-characteristic_p_zeta-irregular_zeros} Thakur observed that the order of vanishing at a trivial zero could exceed the order of vanishing coming from the ``missing'' Euler factor at $\infty$ (see \cite{Diaz-Vargas-irregular_zeros} for further examples of this phenomenon). Such trivial zeros are called \emph{irregular}. The existence of irregular trivial zeros opens the possibility of the near-trivial zeros having high orders of vanishing. In \cite{Goss-zeros_of_Lseries}, Goss states a conjecture that would prohibit this type of degenerate behavior.

	\subsection{Main results}
    The purpose of this article is to study the $R_A(y)$ under the generic assumption that $X$ is an ordinary curve. Under this assumption, we completely determine the $R_A(y)$ in terms of the genus $g$ of $X$, the degree of $\infty$, the cardinality $r$ of the residue field at $\infty$, and the $p$-adic digits of $y$. This result may be regarded as a corrected version of Goss's original conjecture. Our calculation immediately implies that the trivial zeros always have the `correct multiplicity'. That is, we verify Goss's conjecture on irregular trivial zeros under the generic assumption that $X$ is ordinary. We also deduce nonvanishing at negative odd integers. Finally, we prove partial analogues of these results for the $v$-adic interpolations of $\zeta_A$ at the finite places $v$ of $A$.
	
		\subsubsection{Distribution of the zeros}
		
		\begin{theorem}\label{t:Riemann-hypothesis}
			Let $r = q^{d}$. Fix $y \in \Z_p$. There exists a (possibly finite) increasing sequence of positive integers $\alpha_{r,1} < \alpha_{r,2} < \cdots$, depending only on $y$ and $r$, such that for \emph{any} ordinary curve $X$, we have
			\begin{equation}
				R_A(y)	=	\{	\underbrace{0,\dots,0}_{g-1+d}	\}	\sqcup	\bigsqcup_{i = 1}^\infty	\Bigl\{	\underbrace{\alpha_{r,i}(r-1),\dots,\alpha_{r,i}(r-1)}_{d \text{ times}}	\Bigr\}.
			\end{equation}
		\end{theorem}
		
		In the special case $A = \F_q[\theta]$, Theorem \ref{t:Riemann-hypothesis} is due to Wan, Diaz-Vargas, Poonen, and Sheats. In this previous work it is shown that the numbers $\alpha_{r,i}$ are related to the $r$-adic digit expansion of $y$. One can deduce from their results that each ``real part'' is divisible by $r-1$, although it is not explicitly stated. We note that our proof of Theorem \ref{t:Riemann-hypothesis} is completely independent of this prior work. In particular, we give a new proof of the Riemann hypothesis for the affine line.

        \begin{remark}
            When $X$ is non-ordinary, we expect that a similar result holds. In particular, we expect that: 
            \begin{equation*}
				R_A(y)	=	\{	\underbrace{s_1,\dots,s_{g-1+d}}	\}	\sqcup	\bigsqcup_{i = 1}^\infty	\Bigl\{	\underbrace{\alpha_{r,i}(r-1),\dots,\alpha_{r,i}(r-1)}_{d \text{ times}}	\Bigr\},
			\end{equation*}
            where we have $0 \leq s_i\leq \alpha_{r,1}$.
            In fact, one can prove that the number of $s_i$'s that are equal to zero is
            $p(X) + d -1$, where $p(X)$ denotes the $p$-rank of $X$. Our expectation is
            in agreement with the calculations in \cite{Bockle}. 
            
        \end{remark}
		\subsubsection{Integral zeta values}
		      For $j \in \Z$ there is a natural $j$-th power character of $K_\infty^+$:
            it sends $\alpha \in K_\infty^+$ to $\alpha^j \in \C_\infty$. We wish to extend this character of $K_\infty^+$ to an element of our character group $\bfS_\pi$. Choose $\pi_*$ to be a $d$-th root of $\pi$. We then define
            $s_j=(\pi_*^{-j},j) \in \bfS_\pi$. Note that for every $\alpha \in K_\infty^+$
            we have $\alpha^{s_j}=\alpha^j$, so the character $s_j$ is really an extension of the $j$-th power map on $K_\infty^+$. We say $s_j$ is even if $r-1|j$ and odd otherwise. We are concerned with
            the vanishing of the special values $\zeta_{A,\pi}(s_j)$.\footnote{For each $j$, there are really $\frac{d}{p^{v_p(d)}}$ different special values to consider, each corresponding to the different $d$-th roots of $\pi$. We will see that they all exhibit the same vanishing behavior.} 

            Consider the function defined on $\C_\infty^\times$ by
			$z_{\pi,s_j}(x)	=	\zeta_{A,\pi}(x \cdot \pi_*^j,j)$.
			This is known to be a polynomial in $x^{-1}$ whose coefficients live in a finite extension of $K_\infty$. When $j$ is even and negative, Goss proved that
			\begin{equation}\label{eq: factor at infinity}
				(1-x^{-d})| z_{\pi,s_j}(x).
			\end{equation}
			This implies that $s_j$ (as well as other possible choices for $j$ using the other $d$-th roots of $\pi$) are zeros of $\zeta_{A,\pi}$. The zeros of $\zeta_{A,\pi}$ that arise in this way are called \emph{trivial zeros}. They are analogous to the trivial zeros of the Riemann zeta function. The order of vanishing of $s_j$ is at least equal $p^{v_p(d)}$. We will say that $s_j$ is \emph{regular} if these orders of vanishing agree, and \emph{irregular} otherwise. 
			\begin{corollary}
				Suppose that $X$ is ordinary. If $s_j$ is even then $\zeta_{A,\pi}(s_j)$ vanishes with order $p^{v_p(d)}$ (i.e. $s_j$ is a regular trivial zero). If $s_j$ is odd then $\zeta_{A,\pi}(s_j) \neq 0$. 
			\end{corollary}
            \begin{proof}
                From Theorem \ref{t:Riemann-hypothesis} we know that each `real part' of the zeros occurs with multiplicity $d$. For $j$ even and negative 
                we see
                that these $d$ zeros with real part $j$ all come from \eqref{eq: factor at infinity}. Also, from Theorem \ref{t:Riemann-hypothesis} we know that the real parts are of the zeros are divisible by $r-1$. This implies $\zeta_{A,\pi}(s_j)\neq 0$ for $j$ odd. 
            \end{proof}

		\subsubsection{\texorpdfstring{$v$}{v}-adic interpolations}
		
			 Let $v$ be a closed point of $X-\infty$ of degree $d_v$. In \cite{Goss-vadic}, Goss constructs a $v$-adic interpolation $\zeta_{A,v}$ of the zeta function in analogy with the Kubota-Leopoldt $p$-adic zeta function. Let $K_v$ denote the completion of $K$ at $v$, and let $\C_v$ be the completion of an algebraic closure of $K_v$. The domain of the $v$-adic zeta function is a topological group with a canonical decomposition
			 \begin{equation}\label{eq:v-adic-Goss-plane}
			 	\bfS_{A,v}	=	\C_v^\times \times \Z/m\Z \times \Z_p,
			 \end{equation}
			 where $m$ is a positive integer. The elements of $\bfS_{A,v}$ act as a group of characters of the ideals of $A$ which are prime to $v$, see \cite{Goss-book}[\S 8.3].
			 
			Let $\bfS_{A,v}^0 = \C_v^\times \times \{1\} \times \Z_p$ denote the ``identity component'' of $\bfS_{A,v}$. Our main results concerning the $v$-adic zeta function pertain to its restriction to $\bfS_{A,v}^0$. Let $B$ denote the coordinate ring of the affine curve $X-v$. Let $\pi_v \in K_v$ be a uniformizer. The next result provides a formula for $\zeta_{A,v}$ in terms of the zeta function $\zeta_{B,\pi_v}$ of $B$:
			 
			 \begin{theorem}\label{t:v-adic-to-classical}
			 	Let $\frakp$ denote the prime ideal of $B$ corresponding to $\infty$. There exists a $1$-unit $\eta \in U_v$ with the following property: For any $s = (x,y) \in \bfS_{\pi_v}$ with $\pi_v$-adic coordinates $(x,y) \in \C_v^\times \times \Z_p$, we have
			 	\begin{equation}\label{eq:v-adic-to-classical}
			 		\zeta_{A,v}(x,1,y)	=	\zeta_{B,\pi_v}(\eta \cdot x,y) \cdot (1- \frakp^{-(\eta \cdot x,y)}).
			 	\end{equation}
			 \end{theorem}
			 
			\begin{remark}
				The Euler product of $\zeta_{A,v}$ is taken over all primes of $A$ distinct from $v$. Meanwhile, the Euler product of $\zeta_{B,\pi_v}$ is taken over all primes of $B$. The extra factor on the right side of \eqref{eq:v-adic-to-classical} corrects for this one-factor discrepancy. 
			\end{remark}
			
			By combining Theorem \ref{t:v-adic-to-classical} with Theorem \ref{t:Riemann-hypothesis}, we immediately deduce a description of the $v$-adic distribution of the zeros for the restriction of $\zeta_{A,v}$ to $\bfS_{A,v}^0$:
			
			\begin{theorem}\label{t:v-adic-Riemann-Hypothesis}
				Let $r_v = q^{d_v}$. Fix $y \in \Z_p$. The $v$-adic Newton slopes of the entire function $\zeta_{A,v}(-,1,y)$ on $\C_v^\times$ are exactly
				
				\begin{equation*}
					\{	\underbrace{0,\dots,0}_{g-1+d_v+d}	\}	\sqcup	\bigsqcup_{i = 1}^\infty	\Bigl\{	\underbrace{\frac{\alpha_{r_v,i}}{d_v},\dots,\frac{\alpha_{r_v,i}}{d_v}}_{d_v \text{ times}}	\Bigr\},
				\end{equation*}
            where the $\alpha_{r_v,i}$ are the same integers introduced in Theorem \ref{t:Riemann-hypothesis}.
			\end{theorem}
			
			There are also $v$-adic analogues of trivial zeros of $\zeta_A$. Instead of coming from a missing Euler factor at $\infty$, the trivial zeros of $\zeta_{A,v}(-,1,y)$ come from a missing Euler factor at $v$. When $j \in (r_v -1)\Z$, one can show that $\zeta_{A,v}(x,1,y)$ is divisible by $\pi_v^jx^{d_v}-1$. This factor contributes $d_v$ zeros to $\zeta_{A,v}(-,1,y)$, and each zero contributes a segment of slope $\frac{j}{d_v}$ to the Newton polygon. We may define the notions of irregular and regular trivial zeros as with $\zeta_{A,\pi}$. From Theorem \ref{t:v-adic-Riemann-Hypothesis}, we immediately deduce the following:
    
		    \begin{corollary}\label{Corollary: trivial zeros v-adic}
		    		Suppose that $X$ is ordinary. Then all trivial zeros of $\zeta_{A,v}$ are regular. If in addition we have $p \nmid d_v$, then all trivial zeros are simple.
		    \end{corollary}
	
	\subsection{Outline and summary}

        \subsubsection{Summary of the proof}

        The first step is to understand the zeta function as a family of $L$-functions
        of Galois characters.
        More precisely, for fixed $y \in \Z_p$ there exists a character
        \[ \rho_{A,\pi}^{\otimes y}:\pi_1^{\text{\'et}}(\Spec(A)) \to \C_\infty^\times \]
        such that $\zeta_{A,\pi}(-,y) =L(\rho_{A,\pi}^{\otimes y},-)$. 
        This vantage has been utilized when studying questions of meromorphicity of equicharacteristic zeta functions (see e.g. \cite{BocklePink-cohomology_of_crystals}, \cite{Bockle-global_lfunctions}, and \cite{Taguchi-Wan}).
        An essential observation is that the localization of $\rho_{A,\pi}^{\otimes y}$
        to $\Spec(K_\infty)$ only depends $y$ and the residue field at $\infty$ (see Corollary \ref{c: localization does not depend on curve}). In particular, we see that $\rho_{A,\pi}^{\otimes y}$ and $\rho_{\F_r[\theta],\pi}^{\otimes y}$ have the same localizations at $\infty$. This allows us to relate
        the zeta function of $A$ to the zeta function of the affine line. 
        
        The next step is to give a new proof of the Riemann hypothesis for the affine line 
        over $\F_q$. Using the Anderson-Monsky trace formula, for fixed $y \in \Z_p$ we can compute 
		$\zeta_{\F_r[\theta],\pi}(-,y)$ in terms of the ``Fredholm determinant'' $\det(1-x\Theta_{\F_r[\theta]}|M_{\F_r[\theta]})$
        of a certain compact operator $\Theta_{\F_r[\theta]}$ on a $\pi$-adic Banach space $M_{\F_r[\theta]}$. Both the operator $\Theta_{\F_r[\theta]}$ and the space $M_{\F_r[\theta]}$ are determined by $\rho_{\F_r[\theta],\pi}^{\otimes y}$. To study this Fredholm determinant, we choose an orthonormal basis of $M_{\F_r[\theta]}$ indexed by a countable set $I$ and write $\Theta_{\F_r[\theta]}$ as an $I\times I$ matrix $\Psi$. The Fredholm determinant has an explicit description:
        \begin{align*}
            \det(1-s\Psi) &= \sum_{n=0}^\infty c_ns^n, \\
            c_n &= (-1)^n \sum_{\stackrel{S \subset I}{|S|=n}} 
            \sum_{\sigma \in \mathrm{Sym}(S)} \text{sgn}(\sigma) \prod_{i \in S} \Psi_{i,\sigma(i)},
        \end{align*}
        where $\Psi_{i,j}$ denotes the $(i,j)$-th entry of $\Psi$. 
        The heart of the proof of Theorem \ref{t:Riemann-hypothesis} for $\F_q[\theta]$ is
        proving that there is a unique term in the sum defining $c_n$ whose $\pi$-adic
        valuation is minimal. 
        In the `toy case' (i.e. $\F_p[\theta]$), this is Corollary \ref{c: unique minimal permutation} whose proof is straightforward. The general case is Theorem \ref{theorem: unique minimal term} and the proof
        is a difficult combinatorial argument. The proof of Theorem \ref{theorem: unique minimal term} is the most technical part of this article . Once we establish there is a unique term of minimal valuation, we can pinpoint the $\pi$-adic valuation of $c_n$ and determine the Newton polygon of $\zeta_{\F_r[\theta],\pi}(-,y)$.

        Finally, we address the zeta function for general $A$. As in the affine line case,
        we compute $\zeta_{A,\pi}(-,y)$ by computing $\det(1-x\Theta_{A}|M_A)$ via the trace formula. In general it is difficult to make sense of both the operator $\Theta_A$ and the space $M_A$ when $X$ has positive genus. 
        To overcome this difficulty, we use a `local twisting' trick, which is similar to a technique introduced by the first author in \cite{Kramer-miller-exponential_sums} and \cite{kramer-miller-artin_lfunctions} to study exponential sums. The rough idea is that $M_A$ and $M_{\F_r[\theta]}$ are both included in an overspace $V$.
        Both operators $\Theta_{A}$ and $\Theta_{\F_r[\theta]}$ extend to all of $V$ in a natural way and are conjugate, i.e. there is an invertible operator $D$ with $D\Theta_AD^{-1}=\Theta_{\F_r[\theta]}$. The conjugacy of $\Theta_A$ and $\Theta_{\F_r[\theta]}$ is a consequence of the fact that $\rho_{A,\pi}^{\otimes y}$ and $\rho_{\F_r[\theta],\pi}^{\otimes y}$ have the same localizations at $\infty$.
        Then we have
        \[ \det(1-x\Theta_{A}|M_A) = \det(1-x\Theta_{\F_r[\theta]}|D M_A).\]
        Thus, to compute $\zeta_{A,\pi}(-,y)$ we are studying the same operator $\Theta_{\F_r[\theta]}$ that we used to study $\zeta_{\F_r[\theta],\pi}(-,y)$. The difference is that this operator is acting on different subspaces of $V$.
        The matrix of $\Theta_{\F_r[\theta]}$ acting on $D M_A$ looks very similar to the matrix of $\Theta_{\F_r[\theta]}$ acting on $M_{\F_r[\theta]}$. We can then exploit our understanding of $\Theta_{\F_r[\theta]}$ acting on $M_{\F_r[\theta]}$ from studying $\zeta_{\F_r[\theta],\pi}(-,y)$.

    \subsubsection{Outline of article}
	     In \S\ref{s:zeta} we introduce the Goss zeta function $\zeta_{A,\pi}$ and establish some basic properties. Our construction
      is a modification of the standard one (see e.g. \cite{Goss-L-series_edition}), but has the advantage of being coordinate free. In particular, it does not depend on 
      the choice of parameter $\pi$. In \S \ref{ss:zeta} we see that for any uniformizer $\pi \in K_\infty$, the function $\zeta_{A,\pi}(- ,y)$ is the $L$-function of a character
		\begin{equation*}
			\rho_{A,\pi}^{\otimes y}:\pi_1^{\text{\'et}}(\Spec(A)) \to \C_\infty^\times.
		\end{equation*}
		Let $\rho_{\infty,\pi}^{\otimes y}$ denote the localization of $\rho_{A,\pi}^{\otimes y}$ along $\Spec(K_\infty) \to \Spec(A)$. Our key observation is that $\rho_{\infty,\pi}^{\otimes y}$ depends only on $\pi$ and $r$--not on $A$. More precisely: let $\theta \in K_\infty$ be any element of valuation $-1$, so that $K_\infty = \F_r (\!(\theta^{-1})\!)$. Associated to the zeta function $\zeta_{\F_r[\theta],\pi}$ of the affine line we have a character
		\begin{equation*}
			\rho_{\F_r[\theta],\pi}^{\otimes y}:\pi_1^{\text{\'et}}(\mathbb{A}_{\F_r}^1) \to \C_\infty^\times.
		\end{equation*}
		Then the localization of this character at $\infty$ agrees with that of $\rho_{A,\pi}^{\otimes y}$. Finally, in \S \ref{ss:v-adic} we construct the $v$-adic zeta function on $\bfS_{A,v}^0$ and prove Theorem \ref{t:v-adic-Riemann-Hypothesis}
		
		In \S \ref{s: tau crystals and the trace formula} we give a summary of the theory of $F$-modules and $\tau$-modules.  We explain the Katz equivalence between \'etale $\tau$-modules and representations of the \'etale fundamental group. Finally, we review the Anderson-Monsky trace formula for curves. This formula allows us to express the $L$-function as the characteristic series of an operator on an infinite dimensional space.

        The remainder of the paper is dedicated to proving Theorem \ref{t:Riemann-hypothesis}. In \S \ref{s: toy example} we first consider the `toy case' where $X$ is defined
        over $\F_p$ and $\infty$ is a rational point. This is done solely for expository purposes and the later sections have no logical dependence on \S \ref{s: toy example}. The proof of the `toy case' uses the same strategy and insights, while avoiding the technical combinatorial arguments present in the general case. In
        \S \ref{s: affine line 1} we prove Theorem \ref{t:Riemann-hypothesis} 
        for the affine line over $\F_q$ using the Anderson-Monsky trace formula. This
        section is the most technical part of the paper, although our proof is significantly shorter than Sheats' proof of the Riemann hypothesis for $\F_q[\theta]$. Finally, in \S \ref{s: general ordinary curves} the proof
        of Theorem \ref{t:Riemann-hypothesis} is proven in full generality. 
        \subsection{Acknowledgements} 
        The authors would like to thank Daqing Wan who suggested that the $p$-adic perturbation methods
        developed in \cite{kramermiller2021newton} may be applicable to the Goss zeta function. Although these earlier methods are not used in this paper, the spirit of the approach taken in this paper is certainly related! The first author would also like to thank Dinesh Thakur for his encouragement in this project and his enthusiasm towards the new proof of the Riemann hypothesis for $\F_q[\theta]$.

    \section{The Goss zeta function}\label{s:zeta}

	\subsection{Character groups}\label{ss:character-groups}
	
		\subsubsection{Positive numbers}
		
			Let $P$ be a closed point of $X$. Let $K_P$ denote the completion of the function field $K(X)$ at $P$. We will fix $\C_P$ to be the completion of an algebraic closure of $K_P$. Unless otherwise stated, all extensions of $K_P$ will be assumed to be contained in $\C_P$. Let $L$ be a complete extension of $K_P$. Recall that the unit group $\calO_L^\times$ admits a canonical decomposition $\calO_L^\times	=	\mu_L	\times	U_L$, where $\mu_L$ denotes the group of roots of unity in $\calO_L$ and $U_L$ denotes the group of \emph{$1$-units}.
			
			\begin{definition}
				The group of \emph{positive numbers} of $L$ is the quotient $L^+ = L^\times/\mu_L$.
			\end{definition}
			
			Since $U_L$ intersects trivially with $\mu_L$, we will usually identify $U_L$ with its image in $L^+$. In the case $L = K_P$, we will generally use the subscript ``$P$'' in place of ``$K_P$.'' When $L/K_P$ is finite, we set $v_L = e_{L/K_P} v_P$ to be the normalized valuation of $L$. In this case, there is a commutative diagram with exact rows:
			\begin{equation}\label{eq:positive-extension}
				\begin{tikzcd}
					0	\arrow[r]	&	U_L	\arrow[r]	&	L^+	\arrow[r,"v_L"]	&	 \Z	\arrow[r]	&	0	\\
					0	\arrow[r]	&	U_P	\arrow[r] \arrow[u]	&	K_P^+	\arrow[r,"v_P"] \arrow[u]	&	\Z	\arrow[r] \arrow[u,"e_{L/K_P}"']	&	0
				\end{tikzcd}.
			\end{equation}
			
			The group $\C_P^+$ has the special property that it is \emph{uniquely divisible}: If $d$ is a positive integer, then any $\alpha \in \C_P^+$ has a unique $d$-th root $\alpha^{1/d}$ in $\C_P^+$.
			
			\begin{definition}
				Let $\pi \in K_P$ be a uniformizer. We define the $\pi$-adic \emph{$1$-unit character} to be the character $\langle \cdot \rangle_\pi: K_P^+ \to U_P$ defined by the formula
				\begin{equation*}
					\langle \alpha \rangle_\pi = \frac{\alpha}{\pi^{v_P(\alpha)}} \in U_P.
				\end{equation*}
			\end{definition}
			
			\begin{lemma}\label{l:positive-uniquely-divisible}
				Let $n$ be a positive integer. Let $L/K_P$ be a finite extension. Let $\alpha \in K_P^+$ such that $v_L(\alpha) \in n\Z$. Then there exists a finite purely inseparable extension $L_1/L$ such that $\alpha^{1/n} \in L_1^+$.
			\end{lemma}
			\begin{proof}
				The condition on $v_L(\alpha)$ ensures that $\alpha/\langle \alpha \rangle_\pi$ has an $n$th root in $L$. It suffices then to consider $\alpha \in U_P$. Factor $n = n^\tame \cdot p^{v_p(n)}$. Since  $1/n^\tame \in \Z_p$ and $U_P$ is naturally a $\Z_p$-module, we have $\alpha^{1/n^\tame} \in U_P$. The claim follows from:
				\begin{equation*}
					\alpha^{1/n}	=	(\alpha^{1/n^\tame})^{1/p^{v_p(n)}} \in K_P^{1/p^{v_p(n)}}.
				\end{equation*}
			\end{proof}
		
		\subsubsection{Characters}\label{ss:characters}
		
			Let $L$ be a finite extension of $K_P$. We will now discuss the character theory of the group $L^+$. Since the group $\C_P^\times$ is divisible, we see by dualizing the top row of \eqref{eq:positive-extension} that there is a short exact sequence
			\begin{equation}\label{eq:hom-exact-sequence}
				0	\to	\Hom(\Z,\C_P^\times)	\to	\Hom(L^+,\C_P^\times)	\to	\Hom(U_P,\C_P^\times) \to	0.
			\end{equation}
			
			\begin{definition}
				Let $x \in \C_P^\times$. We define a character $s_x \in \Hom(\Z,\C_p^\times)$ via
				\begin{equation*}
					s_x(n)	=	x^{-n}.
				\end{equation*}				
			\end{definition}
			
			The map $x \mapsto s_x$ gives an isomorphism of $\C_P^\times$ with $\Hom(\Z,\C_P^\times)$. As discussed in the introduction, the character group $\Hom(U_P,\C_P^\times)$ is very large. For this reason we restrict our attention to characters of a special type:
			
			\begin{definition}
				The \emph{Goss plane} over $L$ is the group $\bfS_L$ of characters $s \in \Hom(L^+,\C_\infty^\times)$ with the following property: There exists a $p$-adic integer $y(s) \in \Z_p$ such that the restriction of $s$ to $U_L$ is the $y(s)$-th power map. The group operation on $\bfS_L$ will be written additively. For $s \in \bfS_L$, we will write $\alpha^s$ for the value of $s$ at a positive number $\alpha \in L^+$. 
			\end{definition}
			
			The next lemma introduces the ``real part'' map on $\bfS_L$ as described in the introduction.
			
			\begin{lemma}\label{l:real-part-map}
				There exists a unique homomorphism $r:\bfS_L \to \Q$ with the following property: For any $\alpha \in L^+$, we have
				\begin{equation*}
					v_L(\alpha^s)	=	r(s)\cdot v_L(\alpha).
				\end{equation*}
			\end{lemma}
			\begin{proof}
				Every character $s \in \bfS_L$ maps $U_L$ into the unit circle $\calO_{\C_P}^\times$ in $\C_P$. The valuation $v_L$ induces isomorphisms $L^+/U_L \cong \Z$ and $\C_P^+/\calO_{\C_P}^\times \cong \Q$. The desired map is:
				\begin{equation*}
					\bfS_L	\to	\Hom(L^+/U_L ,\C_P^+/\calO_{\C_P}^\times) \cong \Hom(\Z,\Q)	\cong \Q.
				\end{equation*}				
			\end{proof}
			
			Let us write $\bfS_P = \bfS_{K_P}$ for the Goss plane over $K_P$. There is a natural restriction map $\bfS_L \to \bfS_P$. Dualizing \eqref{eq:positive-extension}, we obtain a commutative diagram with exact rows:
			\begin{equation}\label{eq:character-group-extension}
				\begin{tikzcd}
					0	\arrow[r]	&	\C_P^\times	\arrow[r]	\arrow[d,"e_{L/K_P}"']	&	\bfS_L	\arrow[r]	\arrow[d]	&	\Z_p		\arrow[r]	\arrow[d,equals]		&	0	\\
					0	\arrow[r]	&	\C_P^\times	\arrow[r]	&	\bfS_P	\arrow[r]	&	\Z_p	\arrow[r]	&	0
				\end{tikzcd}.
			\end{equation}
			In particular, the restriction map is an isomorphism whenever the \emph{tame ramification index} $e_{L/K_P}^\tame$ is $1$. The next result states that the group extensions $\bfS_L$ are completely classified by the tame ramification indices $e_{L/K_P}^\tame$.
			
			\begin{proposition}\label{p:character-groups-isomorphic}
				Let $L,L'$ be two finite extensions of $K_P$ with the same tame ramification index over $K_P$. If $L \cdot L'$ denotes the compositum of $L$ and $L'$, then the restriction maps induce canonical isomorphisms
				\begin{equation*}
					\bfS_L	\xleftarrow{\sim}	\bfS_{L \cdot L'}	\xrightarrow{\sim}	\bfS_{L'}.
				\end{equation*}
			\end{proposition}
			\begin{proof}
				Abhyankar's lemma ensures that $L \cdot L'$ has the same tame ramificaiton index over $K_P$ as $L$ and $L'$.
			\end{proof}
		
		\subsubsection{Coordinates}
		
			Let $L$ be a finite extension of $K_P$. We will now explain how one can obtain a ``coordinate picture'' of the Goss plane $\bfS_L$ by fixing a choice of uniformizer $\pi \in K_P$. Recall that $\pi$ determines a $1$-unit character $\langle \cdot \rangle_\pi \in \bfS_P$.
			
			\begin{proposition}\label{p:1-unit-ext}
				 The $1$-unit character extends uniquely to a character $\langle \cdot \rangle_\pi$ in $\bfS_L$ with values in $U_{\C_P}$.
			\end{proposition}
			\begin{proof}
				Note that any extension of $\langle \cdot \rangle_\pi$ in $\bfS_L$ must act as the identity on $U_L$. Let $e = e_{L/K_P}$. For any $\alpha \in L^+$, we have $v_P(\alpha^e) \in \Z$. Since $\langle \pi \rangle_\pi = 1$, we must have
				\begin{equation*}
					\langle \alpha^e \rangle_\pi	=	\frac{\alpha^e}{\pi^{v_P(\alpha^e)}} \in U_L.
				\end{equation*}
				Then $\langle \alpha \rangle_\pi = \langle \alpha^e \rangle_\pi^{1/e} \in U_{\C_P}$ provides the unique extension of $\langle \cdot \rangle_\pi$.
			\end{proof}
			
			In light of Proposition \ref{p:1-unit-ext}, the choice of $\pi$ provides a section of the natural map $\bfS_L \to \Z_p$ given by $y \mapsto \langle \cdot \rangle_\pi^y$. Thus, the choice of $\pi$ determines a splitting
			\begin{equation*}
				\bfS_L	\cong	\C_P^\times \times \Z_p.
			\end{equation*}
			
			\begin{notation}
				For $s \in \bfS_L$, the corresponding pair $(x,y) \in \C_P^\times \times \Z_p$ will be called the \emph{$\pi$-adic coordinates} of $s$. The action of $s$ in terms of its coordinates is given by:
				\begin{equation*}
					\alpha^s	=	x^{-v_L(\alpha)} \langle \alpha \rangle_\pi^y.
				\end{equation*}
			\end{notation}
			
			\begin{remark}
				Although $\bfS_L$ depends up to canonical isomorphism only on $e_{L/K_P}^\tame$, the diagram \eqref{eq:character-group-extension} and the $\pi$-adic coordinates on $\bfS_L$ depend on $e_{L/K_P}$.
			\end{remark}

	\subsection{The Goss zeta function}\label{ss:zeta}
	
		\subsubsection{Divisors} \label{sss:divisors}
		
			As a first step in defining the Goss zeta function, we will explain how for certain $L$, we can regard $\bfS_L$ as a character groups on classes of divisors on $X$. Our starting point is the canonical divisor exact sequence:
			\begin{equation*}
				0	\to	K^\times/\F_q^\times	\to	\Div^0(X)	\to	\Pic^0(X)	\to	0.
			\end{equation*}
			Since $\Pic^0(X)$ is finite, we have an induced isomorphism between the \emph{rational} divisor groups $\Q \otimes K^\times/\F_q^\times \cong \Q \otimes \Div^0(X)$. Because $\C_P^+$ is uniquely divisible, the natural embedding $K^\times/\F_q^\times \hookrightarrow K_P^+$ extends uniquely to a map
			\begin{equation}\label{eq:divisor-to-positive-number}
				\Q \otimes \Div^0(X)	\to	\C_P^+.
			\end{equation}
			The next observation follows immediately from Lemma \ref{l:positive-uniquely-divisible} and the finiteness of $\Pic^0(X)$.
			
			\begin{lemma}\label{l:positive-map-on-divisors}
				Let $L$ be a finite extension of $K_P$. There is a finite purely inseparable extension $L_1/L$ such that \eqref{eq:divisor-to-positive-number} restricts to a map
				\begin{equation*}
					\psi: \frac{1}{e_{L/K_P}} \Div^0(X)	\to	L_1^+.
				\end{equation*}
			\end{lemma}
			
			Suppose now that $D \in \Div(X-P)$ is prime to $P$, but not necessarily of degree $0$. We can associate to $D$ a rational divisor of degree $0$ as follows:
			\begin{equation*}
				D_P	:=	D - \frac{\deg(D)}{\deg(P)} P\in \frac{1}{\deg(P)} \Div^0(X).
			\end{equation*}
            In particular, if $L/K_P$ is sufficiently ramified, by Lemma \ref{l:positive-map-on-divisors} we obtain a map $\Div^0(X-P) \to L_1^+$ sending $D \mapsto \psi(D_P)$, where $L_1$ is a finite purely inseparable extension of $L$. 
            Then from Proposition \ref{p:character-groups-isomorphic} we can regard $\bfS_L$
            as a group of characters on $\Div^0(X-P)$.
	
		\subsubsection{Definitions}
		
			Recall that $A$ denotes the coordinate ring of the affine curve $X - \infty$. We now introduce our main character group of interest:
			
			\begin{definition}
				We define the group $\bfS_A = \bfS_L$, where $L$ is any finite extension of $K_P$ with ramification index $d$.
			\end{definition}
			
			Note that $\bfS_A$ depends up to canonical isomorphism only on $d^\tame=\frac{d}{p^{v_p(d)}}$, although any choice of coordinates will depend on $d$. As explained above, the characters in $\bfS_A$ act on the non-zero fractional ideals of $A$ (regarded as divisors in $\Div(X-\infty)$). Explicitly: If $s \in \bfS_A$ has $\pi$-adic coordinates $(x,y)$, then we recover Goss' \emph{exponentiation of ideals} \cite{Goss-book}[\S 8.2]:
			\begin{equation*}
				\fraka^s	=	x^{\deg(\fraka)} \cdot \langle \fraka_\infty \rangle_\pi^y.
			\end{equation*}
			
			\begin{definition}
				The Goss zeta function is the function on $\bfS_A$ defined by
				\begin{equation*}
					\zeta_A(s)	=	\sum_\fraka \fraka^{-s} = \prod_\frakp	\frac{1}{1-\frakp^{-s}},
				\end{equation*}
				where the sum is taken over all non-zero ideals $\fraka \subseteq A$, and the product is taken over non-zero prime ideals $\frakp \subset A$.
			\end{definition}
			
		\subsubsection{\texorpdfstring{$1$}{1}-unit characters and \texorpdfstring{$L$}{L}-functions} \label{sss: 1-unit characters}
			
			Let $\pi \in K_\infty$ be a uniformizer. Under the coordinate isomorphism $\bfS_A \cong \C_\infty^\times \times \Z_p$, we may regard $\zeta_A$ as a function of two variables. Explicitly, we will write
			\begin{equation*}
				\zeta_{A,\pi}(x,y)	= \zeta_A(s)	=	\prod_\frakp	\frac{1}{1-x^{\deg(\frakp)}\langle \frakp_\infty \rangle_\pi^y}.
			\end{equation*}
			
			By class field theory for function fields, the ``$1$-unit character'' sending a non-zero fractional ideal $\fraka$ to $\langle \fraka_\infty \rangle_\pi$ corresponds to a character of the fundamental group
			\begin{equation*}
				\rho_{A,\pi}:\pi_1(X - \infty) \to U_{\C_\infty}.
			\end{equation*}
			By the finiteness of the class group of $A$, the extension $\bfV_\infty$ of $\calO_\infty$ generated by the $\langle \frakp_\infty \rangle_\pi$ is finite \footnote{Although the $1$-unit character $\langle \cdot \rangle_\pi$ depends on the uniformizer $\pi$, it is not hard to see that the extension $\bfV_\infty$ of $\mathcal{O}_\infty$  is a totally inseparable extension that does not depend on the choice of uniformizer.}. Since $\rho_{A,\pi}$ takes values in the $1$-units of $\bfV_\infty$, for each $p$-adic integer $y$ we may form the character $\rho_{A,\pi}^{\otimes y}$. Associated to each character we have an $L$-series
			\begin{equation*}
				L(\rho_{A,\pi}^{\otimes y},T)	=	\prod_\frakp \frac{1}{1-T^{-\deg(\frakp)} \rho_{A,\pi}^{\otimes y}(\Frob_\frakp)} \in 1+T^{-1} \bfV_\infty \llbracket T^{-1} \rrbracket.
			\end{equation*}
			We immediately see from the definitions:
			
			\begin{proposition} \label{p: zeta function via L-functions}
				For $(x,y) \in \C_\infty^\times \times \Z_p$, we have
				\begin{equation*}
					\zeta_{A,\pi}(x,y)	=	L(\rho_\pi^{\otimes y},x^{-1}).
				\end{equation*}
			\end{proposition}
			
			Thus, the question concerning the values $r(s)$ when $s$ is a root of $\zeta_A$ reduces to the question of computing the $\pi$-adic \emph{Newton polygon} of $L(\rho_{A,\pi}^{\otimes y},T)$. This will be the approach taken in the remainder of the paper.
			
		\subsubsection{Localization}
		
			Let us write $\eta_{A,\infty}:\Spec(K_\infty) \to \Spec(A)$ for the localization map. The following observation is immediate from the compatibility of local and global class field theory for function fields.
			
			\begin{proposition}
				Under the local class field theory correspondence, the character $\eta_{A,\infty}^* \rho_{A,\pi}$ corresponds to the $1$-unit character $\langle \cdot \rangle_\pi$ of $K_\infty^\times$.
			\end{proposition}
			This leads to one of the key observations in this paper, which allows us to relate the zeta function of higher-genus curves to the zeta function of the affine line.
			
			\begin{corollary}\label{c: localization does not depend on curve}
				For $i = 1,2$ let $X_i$ be a smooth, projective, geometrically connected curve over a finite extension $k_i$ of $\F_p$. Let $\infty_i$ be a closed point of $X_i$, and let $A_i$ denote the ring of regular functions on $X_i - \infty_i$. Suppose that each $\infty_i$ has the same residue field $\F_r$. Choose isomorphisms $\F_r(\!(\pi)\!) \xrightarrow{\sim} K_{\infty_i}$. Then we have an induced isomorphism
				\begin{equation*}
					\eta_{A_1,\infty_1}^* \rho_{A_1,\pi}	\cong	\eta_{A_2,\infty_2}^* \rho_{A_2,\pi}.
				\end{equation*}
			\end{corollary}
		
	\subsection{\texorpdfstring{$v$}{v}-adic theory}\label{ss:v-adic}
	
		\subsubsection{The \texorpdfstring{$v$}{v}-adic zeta function}
	
			Let $v$ be a place of $K$ different from $\infty$. In light of the natural decomposition $\calO_v^\times = \mu_v \times U_v$, every $\alpha \in \calO_v^\times$ admits a unique factorization
			\begin{equation*}
				\alpha	=	\omega_v(\alpha) \cdot \langle \alpha \rangle_v,
			\end{equation*}
			where $\omega_v(\alpha)$ is a root of unity and $\langle \alpha \rangle_v \in U_v$. If $D = (\alpha)$ is a non-zero principal divisor of $X$ which is prime to $v$, we see that the image of $D$ in $K_v^+$ is $\langle \alpha \rangle_v$. By Lemma \ref{l:positive-map-on-divisors}, we see that $\langle \cdot \rangle_v$ extends uniquely to a map
			\begin{equation*}
				\langle \cdot \rangle_v: \frac{1}{d} \Div^0(X-v)	\to	U_{\C_v}.
			\end{equation*}
			
			Let $\fraka$ be a non-zero fractional ideal which is prime to $v$. Then the rational divisor $\fraka_\infty$ as defined in \S \ref{sss:divisors} lies in $\Div^0(X-v)$. By class field theory, the map $\fraka \mapsto \langle \fraka_\infty \rangle_v$ extends uniquely to a character of the fundamental group
			\begin{equation*}
				\rho_{A,v}:\pi_1(X - \{\infty,v\}) \to U_{\C_v}.
			\end{equation*}
			By the finiteness of the class group of $A$, the extension $\bfV_v$ of $\calO_v$ generated by the $\langle \fraka_\infty \rangle_v$ is finite. Since $\rho_{A,v}$ takes values in the $1$-units of $\bfV_v$, for each $p$-adic integer $y$ we may form the character $\rho_{A,v}^{\otimes y}$. Associated to each character we have an $L$-series
			\begin{equation*}
				L(\rho_{A,v}^{\otimes y},T)	=	\prod_\frakp \frac{1}{1-T^{-\deg(\frakp)} \rho_{A,v}^{\otimes y}(\Frob_\frakp)} \in 1+T^{-1} \bfV_v \llbracket T^{-1} \rrbracket,
			\end{equation*}
			where the product is over all non-zero prime ideals $\frakp \subset A$ coprime to $v$.
			
			\begin{definition}
				Let $\bfS_{A,v}^0 = \C_v^\times \times \Z_p$. We define the \emph{$v$-adic zeta function} of $A$ to be the function on $\bfS_{A,v}^0$ defined by the relation
				\begin{equation*}
					\zeta_{A,v}(x,y)	=	L(\rho_{A,v}^{\otimes y},x^{-1})	=	\prod_\frakp \frac{1}{1-x^{\deg(\frakp)} \langle \frakp_\infty \rangle_v^y}.
				\end{equation*}							
			\end{definition}

\begin{remark}
        The standard $v$-adic zeta function (as defined in \cite{Goss-vadic} and \cite{Goss-book}*{\S 8}) has $\bfS_{A,v}=\C_v \times \mu_{q^f-1} \times \Z_p$ as a domain,
         where $f$ is a certain positive integer and $\mu_{q^f-1}$ is the group of $q^f-1$-th roots of unity. Our $v$-adic zeta function agrees with the $v$-adic zeta function from \cite{Goss-book}*{\S 8}
         restricted to $\C_v \times \{1\} \times \Z_p$, which we think of as the `connected component of the identity' of $\bfS_{A,v}$.
    \end{remark}

			
		\subsubsection{Comparing zeta functions}
			Let $B$ denote the coordinate ring of the affine curve $X - v$ and let $\pi_v$ be a uniformizer of $K_v$. To conclude this section, we will prove a comparison between $\zeta_{A,v}$ and $\zeta_{B,\pi_v}$ which reduces the proof of Theorem \ref{t:v-adic-to-classical} to Theorem \ref{t:Riemann-hypothesis}. 
			
			\begin{theorem}\label{t:v-adic-character-to-classical}
				Consider the rational divisor
				\begin{equation*}
					R	=	\frac{1}{d_v}[v] - \frac{1}{d}[\infty] \in \Q \otimes \Div^0(X).
				\end{equation*}
				For any non-zero fractional ideal $\fraka$ of $A$ coprime to $v$, we have the relation
				\begin{equation*}
					\langle \fraka_\infty \rangle_v	=	\langle \fraka_v \rangle_{\pi_v} \cdot \langle R \rangle_{\pi_v}^{\deg(\fraka)}.
				\end{equation*}
			\end{theorem}
			\begin{proof}
				The characters $\langle \cdot \rangle_v$ and $\langle \cdot \rangle_{\pi_v}$ agree on $\Div^0(X - v)$. Thus we have
				\begin{align*}
					\langle \fraka_\infty \rangle_v	&=	\langle \fraka_\infty \rangle_{\pi_v}	\\
						&=	\left\langle \fraka - \tfrac{\deg(\fraka)}{d} [\infty] \right\rangle_{\pi_v}	\\
						&=	\left\langle \fraka - \tfrac{\deg(\fraka)}{d_v} [v] + \deg(\fraka) R\right\rangle_{\pi_v}	\\
						&=	\left\langle \fraka - \tfrac{\deg(\fraka)}{d_v} [v]\right\rangle_{\pi_v}  \left\langle R\right\rangle_{\pi_v}^{\deg(\fraka)} = \langle \fraka_v \rangle_{\pi_v} \cdot \langle R \rangle_{\pi_v}^{\deg(\fraka)}.
				\end{align*}
			\end{proof}
			
			\begin{corollary}
				Let $\frakp$ denote the prime ideal of $B$ corresponding to $\infty$. Let $R$ be as in Theorem \ref{t:v-adic-to-classical}. Then as functions on $\bfS_{A,v}^0 = \C_v^\times \times \Z_p$, we have
				\begin{equation*}
					\zeta_{A,v}(x,y)	=	\zeta_{B,\pi_v}(\langle R \rangle_{\pi_v} x,y) (1-(\langle R \rangle_{\pi_v} x)^{d} \langle \frakp \rangle_{\pi_v}^{-y}).
				\end{equation*}
			\end{corollary}
            \begin{proof}
                This follows from Theorem \ref{t:v-adic-character-to-classical} by comparing the Euler products of $\zeta_{A,v}(x,y)$ and $\zeta_{B,\pi_v}$.
            \end{proof}
	
	\section{\texorpdfstring{$F$-modules, $\tau$}{tau}-modules, and the trace formula}\label{s: tau crystals and the trace formula}

    As in the previous section we fix a closed point $\infty$ of $X$. We take $A$ to be the ring of functions on $X - \infty$ and $K=K_\infty$ to be the completion of $A$ at $\infty$. Let $\bfK$ be a copy of $K$. We think of $K$ as a space of formal functions and we think of $\bfK$ as our `coefficients'. Define $\bfR=\bfR_\infty$ to be the ring of integers
    of $\bfK$ and $\F_r$ to be the residue field of $\bfR$. Let $d$ be
    the degree of $\infty$ and let $b$ be the degree of $\F_r$ over $\F_p$ (i.e. $r=p^b$).
    Define $\bfV=\bfV_\infty$ as in \S \ref{sss: 1-unit characters}.
    Fix a choice of uniformizer $\pi \in \bfR$, so that
    $\bfR$ is isomorphic to $\F_r\llbracket \pi \rrbracket$.
    We know that
    $\bfV=\bfR[\pi^{1/p^h}]=\F_r\llbracket \pi^{1/p^h} \rrbracket$ for some fixed $h\geq 1$. Finally,
    we define $\bfR^\circ=\F_p\llbracket \pi \rrbracket$ and
    $\bfV^\circ = \F_p \llbracket \pi^{1/p^h} \rrbracket$. 

    \subsection{Functions with growth conditions}

        Fix an element $\theta \in K$ of valuation $-1$, so that we have an identification $K = \F_r (\! ( \theta^{-1} )\! )$. We define
        \begin{align*}
            K_{\F_r} &:= \F_r \otimes_{\F_q} K  = \F_r \otimes_{\F_q} \F_r (\! ( \theta^{-1} )\! ),\\
            A_{\F_r} &:= \F_r \otimes_{\F_q} A, \\
            \bfV_{\F_r} &:= \F_r \otimes_{\F_q} \bfV.
        \end{align*}

            The projection map on $K_{\F_r}$ is defined to be
            \begin{align*}
                \pr: K_{\F_r} &\to \theta \F_r \otimes_{\F_q}\F_r[\theta],  \\
                \sum_i c_i\otimes d_i \theta^i &\mapsto \sum_{i \geq 1} c_i\otimes d_i \theta^i.
            \end{align*}
        Consider the rings
        \begin{align*}
            K_{\F_r} \widehat{\otimes}_{\F_q} \bfV &\cong K_{\F_r} \widehat{\otimes}_{\F_p} \bfV^\circ,\\
            A_{\F_r} \widehat{\otimes}_{\F_q} \bfV &\cong A_{\F_r} \widehat{\otimes}_{\F_p} \bfV^\circ,
        \end{align*}
        where the completion is taken using the $\pi$-adic topology.
        More explicitly, we can describe $K_{\F_r} \widehat{\otimes}_{\F_q} \bfV^\circ$ as:

        \begin{align*}
            K_{\F_r} \widehat{\otimes}_{\F_p} \bfV^\circ  = \Bigg\{ \sum_{-\infty}^\infty c_i \theta^i ~:~ c_i \in \bfV_{\F_r} \text{ and } \lim_{i \to \infty} c_i=0  \Bigg \}.
        \end{align*}
        We extend the projection map $\pr$ to $K_{\F_r} \widehat{\otimes}_{\F_p} \bfV^\circ$ by $\bfV$-linearity:
        \begin{align*}
            \pr: K_{\F_r} \widehat{\otimes}_{\F_p} \bfV^\circ &\to \theta \bfV_{\F_r}\langle \theta \rangle, \\
                \sum_i c_i\theta^i &\mapsto \sum_{i \geq  1} c_i \theta^i,
        \end{align*}
        where $\bfV_{\F_r}\langle \theta \rangle$ denotes the
        one dimensional Tate algebra over $\bfV_{\F_r}$ using the $\pi$-adic valuation. 

        \begin{definition}
            Let $m > 0$ be a rational number. We define the following $\bfV_{\F_r}$-modules of power series with linear growth conditions:
            \begin{equation*}
                L^m := \left\{ \sum_{i = 1}^\infty r_i \theta^i \in \theta \bfV_{\F_r}\langle \theta \rangle~\Big |~  \begin{array}{c} v(r_i) \geq i/m \text{ for all }i\text{, and}\\ \displaystyle\lim_{i \to \infty} v(r_i) - i/m=\infty\end{array} \right\}.
            \end{equation*}
            In addition, we define a module of \emph{overconvergent} power series:
            \begin{equation*}
                L^\dagger := \bigcup_{m > 0} L^m.
            \end{equation*}
        \end{definition}

        Using the projection map and the spaces $L^m$ we introduce some additional
        subrings of $K_{\F_r} \widehat{\otimes}_{\F_r} \bfV$ and
        $A_{\F_r} \widehat{\otimes}_{\F_r} \bfV$ with growth conditions.
        For $*=m$ or $\dagger$ we set

        \begin{align*}
            K_{\F_r} \otimes^*_{\F_r} \bfV &:= \{ f \in K_{\F_r} \widehat{\otimes}_{\F_r} \bfV ~:~ \pr(f) \in L^* \},\\
            A_{\F_r} \otimes^*_{\F_r} \bfV &:= \{ f \in A_{\F_r} \widehat{\otimes}_{\F_r} \bfV ~:~ \pr(f) \in L^* \}.
        \end{align*}

	\subsection{\texorpdfstring{$F$}{F}-modules and \texorpdfstring{$\tau$}{tau}-modules}
        \subsubsection{Basic definitions}\label{sss: F and tau basic defs}
        \begin{notation}
            Let $B$ be an algebra over $\F_r$ and let $\mathcal{B}$ be $B \widehat{\otimes}_{\F_p} \bfV^\circ$. We
            denote by $F$ the $p$-Frobenius endomorphism of
            $\mathcal{B}$ relative to $\bfV^\circ$ (i.e. $F(a \otimes r)=a^p \otimes r$). We set $\tau$ to be $F^{b}$, the $b$-fold iteration of $F$. We remark that $\tau$ is $\bfV$-linear.
            We denote by $W$ the $p$-Frobenius endomorphism of $\calB$ relative to $B$ (i.e. $W(a \otimes r)=a \otimes r^p$).
        \end{notation}
	    \begin{definition}
	        An \emph{$F$-module} over $\mathcal{B}$ is a projective $\mathcal{B}$-module $\mathscr{F}$ together with a $\bfV^\circ$-linear map $\phi_\mathscr{F}:\mathscr{F} \to \mathscr{F}$ that is $F$-semilinear in the sense that
	        \begin{equation*}
	            \phi_\mathscr{F}(am) = F(a) \phi_\mathscr{F}(m)
	        \end{equation*}
	        for all $a \in \mathcal{A}$ and $m \in \mathscr{F}$. The category of $F$-modules over $\mathcal{B}$ is denoted by $\Fmod(\mathcal{B})$. A \emph{$\tau$-module} over
            $\mathcal{B}$ a finite projective $\mathcal{B}$-module $\mathscr{G}$ together with a $\bfV$-linear map $\varphi_\mathscr{G}:\mathscr{G} \to \mathscr{G}$ which is $\tau$-semilinear. The category of $\tau$-modules over $\mathcal{B}$ is denoted by $\taumod(\mathcal{B})$. 
	    \end{definition}

        \begin{definition}
            Let $\mathscr{F}$ be an $F$-module.
            Assume that $\mathscr{F}$ is free with basis $\mathbf{e}=[e_1,\dots,e_n]^T$. 
            Then there exists a matrix $E_\mathscr{F} \in M_{n \times n}(\calB)$ such that $\phi(\mathbf{e})=E_\mathscr{F} \mathbf{e}$. We 
            refer to $E_\mathscr{F}$ as a Frobenius matrix of $\mathscr{F}$. We define
            a Frobenius matrix of a $\tau$-module
            in the analogous fashion. 
        \end{definition}

        \begin{definition}
            Let $E,E' \in M_{n \times n}(\calB)$. We say that $E$ and $E'$ are \emph{$F$-equivalent over $\calB$} if
            there exists $B \in GL_n(\calB)$ such that $B^{-1}CB^F= E'$. Note that
            $E$ and $E'$ are Frobenius structures for the same $F$-module if and only if they are $F$-equivalent over $\calB$. We define $\tau$-equivalence over $\calB$ in the analogous fashion. 
        \end{definition}
        \subsubsection{Restriction of scalars}
        We define a
        restriction functor
        \[\Res: \taumod(\mathcal{B}) \to \Fmod(\mathcal{B}),\]
        in the following way: Let $\mathscr{G}$ be a
        $\tau$-module. For simplicity, assume $\mathscr{G}$ is free (for the general case, write $\mathscr{G}$ as the summand of a free object and proceed accordingly). Let $E_\mathscr{G}$ be a Frobenius structure of $\mathscr{G}$. Then $\Res(\mathscr{G})$ is
        the $F$-module whose underlying module is $\bigoplus_{i=1}^b \mathscr{G}$ and whose Frobenius structure is given by the cyclic block matrix:
        \begin{equation}\label{eq: matrix for restriction}
            \begin{bmatrix} 0   & \dots & 0 &  E_\mathscr{G} \\ 
	1  & \dots & 0 & 0 \\
	0  & \ddots & \vdots & \vdots \\
	0 & \dots  &1 & 0
	\end{bmatrix}.
        \end{equation}
    The following proposition computes the restriction functor for a specific situation that is fundamental in \S \ref{s: affine line 1} and \S \ref{s: general ordinary curves}.

                \begin{proposition}
            \label{prop: shape of restriction of scalars Frobenius structure}
            Let $\beta_1,\dots,\beta_{b}  \in 1 + \pi\calB$.
            Let $\mathscr{G}$ be the $\tau$-module over $\calB$ whose underlying module is $\calB$
            and with Frobenius matrix $\beta_1\beta_2^F\dots \beta_{b}^{F^{b-1}}\in 1 + \pi\calB$.
            Then $\Res(\mathscr{G})$ has 
            \begin{equation}\label{eq: matrix for restriction specific}
                \begin{bmatrix} 0   & \dots & 0 &  \beta_{1} \\ 
	\beta_2  & \dots & 0 & 0 \\
	0  & \ddots & \vdots & \vdots \\
	0 & \dots  &\beta_{b} & 0
	\end{bmatrix}
            \end{equation}
            as a Frobenius matrix.
        \end{proposition}

        \begin{proof}
            This follows from the explicit description of the restriction functor (i.e. matrix \eqref{eq: matrix for restriction}) and then showing that 
            the matrix \eqref{eq: matrix for restriction specific} is $F$-equivalent. 
        \end{proof}

        \subsubsection{Unit-root \texorpdfstring{$F$}{F}-modules and \texorpdfstring{$\tau$}{tau}-modules.}

    \begin{definition}
        We say that an $F$-module $\mathscr{F}$ is
        unit-root if the linearization of $\phi_\mathscr{F}$ (i.e. the map $\phi_\mathscr{F}\otimes_F 1: F^* \mathscr{F} \to \mathscr{F}$) is
        an isomorphism. Similarly, we say that a $\tau$-module $\mathscr{G}$ is unit-root
        if the linearization of $\varphi_\mathscr{G}$ is an isomorphism. 
    \end{definition}

	    \begin{proposition}\label{p:Katz-correspondence}
	        Assume that $B$ is flat over $\F_r$. There is a rank-preserving equivalence between the category of unit-root $F$-modules (resp. $\tau$-modules) over $\mathcal{B}$ and the category of continuous $\bfV^\circ$-valued (resp. $\bfV$-valued) representations of $\pi_1(\Spec(B))$. 
            Furthermore, the restriction of scalar functor from $\tau$-modules to $F$-modules is compatible with the corresponding functor on the representation side of this correspondence. 
	    \end{proposition}
        \begin{proof}
            The analogous statement for $p$-adic representations is \cite{Katz-padic_properties_mod_forms_schemes}*{Chapter 3} (see also \cite{Crew-F_isocrystals_padic_reps}). The same proof works for $\bfV$-valued representations.
        \end{proof}

	    \subsubsection{Overconvergence}
            We now introduce the notion of overconvergence for $F$- and $\tau$-modules. 

            \begin{definition}\label{d:oc-local-tau-module}
                Let $\mathscr{F}$ be an $F$-module over $K_{\F_r} \widehat{\otimes}_{\F_p}\bfV^\circ$. Note that $\mathscr{F}$ is necessarily free. A Frobenius matrix $E_\mathscr{F}$ is called \emph{overconvergent} if the entries of $E_\mathscr{F}$ lie in $K_{\F_r} \otimes^\dagger_{\F_p}\bfV^\circ$.
                If $\mathscr{F}$ admits an overconvergent Frobenius matrix, then we say that $\mathscr{F}$ is an \emph{overconvergent} $F$-module. We make the analogous definition for $\tau$-modules.
            \end{definition}

            \begin{definition}\label{d: localization of F-module}
                Let $\mathscr{F}$ be an $F$-module over $A_{\F_r} \widehat{\otimes}_{\F_p}\bfV^\circ$. The \emph{localization} of $\mathscr{F}$ at $\infty$ is the extension of scalars 
                \begin{equation*}
                    \mathscr{F}_\infty := \mathscr{F} \otimes_{A_{\F_r} \widehat{\otimes}_{\F_p}\bfV^\circ} K_{\F_r} \widehat{\otimes}_{\F_p}\bfV^\circ.
                \end{equation*}We make the analogous definition for $\tau$-modules.
            \end{definition}

            \begin{definition}
                Let $\mathscr{F}$ be an $F$-module over $A_{\F_r} \widehat{\otimes}_{\F_p}\bfV^\circ$. We say that $\mathscr{F}$ is \emph{overconvergent} if the localization $\mathscr{F}_\infty$ is overconvergent in the sense of Definition \ref{d:oc-local-tau-module}.
            \end{definition}

            \begin{proposition}\label{p:locally-oc-implies-globally-oc}
                Let $\mathscr{F}$ be an $F$-module over $A_{\F_r} \widehat{\otimes}_{\F_p}\bfV^\circ$ whose underlying module is free of rank one. If $\mathscr{F}$ is overconvergent, then $\mathscr{F}$ admits a Frobenius matrix $E_\mathscr{F}$ satisfying the growth condition
                \begin{equation*}
                    E_\mathscr{F} \in A_{\F_r} \otimes_{\F_p}^\dagger \bfV^\circ.
                \end{equation*}
            \end{proposition}
            \begin{proof}
                This is a special case of \cite{kramermiller2021newton}*{Proposition 5.11}. 
            \end{proof}

    \subsection{Dwork operators and Fredholm determinants}

    \subsubsection{Full subspaces}
    Let $V$ be a free module over $K_{\F_r} \otimes_{\F_p}\bfV^\circ$ with basis $e_1,\dots,e_d$.
    By abuse of notation, we let $\pr$ denote the map 
    \begin{align*}
        \pr: V &\to \bigoplus_{i=1}^d \theta\bfV_{\F_r}          \langle \theta \rangle e_i \\
        \sum f_ie_i &\mapsto \sum \pr(f_i)e_i. 
    \end{align*}
    \begin{definition}
        Let $M$ be a sub-$\bfV$-module of $V$. We say
        that $M$ is a \emph{full subspace} of $V$ if the following conditions hold:
        \begin{enumerate}
            \item The projection map restricted to $M$ is surjective.
            \item The kernel $\ker(\pr|_M)$ has finite rank over $\bfV$. 
        \end{enumerate}
        
    \end{definition}

    Let $M$ be a full subspace of $V$. For any rational number $m>0$ we define 
    \begin{align*}
        M^m := M \cap \bigoplus_{i=1}^d K_{\F_r} \otimes^m_{\F_p}\bfV^\circ e_i, \\
        M^\dagger := \bigcup_{m>0} M^m = M \cap \bigoplus_{i=1}^d K_{\F_r} \otimes^\dagger_{\F_p}\bfV^\circ e_i.
    \end{align*}

    \begin{definition}\label{def: pole order basis}
        Let $M$ be a full subspace of $V$. Let $B_0$ be a basis of $\ker(\pr(M))$ over $\bfV$. For $k>1$ and
        $1 \leq i \leq d$, there exists $g_{i,k} \in M$ such that $pr(g_{i,k})=\theta^ke_i$. Then
        \begin{align*}
            B:= B_0 \sqcup \Big\{ g_{i,k}\Big\}_{\substack{k>1 \\ 1\leq i \leq d}}
        \end{align*}
        is an orthonormal basis of $M$ over $\bfV$\footnote{We say $B$ is an orthonormal basis of $M$ if every element can be written uniquely as $\sum_{g\in B} c_g g$ with $\lim_{g}c_g = 0$ and if this basis gives a norm on $M$. We are essentially ignoring the norm on $M$. See \cite{Serre-p-adicbanach} for more details.}. We refer to any basis of $M$ arising in this way as a \emph{pole order basis}. 
        
    \end{definition}
    \begin{definition}
        Let $B$ be a pole order basis defined using the notation in Definition \ref{def: pole order basis}. For any rational number $m>0$ we define the set
        \begin{align*}
            B^m:= B_0 \sqcup \Big\{ \pi^{\lceil k/m \rceil}g_{i,k}\Big\}_{\substack{k>1 \\ 1\leq i \leq d}}.
        \end{align*}
        Note that $B^m$ is an orthonormal basis of $M^m$ over $\bfV$. We refer to $B^m$ as a \emph{pole order basis with $m$-growth}.
    \end{definition}

    \begin{definition}
        Let $M$ be a full subspace of $V$. Let $B$ be
        a pole order basis of $M$ and let $B^m$ be the corresponding pole order basis with $m$-growth. Let $\zeta_1,\dots,\zeta_b$ be
        a basis of $\F_r$ over $\F_p$. Then we define
        \begin{align*}
            B^\circ := \Big\{ \zeta_i g \Big \}_{\substack{g \in B \\ 1\leq i \leq b}}, \\
            B^{\circ,m} := \Big\{ \zeta_i g \Big \}_{\substack{g \in B^m \\ 1\leq i \leq b}} .
        \end{align*}
        Note that $B^\circ$ is an orthonormal basis of $M$ over $\bfV^\circ$. We refer to any basis of $M$ over $\bfV^\circ$ that arises in this way as a \emph{pole order basis}. 
    \end{definition}

    \subsubsection{Dwork operators} \label{sss: dwork operators}

    We define the $U_p$-operator on $K_{\F_r}$ as follows:
    \begin{align*}
        U_p(c\theta^i) &= \begin{cases}
            c^{1/p}\theta^{i/p} & p \mid i \\
            0 & p \nmid i
        \end{cases}.
    \end{align*}
    We extend $U_p$ to a $\bfV^\circ$-linear map on $K_{\F_r} \otimes^\dagger_{\F_p}\bfV^\circ$. By abuse of notation
    we define
    \begin{align*}
        U_p: V &\to V \\
        \sum f_i e_i &\mapsto \sum U_p(f_i)e_i.
    \end{align*}
    Note that this map depends on the basis $\{e_1,\dots,e_d\}$.

    \begin{definition}
        Let $M \subset V$ be a full subspace. 
        We say that an $\bfV^\circ$-linear (resp. $\bfV$-linear) map $\Theta: M \to M$ is a $p$-Dwork
        operator (resp. $r$-Dwork operator) if it can be written as $U_p \circ E$ (resp. $U_p^b \circ E)$, where
        $E$ is a $d \times d$-matrix with entries in $K_{\F_r} \otimes^\dagger_{\F_p}\bfV^\circ$.
    \end{definition}

    \begin{proposition}
        \label{prop: Dwork operators are completely continuous for small radius}
        Let $M \subset V$ be a full subspace and let $\Theta$ be a $p$-Dwork operator (resp. $r$-Dwork operator) on $M$. Then $\Theta$ restricts
        to a completely continuous operator on $M^m$ for $m$ sufficiently large (i.e. $\Theta|_{M^m}$ is
        the limit of operators with finite rank image).
    \end{proposition}
    \begin{proof}
        This can be traced back to work of Monsky (see e.g. \cite{Monsky-formal_cohomologyIII}). Write $\Theta=U_p \circ E$ and take $m$ large enough so the entries of $E$ are in $K_{\F_r} \otimes^m_{\F_p}\bfV^\circ$. As $U_p$ takes $K_{\F_r} \otimes^m_{\F_p}\bfV^\circ$ to $K_{\F_r} \otimes^{p/m}_{\F_p}\bfV^\circ$ we see that 
        $\Theta(M^m) \subset M^{p/m}$. Since the inclusion $M^{p/m} \to M^m$ is completely continuous (compare $B^m$ and $B^{m/p}$ for a pole order basis $B$), we see that $\Theta$ is completely continuous on $M^m$. 
    \end{proof}

    \subsubsection{Fredholm determinants and characteristic series}

    \begin{definition}
        Let $I$ be a countable set and let $\Phi$ be an $I\times I$ matrix with entries in $\bfV$ or
        $\bfV^\circ$. We define the Fredholm determinant (when it exists) 
        to be the formal power series
            \begin{equation}\label{eq: Fredholm determinant 1}
                \det(I-x\Phi) = 1 + c_1 x + c_2 x^2 + \cdots\in 1+x\bfV\llbracket x \rrbracket,
            \end{equation}
            where the coefficient $c_n$ is given by the formula
            \begin{equation} \label{eq: Fredholm determinant 2}
                c_n = (-1)^n \sum_{\substack{S \subseteq I\\|S| = n}} \sum_{\sigma \in \mathrm{Sym}(S)} \mathrm{sgn}(\sigma) \prod_{i \in S} \Phi_{i,\sigma(i)}.
            \end{equation}
        Here $\Phi_{i,j}$ denotes $(i,j)$-th entry of $\Phi$.
    \end{definition}

    \begin{proposition}
        \label{prop: Fredholm exists for completely continuous} (Serre, \cite{Serre}) Let $Z$ be a Banach space over $\bfV$ (resp. $\bfV^\circ$)
        and let $B$ be an orthonormal basis of $Z$. Let $L:Z \to Z$ be a completely continuous $\bfV$-linear (resp. $\bfV^\circ$-linear) operator and let $\Phi$ be the matrix of $L$ with respect to the basis $B$. Then
        $\det(I-x\Phi)$ exists and is an entire function. Furthermore, $\det(I-x\Phi)$ is independent
        of the choice of basis. 
    \end{proposition}

    \begin{definition}
        Adopt the notation of Proposition \ref{prop: Fredholm exists for completely continuous}.
        We define the \emph{charactersitic series} $\det_{\bfV}(I-xL|Z)$ (resp. $\det_{\bfV^\circ}(1-xL|Z)$) to be the power series $\det(I-x\Phi)$. Note that that the characteristic series does not depend on our choice of basis by Proposition \ref{prop: Fredholm exists for completely continuous}.
    \end{definition}
    \begin{definition}
        Let $\Theta$ be a $p$-Dwork operator on a full subspace $M \subset V$. From Proposition \ref{prop: Dwork operators are completely continuous for small radius} and Proposition \ref{prop: Fredholm exists for completely continuous} we know the characteristic series $\det_{\bfV^\circ}(1-x\Theta|M^m)$ of $\Theta$ exists for $m$ sufficiently large. Furthermore, this series is independent of $m$. We define the characteristic series of $\Theta$ on $M$ to be
        \begin{align*}
            \det_{\bfV^\circ}(1-x\Theta|M)&:=\det_{\bfV^\circ}(1-x\Theta|M^m),
        \end{align*}
        where $m$ is taken to be sufficiently large. We make the analogous definition for $r$-Dwork operators.
    \end{definition}

    The next proposition says we can compute the characteristic series of a Dwork operator directly with a pole-ordered basis.
    \begin{proposition}
        \label{prop: Fredholm can be computed using pole-ordered bases}
        Let $\Theta$ be a $p$-Dwork operator on a full subspace $M \subset V$. Let
        $B^\circ$ be a pole-ordered basis of $M$ over $\bfV^\circ$ and let $\Phi$ be the matrix of $\Theta$ in terms of $B^\circ$.
        Then $\det(1-x\Phi)$ exists and
        \begin{align*}
            \det(1-x\Phi)= \det_{\bfV^\circ}(1-x\Theta|M).
        \end{align*}
    \end{proposition}

    \subsection{The trace formula} \label{ss: the trace formula}
    We now introduce the trace formula for computing $L$-functions. Let $\mathscr{F}$ be a unit-root $F$-module over $A_{\F_r} \widehat{\otimes}_{\F_p}\bfV^\circ$. We
    assume that $\mathscr{F}$ is free with basis $f_1,\dots,f_d$ and let $E_\mathscr{F}$ be the corresponding
    Frobenius matrix. Let $\Omega_{K_{\F_r} \widehat{\otimes}_{\F_p}\bfV^\circ}$ (resp. $\Omega_{A_{\F_r} \widehat{\otimes}_{\F_p}\bfV^\circ}$) denote the Kahler differentials of $K_{\F_r} \widehat{\otimes}_{\F_p}\bfV^\circ$ (resp. $A_{\F_r} \widehat{\otimes}_{\F_p}\bfV^\circ$) relative to $\bfV$. We define
    \begin{align*}
        V_\mathscr{F}&:= \mathscr{F} \otimes_{A_{\F_r} \widehat{\otimes}_{\F_p}\bfV^\circ} \Omega_{K_{\F_r} \widehat{\otimes}_{\F_p}\bfV^\circ} \\
        M_\mathscr{F}&:=\mathscr{F} \otimes_{A_{\F_r} \widehat{\otimes}_{\F_p}\bfV^\circ} \Omega_{A_{\F_r} \widehat{\otimes}_{\F_p}\bfV^\circ} \\
        e_i &:=f_i \otimes \frac{d\theta^{-1}}{\theta^{-1}}. 
    \end{align*}
    Note that $M_\mathscr{F} \subset V_\mathscr{F}$ and that $V_\mathscr{F}$ is a free $K_{\F_r} \widehat{\otimes}_{\F_p}\bfV^\circ$-module
    with basis $\{e_1,\dots,e_d\}$. In particular, we define the $U_p$ operator on $V_\mathscr{F}$ as in \S \ref{sss: dwork operators}. Finally, we define an operator on $V_\mathscr{F}$:
    \begin{align*}
        \Theta_\mathscr{F} := U_p \circ E_\mathscr{F}.
    \end{align*}

    \begin{remark}
        Although our definition of $U_p$ on $V_\mathscr{F}$ depends on the choice of basis $\{f_1,\dots,f_d\}$, one can readily verify that $\Theta_{\mathscr{F}}$ is independent of this choice. 
    \end{remark}

    \begin{lemma}
        \label{lemma: M is overconvergent-newname}
        The subset $M_\mathscr{F}$ of $V_\mathscr{F}$ is full and $\theta_\mathscr{F}$ is a $p$-Dwork operator on $M_\mathscr{F}$.
    \end{lemma}
    \begin{proof}
        The Riemann Roch theorem implies that $M_\mathscr{F}$ is full. For the second statement, we need to show that $U_p \circ E_\mathscr{F}(M_\mathscr{F}) \subset M_\mathscr{F}$. It is clear that $E_\mathscr{F}(M_\mathscr{F}) \subset M_\mathscr{F}$. Next, consider $\Omega_{A_{\F_r} \widehat{\otimes}_{\F_p}\bfV^\circ} f_i$, which is a subset of  $M_\mathscr{F}$. The restriction of $U_p$ to $\Omega_{A_{\F_r} \widehat{\otimes}_{\F_p}\bfV^\circ} f_i$ is simply the Cartier operator relative to $\bfV^\circ$ on $\Omega_{A_{\F_r} \widehat{\otimes}_{\F_p}\bfV^\circ}$.
        Thus, $U_p(\Omega_{A_{\F_r} \widehat{\otimes}_{\F_p}\bfV^\circ} f_i) \subset \Omega_{A_{\F_r} \widehat{\otimes}_{\F_p}\bfV^\circ} f_i$. It follows that $U_p(M_\mathscr{F}) \subset M_\mathscr{F}$, which gives the lemma.
    \end{proof}

    \begin{theorem}
        \label{theorem: trace formula}
        Let $\rho$ be the representation over $\bfV^\circ$ corresponding to $\mathscr{F}$ (given by Theorem \ref{p:Katz-correspondence}). If $\mathscr{F}$ is overconvergent then 
        \begin{align*}
            L(\rho,s) &= \det_{\bfV}(1-x\Theta_\mathscr{F}^b |M_\mathscr{F}).
        \end{align*}
    \end{theorem}
    \begin{proof}
        This is essentially \cite{Monsky-formal_cohomologyIII}*{Theorem 5.3}. We remark that even though Monsky's paper works with coefficients in characteristic $0$ (instead of $\bfV$), the proof works with minimal modifications in general. Alternatively, the proof of \cite{Anderson-trace_formula} can be modified to the overconvergent situation, as is thoroughly explained in \cite{Taguchi_wan_entireness}.
    \end{proof}

    \begin{corollary} \label{cor: a-th root trick}
        Let $\mathscr{G}$ be a free $\tau$-module corresponding to a representation $\rho$ over $\bfV$. 
        Let $\mathscr{F}=\Res(\mathscr{G})$. Then
        \begin{align*}
            \NP_\pi(L(\rho,x^b)^b) &=\NP_\pi(\det_{\bfV^\circ}(1-x\Theta_\mathscr{F} |M_\mathscr{F})).
        \end{align*}
    \end{corollary}
    \begin{proof}
        For any $f(x) \in \bfV\llbracket s \rrbracket$ we let $N_{\bfV/\bfV^\circ}(f(x))$ be the
        product over the conjugates $f(x)^\sigma$, where $\sigma \in Gal(\bfV/\bfV^\circ)$, and $\sigma$ is applied to the coefficients of $f(x)$. Then we have
        \begin{align*}
            L(\Res_{\bfV/\bfV^\circ}(\rho),x) &= N_{\bfV/\bfV^\circ}(L(\rho,x)).
        \end{align*}
        From Theorem \ref{theorem: trace formula} we have 
        \begin{align*}
            L(\Res_{\bfV/\bfV^\circ}(\rho),x) &= \det_{\bfV}(1-x\Theta_\mathscr{F}^b |M_\mathscr{F}).
        \end{align*}
        We know that
        \begin{align*}
            \det_{\bfV^\circ}(1-x\Theta_\mathscr{F}^b |M_\mathscr{F}) &= N_{\bfV/\bfV^\circ}(\det_{\bfV}(1-x\Theta_\mathscr{F}^b |M_\mathscr{F})).
        \end{align*}
        Since $L(\Res_{\bfV/\bfV^\circ}(\rho),x)$ is the $L$-function of an $\bfV^\circ$-valued representation, we know that $N_{\bfV/\bfV^\circ}( L(\Res_{\bfV/\bfV^\circ}(\rho),x))$ equals $ L(\Res_{\bfV/\bfV^\circ}(\rho),x)^b$. In particular, we have
        \begin{align*}
        N_{\bfV/\bfV^\circ}(L(\rho,x)) ^b&=  L(\Res_{\bfV/\bfV^\circ}(\rho),x)^b \\
            &=\det_{\bfV^\circ}(1-x\Theta_\mathscr{F}^b |M_\mathscr{F}).  
        \end{align*}
        Replacing $x$ with $x^b$ then gives
        \begin{align*}
            N_{\bfV/\bfV^\circ}(L(\rho,x^b)) ^b &= \det_{\bfV^\circ}(1-x^b\Theta_\mathscr{F}^b |M_\mathscr{F})\\
            &=\prod_{\zeta^b =1} \det_{\bfV^\circ}(1-\zeta x\Theta_\mathscr{F} |M_\mathscr{F}).
        \end{align*}
        Each $\displaystyle\det_{\bfV^\circ}(1-\zeta x\Theta_\mathscr{F} |M_\mathscr{F})$ has the same $\pi$-adic Newton polygon. Also, each Galois conjugate of $L(\rho,x^b)^b$ has the same Newton polygon.
        Thus,
        \begin{align*}
            \NP_\pi(L(\rho,x^b)^b) &=\NP_\pi(\det_{\bfV^\circ}(1-x\Theta_\mathscr{F} |M_\mathscr{F})).
        \end{align*}
    \end{proof}
    
    \section{Notation for matrices and permutations} \label{ss: some preliminaries on permutations}
    In this section we standardize some notation and conventions that will be used throughout the remainder if this article. 
    
    \subsection{Permutations}
	
		\subsubsection{Enriched permutations and matrices}
	       \label{sss:enriched permutaitons and matrices}
			Let $I$ be a set. Let $\Phi = (\Phi_{i,j})$ be an $I \times I$ matrix with entries in $\bfV_\infty$. For any subsets $I_1,I_2 \subset I$, we let $\Phi_{I_1\times I_2}$ denote the $I_1 \times I_2$ matrix whose $(i_1,i_2)$-th entry is $\Phi_{i_1,i_2}$. In light of the product formula \eqref{eq: Fredholm determinant 2} for the coefficients of $\det(I-s \Phi)$, the proof of Theorem \ref{t:Riemann-hypothesis} will mainly be concerned with the following notion:
			
			\begin{definition}
				An \emph{enriched permutation} of $I$ is a pair $(S,\sigma)$, where $S \subseteq I$ is a finite subset and $\sigma$ is a permutation of $S$. When there is no risk of confusion, we will often omit the underlying subset and refer to $(S,\sigma)$ simply by $\sigma$.
			\end{definition}
			
			If $(S,\sigma)$ is an enriched permmutation of $I$ then we can extend $\sigma$ to a permutation of $I$ by setting $\sigma(i) = i$ for $i \notin S$. If $(T,\tau)$ is a second enriched permutation of $I$, then we may form the composite enriched permutation $(S \cup T, \sigma \tau)$. We will say that $\sigma$ and $\tau$ are \emph{disjoint} if $S$ and $T$ are disjoint.
			
			\begin{definition}
				Let $(S,\sigma)$ be an enriched permutation of $I$. We define
				\begin{equation*}
					v(\sigma,\Phi)	=	v_\pi\left( \prod_{i \in S} \Phi_{i,\sigma(i)} \right) = \sum_{i \in S} v_\pi(\Phi_{i,\sigma(i)}).
				\end{equation*}
			\end{definition}
			
		\subsubsection{Rotational permutations}
		
			Let $X$ be a set. We will frequently encounter matrices indexed by sets of the form $I = \Z/b\Z \times X$. For any such $I$ let us write
			\begin{equation*}
				\mathbf{i}:I \to \Z/b\Z
			\end{equation*}
			for the projection onto the first factor. Most $I \times I$ matrices $\Phi$ of interest to us will be \emph{block cyclic} in the sense that $\Phi_{k_1,k_2} = 0$ unless $\bfi{k_2} = \bfi{k_1}-1$ in $\Z/b\Z$. If $\Phi$ is block cyclic, then for many enriched permutations of $I$ the corresponding term of \eqref{eq: Fredholm determinant 2} will vanish. The remaining permutations will be our primary object of study:
			
			\begin{definition}
				We say that an enriched permuation $(S,\sigma)$ is \emph{rotational} if $\bfi{\sigma(k)} = \bfi{k}-1$ for all $k \in S$. For any such permutation, the cardinality $|S|$ is a multiple of $b$. We define the \emph{size} of $(S,\sigma)$ to be
				\begin{equation*}
					\bfs{\sigma}	:=	\frac{|S|}{b}.
				\end{equation*}
			\end{definition}
			
			\begin{definition}
				Let $\Phi$ be a block cyclic $I \times I$ matrix. We say that a rotational permutation $\sigma$ is \emph{$v$-minimal of size $n$} for $\Phi$ if $\bfs{\sigma} = n$ and for any rotational permutation $\sigma'$ of $I$ with $\bfs{\sigma'} = n$, we have
				\begin{equation*}
					v(\sigma,\Phi)	\leq	 v(\sigma',\Phi).
				\end{equation*}
			\end{definition}
			
			Let $\Phi$ be an $I \times I$ matrix with Fredholm determinant $\det(I - x \Phi) = \sum_i c_i x^i$. If $\Phi$ is cyclic, then $c_i$ vanishes unless $i = bn$ is a multiple of $b$. If $\sigma$ is $v$-minimal of size $n$ for $\Phi$, by \eqref{eq: Fredholm determinant 2} we have the following inequality:
			\begin{equation*}
				v_\pi(c_{bn})	\geq	v(\sigma,\Phi).
			\end{equation*}
			Importantly, this is an equality if $\sigma$ is the \emph{unique} $v$-minimal permutation of size $n$ for $\Phi$.
			
	\subsection{Standard index sets}\label{ss: standard indexing sets}
	
		\subsubsection{Definitions}
	
			\begin{definition}
				Let $h$ be a positive integer. We define the standard index set
				\begin{equation*}
					J_h := \Z/b\Z \times \{1,\dots,h\} \times \Z_{> 0}.
				\end{equation*}
				For $k = (i,j,m) \in J_h$, we will write $|k| = m \in \Z_{> 0}$.
			\end{definition}
			
			For notational convenience we will identify $J_1 = \Z/b\Z \times \Z_{> 0}$. For each $h$, the set $J_h$ is equipped with a natural projection
			\begin{align*}
				\iota:J_h &\to J_1	\\
					(i,j,m)	&\mapsto	(i,m).
			\end{align*}
			This map restricts to a bijection on each of the subsets
			\begin{equation*}
				J_h^{(j)}	:=	\Z/b\Z \times \{j\} \times \Z_{> 0}.
			\end{equation*}
			
			\begin{definition}
				Let $\sigma$ be an enriched permutation of $J_h$. We say that $\sigma$ is \emph{decomposable} if each $J_h^{(j)}$ is invariant under the action of $\sigma$.
			\end{definition}
			
			Let $\sigma$ be a decomposable enriched permutation of $J_h$. Then there exist unique (possibly empty) enriched permutations $\sigma^{(j)}$ of $ J_h^{(j)}$ such that
			\begin{equation*}
				\sigma	=	\sigma^{(1)} \cdots \sigma^{(h)}.
			\end{equation*}
			
			\begin{definition}
				Let $\sigma$ be a decomposable enriched permutation of $J_h$. We associate to $\sigma$ an enriched permutation of $J_1$
			\begin{equation*}
				\iota_*\sigma	=	\iota_*\sigma^{(1)} \cdots \iota_*\sigma^{(j)},
			\end{equation*}
			where for each $j$, $\iota_* \sigma^{(j)}$ is the enriched permutation of $J_1$ corresponding to $\sigma^{(j)}$ under the bijection $\iota:J_h^{(j)} \to J_1$.
			\end{definition}
            Finally, we introduce the notion of $p$-bounded enriched permutations. This notion will be used extensively in \S \ref{s: affine line 1}.
             \begin{definition}
        Let $(\sigma,S)$ be a rotational enriched permutation of $J_h$. We say $\sigma$ is $p$-bounded if $|k| \leq p|\sigma(k)|$ for all $k \in S$. 
    \end{definition}
		
		\subsubsection{Notation for rotational \texorpdfstring{$b$}{b}-cycles}
		
			For the matrices studied below, the $v$-minimal permutations are very special in that they are decomposable and factor as a product of disjoint $b$-cycles. For such permutations, the map $\iota_*$ allows us to reduce many calculations to rotational $b$-cycles in $J_1$. We will now introduce notation to simplify the manipulation of such $b$-cycles:
			
			\begin{definition}
				Let $\Gamma = \Z^{\Z/b\Z}$ denote the abelian group of all $\Z/b\Z$-tuples of integers. If $\sigma = (k_b \cdots k_1)$ is a rotational $b$-cycle of $J_1$, then we define the \emph{coordinates} of $\sigma$ to be the tuple $(|k_i|)_{i \in \Z/b\Z} = (|k_1|,\dots,|k_b|)$.
			\end{definition}
			
			The coordinate map provides a bijection between rotational $b$-cycles of $J_1$ and the subset $\Gamma^+$ consisting of tuples in $\Gamma$ with positive entries. We will frequently identify a rotational $b$-cycle $\sigma$ with its coordinate vector. For example, for any $\vec{m} \in \Gamma$ we may form the sum $\sigma + \vec{m}$ even when the latter does not represent a rotaional $b$-cycle.
			
		\subsubsection{Ordering of permutations}
		
			\begin{definition}
				We equip the set of rotational enriched permutations of $J_h$ with a preorder as follows: Given $(S,\sigma)$ and $(T,\tau)$, we write $\sigma \leq \tau$  if for all $k \in S$ and $l \in T$ we have $|k| \leq |l|$ whenever $\bfi{k} = \bfi{l}$.
			\end{definition}
			
			The preceding definition defines a partial order on the set of rotational enriched permutations of each $J_h^{(j)}$. If $\sigma$ and $\tau$ are decomposable enriched permutations of $J_h$, then
			\begin{equation*}
				\sigma \leq \tau\Longleftrightarrow \iota_* \sigma \leq \iota_*\tau.
			\end{equation*}
			Observe that if $\sigma = (m_1,\dots,m_b)$ and $\tau = (n_1,\dots,n_b)$ are rotational $b$-cycles of $J_1$, then $\sigma \leq \tau$ is equivalent to $m_i \leq n_i$ for all $i$.
			
			\begin{definition}
				Let $\sigma$ be a rotational permutation of $J_h$. We say that $\sigma$ is \emph{lexicographical} if there exist disjoint $b$-cycles $\sigma_1 \leq \cdots \leq \sigma_n$ such that $\sigma = \sigma_1 \dots \sigma_n$.
			\end{definition}

    \section{Curves over \texorpdfstring{$\F_p$}{Fp} with \texorpdfstring{$\infty$}{infty} rational}
    \label{s: toy example}
    We will continue with the notation from \S \ref{s: tau crystals and the trace formula}. We will make the additional assumption
    that $\F_r$ and $\F_q$ are both equal to $\F_p$, i.e. the curve
    $X$ is defined over $\F_p$ and $\infty$ is a point of degree one.
    In particular, note that $F$-modules are the same thing as $\tau$-modules.
    The purpose of this section is to prove Theorem \ref{t:Riemann-hypothesis} in this special case. The proof is significantly simpler than the general case,
    but still highlights most of the key ideas.

    \subsection{The Riemann hypothesis for \texorpdfstring{$\F_p[\theta]$}{Fp[theta]}}
    We first treat the case where $X=\mathbb{P}^1_{\F_p}$ and $\infty$ is the usual point at infinity. Then $A= \F_p[\theta]$. Since $\F_p[\theta]$ has
    class number one we have $\bfV=\bfR$. Also, we have $\bfV^\circ=\bfV$ since
    both have the same residue field.
    We choose our uniformizer
    of $\bfV$ to be $\pi=\theta^{-1}$. Fix $y \in \Z_p$ that is not a positive integer.
    \subsubsection{The \texorpdfstring{$F$}{F}-module} \label{sss: toy example F module}
    Recall that in \S \ref{sss: 1-unit characters} we defined
    a character $\rho_{\F_p[\theta],\pi}:\pi_1(\mathbb{A}_{\F_p}^1) \to \bfV^\times$. From Proposition \ref{p:Katz-correspondence}
    the character $\rho_{\F_p[\theta],\pi}$ corresponds to a rank one
    $F$-module $\mathscr{F}_{\F_p[\theta]}$ over $\bfV\langle \theta \rangle$. Let $\mathscr{F}_{\F_p[\theta]}^{\otimes y}$ be the $F$-module
    associated to $\rho_{\F_p[\theta],\pi}^{\otimes y}$.
    Consider the $p$-adic expansion
    \begin{align*}
        y&= \sum_{n=0}^\infty y_n p^n.
    \end{align*}
    We define
    \begin{align}\label{eq: Frob matrix element}
        \beta&:= \prod_{n=0}^\infty (1 - \pi^{p^n}\theta)^{y_n}.
    \end{align}
    \begin{proposition}
        \label{prop: taguchi wan over F_p and more}
        The element $\beta$ is a Frobenius matrix for $\mathscr{F}_{\F_p[\theta]}^{\otimes y}$. In particular, $\mathscr{F}_{\F_p[\theta]}^{\otimes y}$ is overconvergent.
    \end{proposition}
    \begin{proof}
        The element $1-\pi\theta$ is a Frobenius matrix for $\mathscr{F}_{\F_p[\theta]}$. This is proven in \cite{Taguchi-Wan} or it can be observed by comparing the pullback of $\mathscr{F}_{\F_p[\theta]}$ along closed points to the image of
        $\rho_{\F_p[\theta],\pi}$ evaluated along the corresponding Frobenius elements. We see that $(1-\pi\theta)^y$ is a Frobenius matrix for $\mathscr{F}_{\F_p[\theta]}^{\otimes y}$. From the generalized binomial theorem we deduce that $\mathscr{F}_{\F_p[\theta]}^{\otimes y}$ is overconvergent. Finally, note that $\beta$ is $F$-equivalent to $(1-\pi\theta)^y$ over $\bfV\langle \theta \rangle$, which proves the proposition.
    \end{proof}
    Let $\{f_1\}$ be a basis of $\mathscr{F}_{\F_p[\theta]}^{\otimes y}$ such that the corresponding Frobenius matrix is $\beta$. Define $V_{\mathscr{F}_{\F_p[\theta]}^{\otimes y}}$, $M_{\mathscr{F}_{\F_p[\theta]}^{\otimes y}}$, and $e_1$ as in \S \ref{ss: the trace formula}. In our specific situation we have: 
    \begin{align*}
        M_{\mathscr{F}_{\F_p[\theta]}^{\otimes y}} &= \theta \bfV\langle \theta \rangle e_1, \\
        \Theta_{\mathscr{F}_{\F_p[\theta]}^{\otimes y}}&= U_p \circ \beta. 
    \end{align*}
    From Proposition \ref{p: zeta function via L-functions} and Theorem \ref{theorem: trace formula} we have
    \begin{align}\label{eq: affine line trace formula F_p}
        \zeta_{\F_p[\theta],\pi}(x,y) = \det_{\bfV} \Big (1 - x U_p \circ \beta ~|~ M_{\mathscr{F}_{\F_p[\theta]}^{\otimes y}} \Big ).
    \end{align}
    Define the set $B=\{ \theta^ke_1 ~\Big | ~ k\geq 1  \}$. We view $B$ as being indexed by $J_1=\Z_{>1}$ so that $k$ corresponds to $\theta^k e_1$. Note that
    $B$ is a pole order basis of $M_{\mathscr{F}_{\F_p[\theta]}^{\otimes y}}$ over $\bfV$. Let $\Psi$ be the matrix of $U_p \circ \beta$ acting on $M_{\mathscr{F}_{\F_p[\theta]}^{\otimes y}}$ in terms of the basis $B$.
    Then from Proposition \ref{p: we can compute using matrix} and \eqref{eq: affine line trace formula F_p} we have
    \begin{align}\label{eq: affine line trace formula F_p matrix}
        \zeta_{\F_p[\theta],\pi}(x,y) = \det \Big (1 - x \Psi  \Big ).
    \end{align}

    \subsubsection{The matrix \texorpdfstring{$\Psi$}{Psi}}
    For $n\geq 1$ we define $d(n)$ to be $p^w$ where $w$ is the unique number
    satisfying
    \begin{align*}
        \sum_{j=0}^{w-1} y_j< n \leq \sum_{j=0}^w y_j,
    \end{align*}
    For $n<1$ we define $d(n)=0$. Note that $y=\sum d(n)$.   For $m \in \Z$ we define
    \begin{align*}
        y(m):= \sum_{n=1}^{m} d(n),
    \end{align*}
    where we have the sum equal $0$ when $m<1$.
    These numbers give the valuations of the coefficients of $\beta$.
    More precisely, write
    \begin{align*}
        \beta&= \sum_{-\infty }^\infty a_n \theta^n,
    \end{align*}
    where we adopt the convention that $a_n=0$ for $n<0$.
    A quick calculation using the product formula \eqref{eq: Frob matrix element} for $\beta$ yields:
    \begin{proposition}
        \label{p: valuation of frob matrix F_p}
        We have $v_\pi(a_n)=y(n)$ for all $n\geq 1$. 
    \end{proposition}
    For $m_1,m_2 \in \Z_{\geq 1}$ we calculate
    \begin{align*}
        \Psi_{m_1,m_2} &=a_{pm_1-m_2},
    \end{align*}
    so that
    \begin{align}\label{eq: valuation of }
        v_\pi(\Psi_{m_1,m_2}) &=\begin{cases}
            y(pm_1-m_2) & m_2 \leq pm_1 \\
            \infty & \text{ otherwise} 
        \end{cases} .
    \end{align}

    \begin{definition}
        We define a function $R:J_1 \times J_1 \to \Z_{\geq 0}$ as follows: For $k_1,k_2 \in J_1$
        we set
        \begin{align*}
            R(k_1,k_2) &:= y(pk_1-k_2). 
        \end{align*}
        For an enriched permutation $(S,\sigma)$ of $J_1$ we define
        \begin{align*}
            R(\sigma) &:= \sum_{k \in S} R(k,\sigma(k)),
        \end{align*}
        and we refer to $R(\sigma)$ as the $R$-value of $\sigma$.
    \end{definition}
    The following lemma compares the valuation and $R$-value of $(S,\sigma)$.
    \begin{lemma}
        \label{l: F_p R compare with v}
        We have $v(\sigma,\Psi)\geq R(\sigma)$ with equality if and only if
        $\sigma$ is $p$-bounded. 
    \end{lemma}
    \begin{proof}
        This is immediate from the definition of $R$ and \eqref{eq: valuation of }.
    \end{proof}

     \begin{definition}
        Let $S \subset J_1$. We let $(I_S,S)$ be the enriched permutation that is the identity on $S$. That is, we have $I_S(k)=k$ for all $k \in S$.
        
    \end{definition}
    \begin{definition}
        We define $S_n$ to be the set $\{1,\dots,n\}$, which we regard as a subset of $J_1$.
    \end{definition}
    
    \begin{proposition}
        \label{p: minimal fixed point permutation F_p}
        Let $(S,\sigma)$ be an enriched permutation of $J_1$. Then we have $R(I_S) \leq R(\sigma)$,
        with equality if and only if $\sigma$ equals $I_S$.
    \end{proposition}
    \begin{proof}
        We proceed by induction on $|S|$. If $|S|=1$ the proposition is clear.
        Assume the statement holds for subsets of $J_1$ of cardinality $n-1$. Let $(S,\sigma)$ be an enriched permutation with $|S|=n$ that is
        not the identity on $S$. We may assume further that $\sigma$ does not have any fixed points (otherwise the proposition follows immediately from the inductive hypothesis).
        Let $m$ be the smallest integer in $S$. Define a permutation $(\sigma', S \backslash \{m\})$ by
        \begin{align*}
            \sigma'(k) &= \begin{cases}
                \sigma(k) & k \neq \sigma^{-1}(m) \\
                \sigma(m) &  k = \sigma^{-1}(m)
            \end{cases}.
        \end{align*}
        It is enough to prove $R(I_{\{m\}}\sigma')< R(\sigma)$. 
        To see this, we compute:
        \begin{align}
        \begin{split}\label{eq: toy case identity lemma}
            R(\sigma) - R(I_{\{m\}}\sigma') &= R(\sigma^{-1}(m),m) + R(m,\sigma(m)) - R(\sigma^{-1}(m),\sigma(m)) - R(m,m) \\
            &=\sum_{k=p\sigma^{-1}(m)-\sigma(m)+1}^{p\sigma^{-1}(m)-m} d(k)-\sum_{k=pm-\sigma(m)+1}^{pm-m}d(k) \\
            &= \sum_{k=1}^{\sigma(m)-m}\Big[ d(p\sigma^{-1}(m) - \sigma(m) + k) - d(pm - \sigma(m) + k)\Big].
        \end{split}
        \end{align}
        Note that $d$ is monotonically increasing and is strictly increasing for positive integers. As $\sigma^{-1}(m)>m$, we see that every term in the last sum of \eqref{eq: toy case identity lemma} is non-negative. 
        By considering $k=\sigma(m)-m$, we see that this sum is indeed positive.

    \end{proof}
        \begin{corollary}\label{c: unique minimal permutation}
            The enriched permutation $(I_{S_n},S_n)$ is the only $v$-minimal enriched permutation of $\Psi$ of size $n$. 
        \end{corollary}
        \begin{proof}
            By Lemma \ref{l: F_p R compare with v} and Proposition \ref{p: minimal fixed point permutation F_p} we know
            any $v$-minimal permutation is of the form $(I_S,S)$. We also know
            $v(I_S,\Psi)=\sum_{m \in S}y(pm-m)$. The corollary follows by observing that $y(m_1)>y(m_0)$ whenever $m_1>m_0$.
        \end{proof}

        \begin{remark}
            Generalizing Corollary \ref{c: unique minimal permutation} to
            $A=\F_q[\theta]$ is the principle obstacle encountered when moving beyond the case of curves over $\F_p$. The analogous result (Theorem \ref{theorem: unique minimal term}) for $\F_q[\theta]$ is substantially more difficult and technical.
        \end{remark}
        \begin{theorem} \label{t: Wan's theorem}
            (Wan) The slopes of $\NP_\pi(\zeta_{\F_p[\theta],\pi}(x,y))$
            are $y(p-1),y(2p-2), \dots$, and each slope occurs with multiplicity one.
        \end{theorem}
        \begin{proof}
            From \eqref{eq: affine line trace formula F_p matrix} it suffices
            to consider the Newton polygon of $\det(1-x\Psi)$. Write
            \begin{align*}
                \det(1-x\Psi) &= 1 + c_1x + c_2x^2 + \dots.
            \end{align*}
            By Corollary
            \ref{c: unique minimal permutation} there is a unique term in the sum \eqref{eq: Fredholm determinant 2} defining $c_n$ with minimal valuation. In particular, we see that
            \begin{align*}
                v_\pi(c_n) &= v(I_{S_n},\Psi) = \sum_{k=1}^n y(pk-k).
            \end{align*}
            This gives the theorem.
        \end{proof}
        \subsection{Higher genus curves over \texorpdfstring{$\F_p$}{Fp}} \label{ss: higher genus over F_p}
        We now treat the case where $X$ is a curve over $\F_p$ and
        $\infty$ is a point of degree $1$. In this case we have $\bfR=\F_p\llbracket \pi \rrbracket$ and $\bfV=\F_p\llbracket \pi^{1/p^h}\rrbracket$, where $h$ is some positive integer.
        \subsubsection{The intial setup}
        Let $\rho_{A,\pi}$ be the representation defined in \S \ref{ss:zeta}.
        Let $\mathscr{F}_A$ be the $F$-module associated to $\rho_{A,\pi}$
        using Proposition \ref{p:Katz-correspondence} and let $\mathscr{F}_A^{\otimes y}$ be the $F$-module associated to $\rho_{A,\pi}^{\otimes y}$. We let $\mathscr{F}_{A,\infty}^{\otimes y} $ be the localization of $\mathscr{F}_A^{\otimes y}$ at $\infty$ as in Definition \ref{d: localization of F-module}.

        \begin{proposition}
            \label{p: Frobenius element for general curve over F_p}
            The folloing hold:
            \begin{enumerate}
                \item The $A \widehat{\otimes}_{\F_p} \bfV^\circ$-module
                underlying $\mathscr{F}_A^{\otimes y}$ is free of rank one.
                Furthermore, the Frobenius matrix $\alpha$ of $\mathscr{F}_A^{\otimes y}$ can be taken to satisfy $\alpha\equiv 1 \mod \pi^{1/p^h}$.
                \item Recall the element $\beta \in \bfV\langle \theta \rangle$ from \eqref{eq: Frob matrix element}. Then $\beta$
                is a Frobenius matrix of $\mathscr{F}_{A,\infty}^{\otimes y}$.
                In particular, there exists $d \in K \widehat{\otimes}_{\F_p} \bfV^\circ$ with $d \equiv 1 \mod \pi^{1/p^h}$ such that
                $\beta = d\alpha d^{-F}$.
            \end{enumerate}
        \end{proposition}
        \begin{proof}
            The first statement is the same as Proposition \ref{proposition: Goss crystal is free, OC and 1 mod pi}. For the second statement,
            we know from Corollary \ref{c: localization does not depend on curve} that the localizations of $\rho_{A,\pi}^{\otimes y}$ and $\rho_{\F_p[\theta],\pi}^{\otimes y}$ to $\Spec(K)$ give the same representation.
            In particular, $\mathscr{F}_{A,\infty}^{\otimes y}$ and $\mathscr{F}_{\F_p[\theta],\infty}^{\otimes y}$ are isomorphic $F$-modules.
            Thus, the result follows from Proposition \ref{prop: taguchi wan over F_p and more}.
        \end{proof}
        Fix $\alpha$ and $d$ as in Proposition \ref{p: Frobenius element for general curve over F_p}. Let $\{f_1\}$ be a basis of $\mathscr{F}_A^{\otimes y}$
        so that the corresponding Frobenius matrix is $\alpha$. 
        
        Let $V_{\mathscr{F}_{A}^{\otimes y}}$, $M_{\mathscr{F}_{A}^{\otimes y}}$, and $e_1$ be defined as in \S \ref{ss: the trace formula}. Let $\pr$ be the corresponding projection map and define $U_p$ in terms of this basis as in \S \ref{sss: dwork operators}. By Theorem \ref{theorem: trace formula} and Proposition \ref{p: zeta function via L-functions} we have
        \begin{align}
        \begin{split}\label{eq: F_p zeta function via det}
            \zeta_{A,\pi}(x,y) &= L(\rho_{A,\pi}^{\otimes y},x) \\
            &=\det_{\bfV}(1-xU_p \circ \alpha ~|~M_{\mathscr{F}_{A}^{\otimes y}}) \\
            &=\det_{\bfV}(1-xU_p \circ \beta ~|~d^FM_{\mathscr{F}_{A}^{\otimes y}}). 
        \end{split}
        \end{align}

        \subsubsection{A matrix representation\texorpdfstring{ of $U_p \circ \beta$}{}}
        We now choose a pole-order basis of $d^FM_{\mathscr{F}_{A}^{\otimes y}}$ and describe the matrix of $U_p\circ \beta$
        in terms of this basis.
        \begin{lemma}\label{p: proj is surjective}
            The restriction of $\pr$ to $d^FM_{\mathscr{F}_{A}^{\otimes y}}$ is surjective and $\ker(\pr)$ is free over $\bfV$ of rank
            $g$. 
        \end{lemma}
        \begin{proof}
            It suffices to prove the corresponding result modulo $\pi$. Since $d \equiv 1 \mod \pi^{1/p^h}$ and since $\mathscr{F}_{A}^{\otimes y}$ is free of rank one, we are reduced to considering the map
            \begin{align*}
                \overline{\pr}:\Omega_{A} &\to \theta \F_p[\theta] \frac{d\theta^{-1}}{\theta^{-1}}, \\
                \sum_{-\infty}^{n} a_k \theta^k \frac{d\theta^{-1}}{\theta^{-1}}&\mapsto \sum_{1}^{n} a_k \theta^k \frac{d\theta^{-1}}{\theta^{-1}}.
            \end{align*}
            The kernel consists precisely of the elements of $\Omega_A$ that have at worst a simple pole at $\infty$. Since it's impossible for a differential to have a simple pole at exactly one point (residues sum to zero) we see that the kernel consists of the holomorphic differentials. From this observation we see from Riemann Roch that the kernel has the correct dimension and the map is surjective.
        \end{proof}
        For $k \geq 1$, let $g_k$ be an element of $d^FM_{\mathscr{F}_{A}^{\otimes y}}$ satisfying $\pr(g_k)=\theta^k e_1$.
        We know $g_k$ exists from Lemma \ref{p: proj is surjective}. We then define
        \[ B^{mero}:= \{ g_k \} _{k\geq 1}.\]
        Let $I$ be an indexing set with $|I|=g$. From Lemma \ref{p: proj is surjective} there exists a basis $B^{hol}$ of $\ker(\pr)$, which we view as being indexed by $I$. Then
        \[B:= B^{mero} \sqcup B^{hol}\]
        is a pole order basis of $d^FM_{\mathscr{F}_{A}^{\otimes y}}$
        over $\bfV$, which is indexed by $I \sqcup J_1$. Let $\Delta$ be the matrix of $U_p \circ \beta$ with respect to $B$. From Proposition \ref{prop: Fredholm can be computed using pole-ordered bases} and \eqref{eq: F_p zeta function via det} we have
        \begin{equation}\label{eq: F_p zeta via Fredholm for matrix}
            \zeta_{A,\pi}(x,y) = \det(1-x\Delta).
        \end{equation}
        \subsubsection{The matrix \texorpdfstring{$\Delta$}{Delta}}
        \begin{proposition}\label{p: F_p congruence for higher genus matrix}
            Let $k_1 \in  J_1$ and let $k_2 \in  I \sqcup J_1$. 
            Then
            \begin{align*}
                \Delta_{k_1,k_2} \equiv \begin{cases}
                    a_{pk_1 - k_2} \mod \pi^{y(pk_1)} & \text{ if }k_2 \in J_1 \\
                    0 \mod \pi^{y(pk_1)} & \text{ if }k_2 \in I
                \end{cases}.
            \end{align*}
        \end{proposition}
        \begin{proof}
            By the definition of $g_{k_2}$ we may write \[g_{k_2} = \Big[\theta^{k_2} + \sum_{n=0}^\infty c_n\theta^{-n}\Big]e_1.\]
            Then we have
            \begin{align*}
                U_p \circ \beta (g_{k_2}) = \Big[ U_p\Big( \beta\Big ( \theta^{k_2} + \sum_{n=0}^\infty c_n\theta^{-n}\Big ) \Big ) \Big ] e_i.
            \end{align*}
            If $k_1 \in J_1$, the coefficient of $\theta^{k_1}e_1$ is
            \begin{align*}
                a_{pk_1-k_2} + \sum_{n=0}^\infty a_{pk_1 + n}c_n & \equiv a_{pk_1-k_2} \mod \pi^{y(pk_1)}. 
            \end{align*}
            This implies the congruence for $\Delta_{k_1,k_2}$. The statement for $k_2 \in I$ is similar. 
        \end{proof}

        \begin{corollary}\label{c: F_p valuation for higher genus matrix}
            Let $k_1 \in J_1$ and let $k_2 \in I \sqcup J_1$. 
            Then
            \begin{align*}
                v(\Delta_{k_1,k_2}) &= y(pk_1 - k_2)  \text{ if }k_2 \in J_1 \text{ and } pk_1 \geq k_2 \\
                v(\Delta_{k_1,k_2}) &\geq y(pk_1) \text{ otherwise}.
            \end{align*}
        \end{corollary}

        \begin{corollary}
            \label{c: F_p valuation of p-bounded for higher genus}
            Let $\sigma$ be a $p$-bounded enriched permutation of $J_1$. Then
            $v(\sigma,\Delta)=v(\sigma,\Psi)$.
        \end{corollary}

        \begin{corollary}\label{c: F_p det of hol part is a unit}
            If $X$ is ordinary, then $\det(\Delta_{I\times I}) \in \bfV^\times$.
        \end{corollary}
        \begin{proof}
            Since $y(k)>0$ for all $k\geq 1$, we know by Proposition \ref{p: F_p congruence for higher genus matrix} and Corollary \ref{c: F_p valuation for higher genus matrix} that:
            \begin{equation*}
                \Delta \equiv \begin{bmatrix}
                    \Delta_{I \times I} & * \\ 
                    0 & \Delta_{J_1 \times J_1}    
                \end{bmatrix}   
                \equiv \begin{bmatrix}
                    \Delta_{I \times I} & * \\ 
                    0 & \Psi    
                \end{bmatrix}   \mod \pi^{1/p^h} ,
            \end{equation*}
            where $\Psi$ is the matrix from \S \ref{sss: toy example F module}.
            Thus, from \eqref{eq: affine line trace formula F_p matrix} and \eqref{eq: F_p zeta via Fredholm for matrix}
            we have
            \begin{align*}
                \zeta_{A,\pi}(x,y) \equiv \det(1-x\Delta_{I \times I})\zeta_{\F_p,\pi}(x,y) \mod \pi^{1/p^h}.
            \end{align*}
            From Theorem \ref{t: Wan's theorem} we know $\zeta_{\F_p,\pi}(x,y)\equiv 1 \mod \pi^{1/p^h}$. Also, we  know that $\zeta_{A,\pi}(x,y)$ modulo $\pi^{1/p^h}$
            is the same as the Weil zeta function $\zeta(\Spec(A),x)$
            of $\Spec(A)$ reduced modulo $p$. Since $X$ is assumed to be ordinary, we see $\zeta(\Spec(A),x) \mod p$ is a polynomial
            of degree $g$, i.e. $\zeta(\Spec(A),x)$ has $g$ unit roots.
            It follows that $\det(1-x\Delta_{I \times I}) \mod \pi$
            has degree $g$. As $\Delta_{I \times I}$ is a $g\times g$-matrix
            we find $\det(\Delta_{I \times I})$ is a unit in $\bfV$.
            
        \end{proof}
        \subsubsection{Minimal permutations of \texorpdfstring{$\Delta$}{Delta}}
        We now determine the $v$-minimal permutations of $\Delta$.
        \begin{proposition}
            \label{p: F_p case v-minimal higher genus}
            Write $n=n_0 +g$. Let $\sigma$ be a $v$-minimal enriched permutation of $\Delta$ of size $n$. Then $\sigma=\tau I_{S_{n_0}}$ where $\tau$ is a permutation of $I$ with
            $v(\tau,\Delta)=0$.
        \end{proposition}
        The proof is broken up into several steps. We first extend the definition of $R(-,-)$ to $I \sqcup J_1$.
        \begin{definition}
            Let $k_1,k_2 \in I \sqcup J_1$. We define
            \begin{align*}
                R(k_1,k_2) &= \begin{cases}
                    0 & k_1 \in I \\
                    y(pk_1) & k_1 \in J_1 \text{ and } k_2 \in I \\
                    y(pk_1 - k_2) & k_1,k_2 \in J_1
                \end{cases}.
            \end{align*}
            For an enriched permutation $(S,\sigma)$ of $I \sqcup J_1$ we define
            \begin{align*}
                R(\sigma) &= \sum_{k\in S} R(k,\sigma(k)).
            \end{align*}
            We say that $\sigma$ is $R$-minimal of size $n$ if $\bfs{\sigma}=n$ and if
            \begin{align*}
                R(\sigma) &= \min_{\bfs{\sigma'}=n} R(\sigma').
            \end{align*}
        \end{definition}
        \begin{lemma}
            \label{l: F_p v and R comparison}
            Let $(S,\sigma)$ be an enriched permutation of $J_1\sqcup I$. Then we have
            $v(\sigma,\Delta)\geq R(\sigma)$. Furthermore, $v(\sigma, \Delta)=R(\sigma)$ if $\sigma=\tau\sigma_0$
            where $\tau$ is a permutation of $I$ with $v(\tau,\Delta)=0$ and
            $\sigma_0$ is a $p$-bounded permutation of $J_1$
        \end{lemma}
        \begin{proof}
            This follows from the definition of $R$ and Corollary \ref{c: F_p valuation for higher genus matrix}.
        \end{proof}

        \begin{lemma}\label{l: F_p decomposing permutation higher genus}
            Let $(S,\sigma)$ be a rotational enriched permutation of $I\sqcup J_1 $. Assume $\sigma$ does not decompose into the product of
            an enriched permutation of $J_1$ and an enriched permutation of $I$. Then $\sigma$ is not $R$-minimal.
        \end{lemma}
        \begin{proof}
            Let $H=S\cap I$ and $K = S \cap J_1$. There exists $k_1 \in H$
            with $\sigma(k_1) \in K$ and there exists $k_2 \in K$ with
            $\sigma(k_2) \in H$. Define an enriched permutation $(\sigma',S)$ by
            \begin{align*}
                \sigma'(k) &=\begin{cases}
                    \sigma(k) & k \neq k_1,k_2 \\
                    \sigma(k_2) & k=k_1 \\
                    \sigma(k_1) & k=k_2
                \end{cases}.
            \end{align*}
            Since $k_1 \in I$ we know $R(k_1,\sigma(k_1))$ and $R(k_1,\sigma(k_2))$ both vanish. This gives:
            \begin{align*}
                R(\sigma) - R(\sigma') &= R(k_1,\sigma(k_1)) - R(k_1,\sigma(k_2)) + R(k_2, \sigma(k_2)) - R(k_2,\sigma(k_1)) \\
                &= R(k_2,\sigma(k_2)) - R(k_2,\sigma(k_1)) \\
                &= y(pk_2) - y(pk_2-\sigma(k_1)) > 0.
            \end{align*}
            
        \end{proof}

        \begin{lemma}
            \label{l: F_p R minimal have a specific shape. }
            Let $\sigma$ be an $R$-minimal enriched permutation of size $n=n_0+g$. Then 
            $\sigma=\tau I_{S_{n_0}}$ where $\tau$ is a permutation of $I$.
        \end{lemma}
        \begin{proof}
            From Lemma \ref{l: F_p decomposing permutation higher genus}
            and Proposition \ref{p: minimal fixed point permutation F_p} we know
            $\sigma=\tau I_{S_{n_1}}$, where $\tau$ is an enriched permutation
            of $I$. Since $R(\tau)=0$ we have $R(\sigma)=R(I_{S_{n_1}})$. 
            We have $n_1=n-\bfs{\tau}$, so $n_1 \geq n_0$. Since $R(I_{S_{n_1}})\geq R(I_{S_{n_0}})$ with equality only if $n_0=n_1$, the $R$-minimality of $\sigma$ implies $\sigma=\tau I_{S_{n_0}}$.
        \end{proof}

        \begin{proof}
            (Of Proposition \ref{p: F_p case v-minimal higher genus})
            Let $\sigma_0$ be an $R$-minimal enriched permutation of size $n$. 
            From Lemma \ref{l: F_p R minimal have a specific shape. }
            we know $\sigma_0 =\tau I_{S_{n_0}}$ where $\tau$ is a permutation
            of $I$. Since $R(\tau)$ is zero regardless of the underlying permutation of $\tau$, we may assume that $v(\tau,\Delta)=0$ (such a permutation exists by Corollary \ref{c: F_p det of hol part is a unit}.) In particular, from Lemma \ref{l: F_p R compare with v} we see $R(\sigma_0)=v(\sigma_0,\Delta)$ and that $\sigma_0$ is $v$-minimal.

            Let $\sigma$ be a $v$-minimal permutation of $\Delta$.
            We have
            \[R(\sigma_0) = v(\sigma_0,\Delta)=v(\sigma,\Delta) \geq R(\sigma).\]
            In particular $\sigma$ is $R$-minimal. The proposition then follows from Lemma \ref{l: F_p R minimal have a specific shape. } and Lemma \ref{l: F_p R compare with v}.
        \end{proof}

        \subsubsection{Proof of Theorem \ref{t:Riemann-hypothesis}
        for curves over \texorpdfstring{$\F_p$}{Fp}}
        We continue with the notation from the beginning of \S \ref{ss: higher genus over F_p} with the additional assumption that
        $X$ is ordinary. From \eqref{eq: F_p zeta via Fredholm for matrix}
        it suffices to consider
        \[ \det(1-s\Delta) = 1 + a_1 x + a_2 x^2 + \dots. \]
        For $n=n_0 + g$ with $n_0 \geq 0$ we know from Proposition \ref{p: F_p case v-minimal higher genus} that
        \begin{align*}
            a_n &\equiv \det(\Delta_{I \times I}) \prod_{k=1}^{n_0} \Delta_{k,k} \mod \pi^{v(I_{S_{n_0}},\Delta) + 1}.
        \end{align*}
        Thus, by Corollary \ref{c: F_p valuation of p-bounded for higher genus} we have
        \begin{align*}
            v(a_n) &= \sum_{k=1}^{n_0} y(pk - k).
        \end{align*}
        It follows that the $\pi$-adic Newton polygon on $\zeta_{A,\pi}(x,y)$
        has $g$ slope zero segments followed by the slopes $y(p-1),y(2p-2)$, etc. 
        \begin{remark}
            Let $c_n$ be the coefficients of the zeta function of $\F_p[\theta]$ as in the proof of Theorem \ref{t: Wan's theorem}. Using Lemma \ref{p: F_p congruence for higher genus matrix}
            and the proof of Corollary \ref{c: F_p det of hol part is a unit}
            we see that $a_n \equiv \gamma c_{n_0} \mod \pi^{v(c_{n_0})+1}$,
            where $\gamma$ is the leading coefficient of the Weil zeta function
            $\zeta(\Spec(A),x)$. That is, we can determine the first $\pi$-adic coefficient of each $a_n$ in terms of the coefficients of the affine line zeta function and the Weil zeta function.
        \end{remark}
        
    \section{The affine line}\label{s: affine line 1}
    In this section we prove Theorem \ref{t:Riemann-hypothesis}
    for $A=\F_q[\theta]$. We continue with the notation in \S \ref{s: tau crystals and the trace formula}. 
    The point $\infty$ has degree one, so that $q=r$ and $\F_r[\theta]=\F_q[\theta]$.  
    We take our uniformizer at $\infty$ to be $\pi=\theta^{-1}$. 
    Since $\F_r[\theta]$ has class number one we have $\bfV=\bfR=\F_r\llbracket \pi \rrbracket$ and $\bfV^\circ=\bfR^\circ=\F_p\llbracket \pi \rrbracket$. Throughout
    this section we fix $y \in \Z_p$.

    \subsection{Digit Sequences}
	
		As we will see, the Newton slopes of $\zeta_{A,\pi}(-,y)$ are closely connected to the $q$-adic digits of $y$. We collect here some basic definitions concerning digit sums and their growth.
		
		\subsubsection{Basic Definitions}
		
			Let $y \in \Z_p$. Let us decompose $y$ as follows:
			\begin{equation}\label{eq:y-decomposition}
				y	=	\sum_{i = 1}^b p^{i-1} y_i = \sum_{i = 1}^b p^{i-1} \sum_{j = 0}^\infty y_{i,j} q^j,
			\end{equation}
			where $0 \leq y_{i,j} < p$.
			
			\begin{definition}
				We say that $y$ is \emph{$q$-full} if each $y_i$ is not a positive integer.
			\end{definition}

            \begin{convention}
                We assume $y$ is $q$-full unless explicitly stated otherwise. 
            \end{convention}
			
			The condition that $y$ is $q$-full guarantees that the partial sums of the $q$-adic digits of each $y_i$ tend to infinity. This allows us to introduce the following important sequence:
			
			\begin{definition}
				Let $n$ be a positive integer. We set $d_i(n) = p^{i-1} q^w$, where $w$ is the unique positive integer such that
				\begin{equation*}
					\sum_{j = 0}^{w - 1} y_{i,j} < n \leq \sum_{j = 0}^{w} y_{i,j}.
				\end{equation*}				
			\end{definition}
			
			Note that each $d_i$ is a non-decreasing sequence of powers of $p$ that tend to infinity. Moreover, we have
			\begin{equation*}
				y_i	=	\sum_{n = 1}^\infty d_i(n).
			\end{equation*}
			
			\begin{definition}\label{d:y-i-partial-sum}
				Let $m$ be an integer. We define the partial sums
				\begin{equation*}
					y_i(m)	=	\sum_{n = 1}^m d_i(n).
				\end{equation*}
				Note that if $m \leq 0$, the sum for $y_i(m)$ is empty and we have $y_i(m) = 0$. In addition, for any $k \in \Z$ we define
				\begin{equation*}
					y_i(m,k) := y_i(m+k) - y_i(m).
				\end{equation*}
			\end{definition}
            Finally, we make the following evident but important observation:
            \begin{equation}\label{eq: facts about digit differences 3}
                y_i(m,k_1 + k_2)\geq y_i(m,k_1) + y_i(m,k_2) \text{ if $k_1,k_2\geq 0$ or $k_1,k_2\leq 0$}.
            \end{equation}
			
		\subsubsection{Growth of Digit Sequences}
		
			The sequences $d_i(n)$ grow very rapidly with $n$. The next few lemmas quantify this rapid growth.
			
			\begin{lemma}\label{l:digit-sequences-growth}
				Let $n$ be a positive integer. Let $m$ and $k$ be non-negative integers. Then
				\begin{enumerate}
					\item	$d_i(n+p-1) \geq  q d_i(n)$.\label{i:di-growth}
					\item	$y_i(m+p-1,k) \geq qy_i(m,k)$.\label{i:yi-k-groth}
					\item	$y_i(m+p-1) > qy_i(m)$.\label{i:yi-growth}
				\end{enumerate}
			\end{lemma}
			\begin{proof}
				The first claim is clear from the definition of $d_i$. If $k$ is non-negative, then
				\begin{equation*}
					y_i(m+p-1,k)	=	\sum_{n = m+p}^{m+p-1+k} d_i(n) = \sum_{n = m+1}^{m+k} d_i(n+p-1) \geq  q\sum_{n = m+1}^{m+k} d_i(m) = q y_i(m,k).
				\end{equation*}
				For the final claim, note that if $m=0$ we have $y_i(m)=0$, which immediately gives the inequality. For $m>0$ we have
				\begin{equation*}
					y_i(m+p-1) = y_i(p-1,m) + y_i(p-1)\geq  q y_i(0,m) + y_i(p-1) > q y_i(m).
				\end{equation*}
			\end{proof}
			
			\begin{lemma}\label{lemma: key growth bound 2}
				Assume $b>1$ Let $n$ be a positive integer. Then
				\begin{equation*}
					p d_i(n)	>	\sum_{k = 0}^\infty y_i(n-k(p-1)).
				\end{equation*}
			\end{lemma}
			\begin{proof}
				Write $d_i(n) = p^{i-1} q^w$. For each $0 \leq m \leq w-k$, the term $p^{i-1}q^m$ appears in the partial sum for $y_i(n-k(p-1))$ at most $p-1$ times. Thus:
				\begin{align*}
					\sum_{k = 0}^\infty y_i(n-k(p-1))	&\leq	\sum_{k = 0}^w \sum_{m = 0}^{w-k} p^{i-1}(p-1)q^m	\\
						&=	p^{i-1} \frac{p-1}{q-1} \sum_{k = 0}^w (q^{w-k+1}-1)	\\
						&<	p^{i-1} \frac{p-1}{(q-1)^2} q^{w+2}	\\
						&\leq	p p^{i-1} q^w = p d_i(n),
				\end{align*}			
				where we have used the fact that $\frac{p-1}{(q-1)^2} < \frac{p}{q^2}$ for $b > 1$ .	
			\end{proof}

    \subsection{The initial setup} \label{ss: the initial setup for the affine line}
    Recall that in \S \ref{ss:zeta} we defined a character $\rho_{\F_r[\theta],\pi}:\pi_1(\A_{\F_r}^1) \to \bfV^\times$. 
    By Proposition \ref{p:Katz-correspondence} we know $\rho_{\F_r[\theta],\pi}$ corresponds to a rank one $\tau$-module $\mathscr{G}_{\F_r[\theta]}$ over the ring 
        \begin{equation*}
        \bfV \langle \theta \rangle  =   \bfV \widehat{\otimes}_{\F_r} \F_r[\theta].
    \end{equation*}
    More generally, let $\mathscr{G}_{\F_r[\theta]}^{\otimes y}$ be the $\tau$-module associated to $\rho_{\F_r[\theta],\pi}^{\otimes y}$ .
    
    \begin{proposition}\label{p:A1-Frob} (Taguchi-Wan \cite{Taguchi-Wan})
        The element $1-\pi\theta$ is a Frobenius matrix for $\mathscr{G}_{\F_r[\theta]}$. In particular, $(1 - \pi\theta)^y$ is a Frobenius matrix of $\mathscr{G}_{\F_r[\theta]}^{\otimes y}$ and $\mathscr{G}_{\F_r[\theta]}^{\otimes y}$ is overconvergent.
    \end{proposition}
    \begin{proof}
        This follows by looking at the pullback of the $\tau$-module
        to each closed point. 
    \end{proof}

    The decomposition \eqref{eq:y-decomposition} allows us to write
    \begin{equation*}
        (1-\pi\theta)^y=\prod_{i=1}^{b} (1-\pi\theta)^{y_ip^{i-1}}=\prod_{i=1}^{b} \Big[(1-\pi^{p^{i-1}}\theta)^{y_i})\Big]^{F^{i-1}}.
    \end{equation*}
    For $1 \leq i \leq b$ we define the following series:
    \begin{align}
        \beta_i&:=\prod_{j = 0}^\infty (1-\pi^{q^jp^{i-1}}\theta)^{y_{i,j}}. \label{eq: product description of beta}
    \end{align}
    We easily compute that $(1-\pi\theta)^y$ is $\tau$-equivalent to $\beta_1 \dots \beta_{b}^{F^{b-1}}$ over $\bfV\langle \theta \rangle$. Set $\mathscr{F}_{\F_r[\theta]}= \Res(\mathscr{G}_{\F_r[\theta]}^{\otimes y})$.
    By Proposition \ref{prop: shape of restriction of scalars Frobenius structure} we see that the matrix 
    \begin{align*}
        E_{\mathscr{F}_{\F_r[\theta]}} &:= \begin{bmatrix} 0 & 0 & \dots & 0 &  \beta_{1} \\ 
	\beta_2  & 0 & \dots & 0 & 0 \\
	0 & \ddots & \ddots & \vdots & \vdots \\
	\vdots & \ddots & \ddots & 0 & 0 \\
	0 & \dots & 0 &\beta_{b}& 0
	\end{bmatrix}
    \end{align*}
    is a Frobenius matrix for $\mathscr{F}_{\F_r[\theta]}$. 
    Following \S \ref{ss: the trace formula} we define the spaces $V_{\mathscr{F}_{\F_r[\theta]}}$ and $M_{\mathscr{F}_{\F_r[\theta]}}$, as well as the operator $\Theta_{\mathscr{F}_{\F_r[\theta]}}$. In our specific setup, we have a basis $\{e_1,\dots,e_{b}\}$ of $V_{\mathscr{F}_{\F_r[\theta]}}$ so that 
    \begin{align*}
        M_{\mathscr{F}_{\F_r[\theta]}} &= \bigoplus_{i=1}^{b} \theta \bfV\langle \theta \rangle e_i,\\
        \Theta_{\mathscr{F}_{\F_r[\theta]}} &= U_p \circ E_{\mathscr{F}_{\F_r[\theta]}}.
    \end{align*}
    Then from Proposition \ref{p: zeta function via L-functions} and Corollary \ref{cor: a-th root trick} we have 
    \begin{align} \label{eq: matrix calculation of Newton polygon}
    \NP_\pi\Big (\zeta_{\F_r[\theta],\pi}(x^b,y)^b\Big ) &= \NP_\pi\Big (\det_{\bfV^\circ}\Big (1-xU_p \circ E_{\mathscr{F}_{\F_r[\theta]}}| M_{\mathscr{F}_{\F_r[\theta]}} \Big)\Big).  
    \end{align}
    Let $\zeta_1,\dots,\zeta_b$ be a basis of $\F_r$ over $\F_p$. Define the following set:
    \begin{align}
        B^\circ    :=   \Big \{\zeta_j^{p^{-i}}\theta^{k}e_i\Big |1 \leq i,j \leq b \text{ and } k \geq 1\Big \}.\label{eq: basis for affine line case, first version}
    \end{align}
    Note that $B^\circ$ is a pole order basis of $M_{\mathscr{F}_{\F_r[\theta]}}$ over $\bfV^\circ$. We view $B^\circ$ as being indexed by $J_b$: The element $(i,j,m)\in J_b$ corresponds to $\zeta_j^{p^{-i}}\theta^{m}e_i$. Let $\Psi$ be the $J_b \times J_b$-matrix for the operator $\Theta_{\mathscr{F}_{\F_r[\theta]}}$ in terms of $B^\circ$. By combining \eqref{eq: matrix calculation of Newton polygon} with
    Proposition \ref{prop: Fredholm can be computed using pole-ordered bases} we
    have:
    \begin{align} \label{eq: matrix calculation of NP with basis}
        \NP_\pi\Big (\zeta_{\F_r[\theta],\pi}(x^b,y)^b\Big ) &= \NP_\pi\Big (\det(1-x\Psi) \Big).
    \end{align}
    
    \subsection{Statement of the main technical result}

    In this section we state the main technical result of \S \ref{s: affine line 1} (Theorem \ref{theorem: unique minimal term}). This result is the heart of the proof of Theorem \ref{t:Riemann-hypothesis} for $A=\F_r[\theta]$. We begin by observing that the numbers $y_i(m)$ of Definition \ref{d:y-i-partial-sum} give the $\pi$-adic valuations of the coefficients of $\beta_i$. Write 
    \begin{align*}
    	\beta_i=\sum_{-\infty }^\infty a_{i,n}\theta^n,
    \end{align*}
    where we have $a_{i,n}=0$ for $n<0$.
    \begin{lemma}\label{lemma: valuations of coefficients of Yi}
    	We have $v(a_{i,n})=y_i(n)$ for all $n\geq 1$. 
    \end{lemma}
    \begin{proof}
    	This follows by expanding the product \eqref{eq: product description of beta} defining $\beta_i$ and keeping track of the unique term of minimal
    	valuation contributing to each coefficient.
    \end{proof}

    \begin{proposition}\label{p: matrix entries}
        For $k_1=(i_1,j_1,m_1)$ and $k_2=(i_2,j_2,m_2)$ in $J_b$ we have
\begin{align}\label{eq: matrix entries}
	\Psi_{k_1,k_2} &= \begin{cases}
	a_{\bfi{k_1},p|k_1|-|k_2|} & \bfi{k_2} \equiv \bfi{k_1}-1\mod b\text{ and }j_1=j_2 \\
	0 & \textrm{otherwise}.
	\end{cases}
\end{align}
    \end{proposition}
    \begin{proof}
        The proposition follows by observing:
        \begin{align*}
            U_p \circ E_{\mathscr{F}_{\F_r[\theta]}}(\zeta_j^{p^{-i+1}}\theta^{m}e_{i-1}) &=U_p \Big(\beta_i\zeta_j^{p^{-i+1}}\theta^{m} \Big)e_{i} \\
            &= U_p\Big( \sum_{n=0}^\infty a_{i,n} \zeta_j^{p^{-i+1}}\theta^{n+m}  \Big ) e_{i} \\
            &= \sum_{n=1}^\infty a_{i,pn-m}\zeta_j^{p^{-i}}\theta^{n}  e_{i}.
        \end{align*}
    \end{proof}

    Let $\Psi^\circ$ be the $J_1 \times J_1$ matrix defined by 
    \begin{align}\label{eq: matrix entries one copy}
	\Psi^\circ_{k_1,k_2} &= \begin{cases}
	a_{\bfi{k_1},p|k_1|-|k_2|} & \bfi{k_2} \equiv \bfi{k_1} -1\mod b \\
	0 & \textrm{otherwise},
	\end{cases}
\end{align}
    for $k_1,k_2 \in J_1$.
    An immediate consequence of Proposition \ref{p: matrix entries} is that $\Psi$ can
    be broken up into a diagonal $b \times b$ block matrix as follows:
    \begin{align*}
        \Psi &= \begin{bmatrix}
            \Psi^\circ  & \dots & 0 \\
            \vdots & \ddots & \vdots  \\
            0 & \dots & \Psi^\circ
        \end{bmatrix}.
    \end{align*}
    In particular, we have $\det(I - s\Psi) = \det(I-s\Psi^\circ)^b$. Combining this with
    \eqref{eq: matrix calculation of NP with basis} we have
    \begin{align}\label{eq: matrix calculation of NP single copy}
        \NP_\pi\Big(\zeta_{\F_r[\theta],\pi }(x^b,y)\Big) &= \NP_\pi\Big(\det(1-x\Psi^\circ)\Big).
    \end{align}
    \begin{lemma}\label{proposition: valuations of matrix affine line}
        For $k_1,k_2 \in J_1$ we have
        \begin{align*}
            v(\Psi^\circ _{k_1,k_2}) &= \begin{cases}
            y_{\bfi{k_1}}(p|k_1| - |k_2|) & \bfi{k_2}\equiv \bfi{k_1} -1 \mod b ~\textrm{and}~p|k_1|-|k_2|\geq 0 \\
            \infty & \textrm{otherwise}
            \end{cases}.
            \end{align*}
    \end{lemma}
    \begin{proof}
        This follows from Lemma \ref{lemma: valuations of coefficients of Yi} and the definition of $\Psi^\circ$.
    \end{proof}

    \begin{corollary}\label{lemma: the nonzero coefficients}
    	Let $(S,\sigma)$ be an enriched permutation of $J_1$. Then $v(\sigma,\Psi^\circ)$ is finite if and only if $\sigma$ is rotational and $p$-bounded. 
    \end{corollary}
    
    \noindent Our main technical result of \S \ref{s: affine line 1} is the following theorem. Its proof is the content of \S \ref{ss: minimal permutations are lexicographical}, \S \ref{ss: more preliminary defs and obs}, and \S \ref{ss: finishing proof of affine line}.
    \begin{theorem}\label{theorem: unique minimal term}
    	Assume $y$ is $q$-full. There is a unique minimal rotational enriched permutation $\Sigma_n$ of $\Psi^\circ$ of size $n$.
    	Furthermore, if we define $\nu_n= v(\Sigma_n,\Psi^\circ) - v(\Sigma_{n-1},\Psi^\circ)$ then for all $n \geq 1$ we have 
    	\[ \nu_{n+1} > \nu_n. \]
    \end{theorem}
    Before delving into the proof of Theorem \ref{theorem: unique minimal term}, let us first spell out clearly why this theorem is relevant to studying slopes of zeta functions. From \eqref{eq: matrix calculation of NP single copy} we are reduced to studying the $\pi$-adic valuations of the coefficients of $\det(1-x\Psi^\circ)$. From 
    \eqref{eq: Fredholm determinant 1} and \eqref{eq: Fredholm determinant 2}, each coefficient is a sum over terms of the form $\prod_{i \in S} \Psi^{\circ}_{i,\sigma(i)}$, where $(S,\sigma)$
    is an enriched permutation of $J_1$. Theorem \ref{theorem: unique minimal term} says
    that each sum has \emph{exactly one} term with minimal valuation. This allows us to determine the valuation of the coefficients of $\det(1-x\Psi^\circ)$.

    \subsection{\texorpdfstring{$R$}{R}-minimal permutations}\label{ss: minimal permutations are lexicographical}
    In this subsection we introduce the notation of $R$-minimal enriched permutations of $J_h$, which is closely related to $v$-minimal enriched permutations. For enriched permutations of $J_1$, we will see that $R$-minimal and $v$-minimal are the same notion. Our main result implies that $R$-minimal enriched permutations of $J_1$ are lexicographical and $p$-bounded. We will prove results for enriched permutations of $J_h$, despite $\Psi^\circ$ being a $J_1\times J_1$-matrix. 
    The more general result will be used when considering more general curves in \S \ref{s: general ordinary curves}.
    
    \subsubsection{Some definitions}

    This next definition is closely related to the valuations
    of the entries of $M$. 
    \begin{definition}
        Let $k_1, k_2 \in J_h$ with $\bfi{k_2}= \bfi{k_1}-1$. We define 
        \begin{equation}
             R(k_1,k_2) := y_{\bfi{k_1}}(p|k_1|-|k_2|) = \begin{cases}
                v(\Psi^\circ_{\iota(k_1),\iota(k_2)}) & |k_2| \leq p|k_1| \\ 
                0 & \text{ otherwise}
            \end{cases}.
        \end{equation}
        Let $(S,\sigma)$ be a rotational permutation
        of $J_h$. We define the \emph{$R$-value} of $\sigma$ to be
        \begin{align*}
            R(\sigma) &= \sum_{k \in S} R(k,\sigma(k)).
        \end{align*}
    \end{definition}

    Observe that if $\sigma$ and $\tau$ are disjoint rotational permutations of $J_h$, then $R(\sigma \tau) = R(\sigma) + R(\tau)$. In particular, we observe that if $\sigma$ is decomposable then $R(\sigma) = \sum R(\sigma^{(i)})$.
    
	\begin{definition}
			Let $\sigma$ be a permutation of $J_h$. We say that $\sigma$ is \emph{$R$-minimal of size $n$} if $\bfs{\sigma}=n$ and for any rotational permutation $\sigma'$ of $J_h$ with $\bfs{\sigma'} = n$, we have
			\[R(\sigma) \leq R(\sigma').\]
	\end{definition}

    The advantage of $R$-values over $v$-values is that they allow us to treat enriched permutations of $J_h$ with finite and infinite valuation on an equal footing. The relationship between $R$-values and $v$-values is given in the following lemma, which follows immediately from Lemma \ref{proposition: valuations of matrix affine line}.
    
\begin{lemma} \label{l: compare R and valuation affine line}
        Let $\sigma$ be a rotational enriched permutation of $J_h^{(j)}$. Then
        \begin{align*}
            v(\iota_*\sigma,\Psi^\circ) &\geq R(\sigma),
        \end{align*}
        with equality if and only if $\sigma$ is $p$-bounded. 
    \end{lemma} 

    The following proposition greatly restricts the type
    of permutation that can be $R$-minimal.
    \begin{proposition}
        \label{p: minimal permutations have all the nice properties} Let $\sigma$ be a rotational enriched permutation of $J_h$. There exists an enriched permutation $\sigma'$  with $\bfs{\sigma'} = \bfs{\sigma}$ and $R(\sigma') \leq R(\sigma)$ satisfying:
        \begin{enumerate}
            \item $\sigma'$ is lexicographical
            \item $\sigma'$ is decomposable.
            \item $\sigma'$ is $p$-bounded.
        \end{enumerate}
        Furthermore, if $\sigma$ is decomposable but either not lexicographical or not $p$-bounded, then this inequality is strict.
    \end{proposition}
    \begin{corollary}
        \label{c: minimal permutations are lexicographical}
        Any $v$-minimal enriched permutation of $\Psi^\circ$ is
        lexicographical and $p$-bounded.
    \end{corollary}
    \noindent The proof of Proposition \ref{p: minimal permutations have all the nice properties} is broken into several steps. 

    \subsubsection{Some elementary bounds}

     \begin{lemma}\label{lemma: inequality about switching coordinates}
        Let $i \in \Z/b\Z$. Let $k_1,k_2,l_1,l_2 \in J_h$ with $\bfi{k_1} = \bfi{l_1} = i$. Assume that $|k_1| \leq |l_1|$ and $|k_2| \geq |l_2|$.  Then
        \begin{equation*}
            R(k_1,k_2) + R(l_1,l_2)	\geq	R(k_1,l_2) + R(l_1,k_2).
        \end{equation*}
        Furthermore, this inequality is strict of $|k_1|<|l_1|$.
    \end{lemma}
    \begin{proof}
        We compute directly:
        \begin{align*}
            R(k_1,l_2) - R(k_1,k_2)	&=	y_i(p|k_1|-|l_2|) - y_i(p|k_1|-|k_2|)	\\
                &=	y_i(p|k_1|-|k_2|,|k_2|-|l_2|)	\\
                &\leq	y_i(p|l_1|-|k_2|,|k_2|-|l_2|)	\\
                &=	y_i(p|l_1|-|l_2|) - y_i(p|l_1|-|k_2|)	\\
                &=	R(l_1,l_2) - R(l_1,k_2).
        \end{align*}
        From Lemma \ref{l:digit-sequences-growth}\ref{i:yi-k-groth}, we see that the inequality is strict if $|k_1| < |l_1|$.
    \end{proof}
    
    \subsubsection{The proof of Proposition \ref{p: minimal permutations have all the nice properties}}

    \begin{lemma} \label{l: reduce to lexicographical}
        Let $(S,\sigma)$ be an rotational enriched permutation of $J_h$. There exists a rotational enriched permutation $\sigma'$ of $S$ that is lexicographical such that $R(\sigma') \leq  R(\sigma)$. Furthermore, if $\sigma$ is decomposable and not lexicographical, then this inequality is strict.
    \end{lemma}
    \begin{proof}
        We proceed by induction on $\bfs{\sigma}$. The case $\bfs{\sigma} = 1$ is clear, so assume that $\bfs{\sigma} > 1$. For each $i \in \Z/b\Z$, choose $k_i \in S$ with $\bfi{k_i} = i$ that minimizes $|k_i|$, i.e.
        \begin{equation*}
            |k_i|	=	\min_{\substack{k \in S\\\bfi{k}=i}}	|k|.
        \end{equation*}
        
        Let $S_\mathrm{min} = \{k_1,\dots,k_b\}$. Observe that $\sigma_\mathrm{min} = (k_b \cdots k_1)$ is a rotational $b$-cycle of $S_\mathrm{min}$. For each $i$, let $l_i = \sigma^{-1}(k_{i-1})$. We define a new permutation $\sigma_0$ of $S$ as follows:
        \begin{equation*}
            \sigma_0(k)	=	\begin{cases}
                    k_{i-1}	&	k = k_i	\\
                    \sigma(k_i)	&	k = l_i	\\
                    \sigma(k)	&	\text{otherwise}
                \end{cases}.
        \end{equation*}
        Note that $\sigma_0 = \sigma_\mathrm{min} \tau$ where $\tau$ is a permutation of $S - S_\mathrm{min}$. By the choice of $k_i$, we have $\sigma_\mathrm{min} \leq \tau$. Then from the induction hypothesis we see that there exists a lexicographical permutation $(S-S',\tau')$ such that $R(\tau') \leq R(\tau)$. Set
        \begin{equation*}
            \sigma'	=	\sigma_\mathrm{min} \tau'.
        \end{equation*}
        
        Observe that $R(\sigma') \leq R(\sigma_0)$. We claim that $R(\sigma_0) \leq R(\sigma)$. Suppose that $\sigma_0 \neq \sigma$ so that $k_i \neq l_i$ for some $i$. By the minimality condition on the $k_i$, we have $|k_i| \leq |l_i|$ and
        \begin{equation*}
            |\sigma(l_i)|	=	|k_{i-1}|	\leq	|\sigma(k_i)|.
        \end{equation*}
        Applying Lemma \ref{lemma: inequality about switching coordinates}, we see that
        \begin{equation*}
            R(k_i,\sigma(k_i)) + R(l_i,\sigma(l_i))	\geq	R(k_i,\sigma(l_i)) + R(l_i,\sigma(k_i))	=	R(k_i,\sigma_0(k_i)) + R(l_i,\sigma_0(l_i)),
        \end{equation*}
        and that the inequality is strict if $|k_i| < |l_i|$. For any $k$ which is distinct from the $k_1,\dots,k_b$ and $l_1,\dots,l_b$, we know that $\sigma(k) = \sigma_0(k)$. Thus, we see that
        \begin{equation}\label{eq:sigma-0-inequality}
            R(\sigma_0)	=	\sum_{k \in S} R(k,\sigma_0(k))	\leq	\sum_{k \in S} R(k,\sigma(k))	=	R(\sigma).
        \end{equation}
        
        To prove the final claim, note that if $\sigma$ is decomposable then we may reduce to the case $\sigma \in J_h^{(j)}$ for some $j$. If $\sigma$ is not lexicographical then $\sigma \neq \sigma_0$, so $k_i \neq l_i$ for at least one value of $i$. Since $\sigma \in J_h^{(j)}$, this implies $|k_i| < |l_i|$ and the inequality \eqref{eq:sigma-0-inequality} is strict.
    \end{proof}

    \begin{lemma} \label{l: reduce from lexicographical to decomposbile and lexicographical}
        Let $(S,\sigma)$ be a lexicographical enriched permutation. There exists $(S',\sigma')$
        that is decomposable and lexicographical such that $\bfs{\sigma'}=\bfs{\sigma}$
        and $R(\sigma')=R(\sigma)$.
    \end{lemma}
    \begin{proof}
        Write $\sigma=\sigma_1\dots\sigma_n$, where each $\sigma_t$ is a $b$-cycle
        and $\sigma_{t+1}\geq \sigma_t$. Write $\sigma_t=(k_{t,b}~\dots ~k_{t,1})$,
        where $\bfi{k_{t,i}}=i$. Note that
        $|k_{t+1,i}| \geq |k_{t,i}|$ for all $t$ and $i$.
        Let $l_{t,i}=(i,\overline{t},|k_{t,i}|)$, where we take $\overline{t}$ to be the unique integer between $1$ and $h$ congruent to $t$. We define $\tau_t=(l_{t,b}~\dots~l_{t,1})$.
        Note that $R(\tau_t)=R(\sigma_t)$ and $\tau_{t+1}\geq \tau_t$. Thus, the lemma
        will be proven if we can verify that the $\tau_t$'s are all disjoint. By definition, we know that $\tau_t$ and $\tau_s$ are disjoint if $t\not\equiv s \mod h$. Thus, it suffices to prove $\tau_{t+h}>\tau_{t}$ for $1\leq t\leq n-h$.
        This reduces to proving $|k_{t+h,i}|> |k_{t,i}|$. Assume that $|k_{t+h,i}|=|k_{t,i}|$ for some $i$. The $|k_{t,i}|$ are nondecreasing with $t$, so for $0\leq c \leq h$ we see that $k_{t+c,i}$ is of the form $(i,j_c, |k_{t,i}|)$.
        However, there are $h$ possibilities for $j_c$, so by the pigeonhole principle we have $j_{c_1}=j_{c_2}$ for some $0\leq c_1<c_2\leq h$. Thus $k_{t+c_1,i}=k_{t+c_2,i}$. However, since the $\sigma_t$'s are disjoint, the $k_{t+c,i}$'s must be distinct, which gives a contradiction.
    \end{proof}

    \begin{proposition}\label{l: universal p-bounded less than a b-cycle}
        
        Let $C_h^{(j)}$ denote the set of rotational $b$-cycles of $J_h^{(j)}$. There exists a unique function $\bfp: C_h^{(j)} \to C_h^{(j)}$ satisfying the following properties:
        \begin{enumerate}
            \item	Let $\sigma \in C_h^{(j)}$. Then $\bfp(\sigma)$ is $p$-bounded and $R(\bfp(\sigma))\leq R(\sigma)$ with a strict inequality if $\sigma$ is not $p$-bounded.
            \item If $\sigma$ is $p$-bounded then $\bfp(\sigma)=\sigma$.
            \item Let $\sigma,\tau \in C_h^{(j)}$. If $\tau \leq \sigma$ then $\bfp(\tau) \leq \bfp(\sigma)$. If in addition $\sigma$ and $\tau$ are disjoint, then $\bfp(\sigma)$ and $\bfp(\tau)$ are also disjoint.
        \end{enumerate}
    \end{proposition}
    \begin{proof}
        Since $\sigma \in C_h^{(j)}$ is $p$-bounded if and only if $\iota_* \sigma$ is $p$-bounded, it suffices to prove the claim for $h = 1$. Let $(m_1,\dots,m_b)$ denote the coordinates of $\sigma$. Choose $c$ which minimizes $m_{c}$. We define $\bfp(\sigma) = (n_1, \dots, n_b)$, where the coordinates are defined as follows: We set $n_{c} = m_{c}$, and for $0 < j < b$ we define $n_{c-j}$ recursively as
        \begin{equation*}
            n_{c-j}	=	\min\{pn_{c-j+1},m_{c-j+1}\}.
        \end{equation*}
        It is clear that $\bfp(\sigma) \leq \sigma$ and that $n_{c-j} \leq p n_{c-j+1}$ for $1 \leq j < b$. To see that $n_{c} \leq p n_{c+1}$, note that we must have $n_{c+1} = p^j m_{c+1+j}$ for some $0 \leq j < b$. The claim follows since $n_{c} = m_{c}$ is the minimal coordinate of $\sigma$.
        Thus, $\bfp(\sigma)$ is $p$-bounded, which is the second statement.
        
        To prove that $R(\bfp(\sigma)) \leq R(\sigma)$, it suffices to show that for each $i$ we have
        \begin{equation*}
            y_i(pn_i-n_{i-1}) \leq y_i(pm_i-m_{i-1}).
        \end{equation*}
        If $n_{i-1} = m_{i-1}$, then the claim is immediate since $n_i \leq m_i$. Otherwise, $n_{i-1} = pn_i$ and the claim follows since $y_i(pn_i - n_{i-1}) = 0$. Finally, suppose that $\sigma$ is not $p$-bounded. Then there exists some $0 < j < b$ such that $n_{c+j} < m_{c+j+1}$. Indeed, if this were not the case then we would have $\bfp(\sigma)=\sigma$, which would imply that $\sigma$ is $p$-bounded. Set $j$ to be the minimal value for which this holds. In particular, $n_{c+j-1} = m_{c+j-1}$ and so
        \begin{equation*}
            y_i(pn_{c+j}-n_{c+j-1}) < y_i(pm_{c+j}-m_{c+j-1}).
        \end{equation*}
        This proves the first statement.

        To prove the last statement, let $\tau$ be a $b$-cycle with coordinates $\tau = (r_1, \dots ,r_b)$ and suppose that $\tau \leq \sigma$. This means $r_{i} \leq m_{i}$ for all $i$. Write $\bfp(\tau)=(s_1,\dots,s_b)$. As $s_c \leq r_c$
        and $m_c=n_c$ we have $s_c\leq n_c$. 
        Then we have
        \[ s_{c-1} \leq \min( r_{c-1}, ps_c) \leq \min(m_{c-1},pn_c) = n_{c-1}.  \]
        Repeating this shows that $s_i\leq n_i$ for all $i$, so that $\bfp(\tau)\leq \bfp(\sigma)$. Finally, if $\tau$ and $\sigma$ are disjoint we have $r_i<m_i$ for all $i$. Repeating the same argument gives $s_i<n_i$ for all $i$, so that $\bfp(\tau)$ and $\bfp(\sigma)$ are disjoint. 
        
    \end{proof}

    \begin{proof}
        (Of Proposition \ref{p: minimal permutations have all the nice properties})
        Let $\sigma$ be a rotational enriched permutation of $J_h$. By Lemma \ref{l: reduce to lexicographical} and Lemma \ref{l: reduce from lexicographical to decomposbile and lexicographical} there exists a $\sigma'$ that is decomposable and lexicographical, such that $\bfs{\sigma'}=\bfs{\sigma}$ and $R(\sigma')\leq R(\sigma)$ (this last inequality is strict of $d=1$). To prove the proposition, apply $\bfp$ to every $b$-cycle in $\sigma'$ and use Proposition \ref{l: universal p-bounded less than a b-cycle}.
    \end{proof}

    \subsection{\texorpdfstring{$R$}{R}-minimal permutations of \texorpdfstring{$J_1$}{J1}} \label{ss: more preliminary defs and obs}
    \subsubsection{\texorpdfstring{$b$}{b}-cycles}

        We return to studying enriched permutations of $J_1$ and the proof of Theorem \ref{theorem: unique minimal term}. In light of Corollary \ref{c: minimal permutations are lexicographical}, we begin by discussion rotational $b$-cycles. We begin with an important observation:

        \begin{lemma}\label{lemma: a-cycles have different valuations}
            Let $\sigma$ and $\tau$ be two $p$-bounded rotational $b$-cycles of $J_1$. If $R(\sigma) = R(\tau)$, then $\sigma = \tau$.
        \end{lemma}
        \begin{proof}
            Let $\sigma = (m_1,\dots,m_b)$ and $\tau = (n_1,\dots,n_b)$. By comparing $p$-adic digits, we know that $y_i(pm_i-m_{i-1}) = y_i(pn_i-n_{i-1})$ for all $i$. Since both $\sigma$ and $\tau$ are $p$-bounded, we must have $pm_i-m_{i-1} = pn_i-n_{i-1}$ for all $i$. Choose $c$ such that $\alpha = m_c-n_c$ has maximal absolute value. By induction, we see that $m_{c-j}-n_{c-j} = p^j \alpha$ for all $j$. Setting $j = b$ we obtain $\alpha = p^j \alpha$, and so $\alpha = 0$.
        \end{proof}

        \begin{corollary}\label{c:unique-minimal-b-cycle}
            Let $S$ be a set of rotational $b$-cycles of $J_1$ which is stable under $\bfp$ (i.e. $\bfp(S) \subset S$). There is a unique element of $S$ that
            has minimal $R$-value.
        \end{corollary}
        \begin{proof}
            From Proposition \ref{l: universal p-bounded less than a b-cycle} and the fact that $S$ is stable under $\bfp$ we know that any element of $S$ with minimal $R$-value is $p$-bounded. The corollary then follows from Lemma \ref{lemma: a-cycles have different valuations}.
        \end{proof}


 

    We now introduce some special sets of rotational $b$-cycles whose minimal elements will play an important role in the proof of Theorem \ref{theorem: unique minimal term}. Let $n \geq 0$ and let $\vec{m}$ be a $\Z/b\Z$-tuple of integers with $\vec{m} \geq (n,\dots,n)$. Define
    \begin{align*}
		A(\vec{m}) &= \{\sigma~|~\text{$\sigma$ is a rotational $b$-cycle with $\sigma \leq \vec{m}$}\} \\
		S^{\geq n}&=\{\sigma~|~\text{$\sigma$ is a rotational $b$-cycle with $\sigma \geq  (n,\dots,n)$}\}\\
		S^{\geq n}(\vec{m}) &= S^{\geq n} \cap A(\vec{m})
	\end{align*}
    Note that $S^{\geq n}(\vec{m})$ is empty if $\vec{m}\not\geq (n, \dots, n)$. All of these sets are stable under $\bfp$, which allows us to make the following definition:

    \begin{definition}
        With $n,\vec{m}$ as above, we let $\sigma_n^*(\vec{m})$ denote the unique $R$-minimal $b$-cycle of $S^{\geq n}(\vec{m})$.
    \end{definition}
    

    \subsubsection{Main proposition on lexicographical permutations} \label{ss: proof that there is a unique minimal element}
	
	We now turn our attention to lexicographical permutations. Let $n \geq 1$ and let $\vec{m}$ be a $\Z/b\Z$-tuple of integers. Consider the set
	\begin{equation*}
		A_n(\vec{m}) =\{\sigma=\sigma_1\dots\sigma_n ~|~\sigma_i\in A(\vec{m}) \text{ and }~\sigma_i<\sigma_{i+1}~\text{for each }i\}.
	\end{equation*}
    Note that every element $\sigma \in A_n(\vec{m)}$ factors uniquely as $\sigma = \Sigma \sigma_n$, where $\sigma_n \in S^{\geq n}(\vec{m})$ and $\Sigma \in A_{n-1}(\sigma_n-\vec{1})$.
    The remainder of \S \ref{ss: more preliminary defs and obs} is dedicated to establishing the following proposition, which is the technical heart of the proof of Theorem \ref{theorem: unique minimal term}.
    
    \begin{proposition}\label{proposition: key technical proposition}
		Assume $\vec{m} \geq (n,\dots,n)$. There is a unique element $\Sigma_n^*(\vec{m})$ in $A_n(\vec{m})$ that has minimal $R$-value. 
		Furthermore, these permutations satisfy the recurrence
        \begin{equation}\label{eq: recurrance for minimal}
            \Sigma_n^*(\vec{m})=\Sigma_{n-1}^*(\sigma_n^*(\vec{m})-\vec{1})\sigma_n^*(\vec{m}).
        \end{equation}
	\end{proposition}

    The proof of Proposition \ref{proposition: key technical proposition} proceeds by induction on $n$, with the case $n = 1$ following from Corollary \ref{c:unique-minimal-b-cycle}. Let $\tau$ be an element of $A_n(\vec{m})$ that has minimal $R$-value. It is clear that $\tau$ is of the form $\tau = \Sigma_{n-1}^*(\tau_n-\vec{1})\tau_n$, where $\tau_n \in S^{\geq n}(\vec{m})$. We suppose for the sake of contradiction that $\tau_{n} \neq \sigma_n^*(\vec{m})$. We will construct $\Sigma \in A_{n-1}(\sigma_n^*(\vec{m})-\vec{1})$ with the property that
    \begin{equation}\label{eq: the fundamental equality}
        R(\tau) > R(\Sigma \sigma_n^*(\vec{m})).
    \end{equation}
    Since $\Sigma \sigma_n^*(\vec{m}) \in A_n(\vec{m})$ this will contradict the minimality of $\tau$. It will follow that $\tau$ 
    is equal to the right side of \eqref{eq: recurrance for minimal}, which ensures that $\tau=\Sigma_n^*(\vec{m})$ is the unique element of $A_n(\vec{m})$ with minimal $R$-value.
    
    We begin by defining $\Sigma$. Write $\Sigma^*_{n-1}(\tau_n-\vec{1}) = \tau_1\cdots\tau_{n-1}$ in lexicographical order. For $1 \leq j < n$, we set
    \begin{align*}
        \sigma_j    &:=   \min\{\sigma_n^*(\vec{m})-(n-j)\vec{1},\tau_j\}    \\
        \Sigma &:= \sigma_1 \cdots \sigma_{n-1},
    \end{align*}
    where the minimum is taken coordinate-wise.
    This definition ensures that $\sigma_j \leq \tau_j$ and that each $\sigma_j$ has coordinates ``as large as possible'' while still satisfying $\Sigma \in A_{n-1}(\sigma_n^*(\vec{m})-\vec{1})$. The inequality \eqref{eq: the fundamental equality} is equivalent to
	showing the following:
	\begin{align} \label{eq: the fundamental equality 2}
		R(\tau_n) - R(\sigma_n^*(\vec{m})) > R(\Sigma) - R(\Sigma_{n-1}^*(\tau_n-\vec{1})) = \sum_{j=1}^{n-1} \Big[ R(\sigma_j) - R(\tau_j) \Big].
	\end{align}
    The proof of \eqref{eq: the fundamental equality 2} is broken up into several steps. 

    \subsubsection{Some notation}
    Let us write each of the $b$-cycles appearing in the above definitions using coordinates:
    \begin{align*}
        \sigma_n^*(\vec{m}) &=   (c_1,\dots,c_b),    \\
        \sigma_j &=   (c_{j,1},\dots,c_{j,b}),    \\
        \tau_j &=   (k_{j,1},\dots,k_{j,b}).
    \end{align*}
    In particular, we have \[c_{j,i}=\min\{c_i-(n-j)(p-1),k_{j,i}\}.\]
    
    For the estimates below it is convenient introduce some additional notation: Let $\sigma = (m_1,\dots,m_b)$ be a $b$-cycle, let $\vec{r}$ be a $\Z/b\Z$-tuple of integers, and let $z \in \Z$. We set
    \begin{align*}
		y_i(\sigma) &:= y_i(pm_{i}-m_{i-1}), \\
		y_i(\sigma,\vec{r}) &:=y_i(pm_{i}-m_{i-1},pr_{i}-r_{i-1}) ,\\
		y_i(\sigma,z) &:= y_i(pm_{i}-m_{i-1},z),\\ 
		d_i(\sigma) &:= d_i(pm_{i}-m_{i-1}).
	\end{align*}
	These definitions are made so that whenever $\sigma$ and $\sigma+\vec{r}$ are $b$-cycles, the following holds:
	\begin{align*}
		R(\sigma)&= \sum_{i=1}^b y_i(\sigma), \\
		R(\sigma+\vec{r}) - R(\sigma) &= \sum_{i=1}^b y_i(\sigma,\vec{r}).
	\end{align*}

    \subsubsection{Bounding \texorpdfstring{$\tau_n$}{tau\_n} below}
    
    Our first goal is to show that $\tau_n \geq \sigma_n^*(\vec{m})$. Consider the difference vector r $\vec{z} = \tau_n - \sigma_n^*(\vec{m})$. Let us write $\vec{z}^+= \max\{\vec{z},\vec{0}\} = (z^+_1,\dots,z^+_b)$, and set
    \begin{equation*}
        \tau_n^+=\vec{c} + \vec{z}^+ \in S^{\geq n}(\vec{m}).
    \end{equation*}
    
	\begin{lemma} \label{claim: only need to increase indices}
		We have $R(\tau_n) \geq R(\tau_n^+)$ with equality if and only if $\tau_n=\tau_n^+$. 
	\end{lemma}
	\begin{proof}
		Define $\vec{z}^-=\vec{z}-\vec{z}^+$ and write $ \vec{z}^-=(z^-_1,\dots, z^-_b)$. Note that all the coordinates of $\vec{z}^-$ are non-positive. Set $\tau_n^-=\sigma_n^*(\vec{m})+\vec{z}^-$, which is contained in $ S^{\geq n}(\vec{m})$.  Note that $\tau_n=\tau_n^+$ if and only if the coordinates of $\vec{z}^-$ are all $0$. Since $\sigma_n^*(\vec{m})$ is the unique $R$-minimal $b$-cycle in $S^{\geq n}(\vec{m})$ we have
		\begin{align} \label{eq: minimal applied to tau-}
			R(\tau_n^-) - R(\sigma_n^*(\vec{m})) \geq 0,
		\end{align}
        and equality occurs if and only if $\tau_n=\tau_n^+$.
        We claim that for each $i$, $pz_i^+-z_{i-1}^+$ and $pz_i^--z_{i-1}^-$ must either both be non-negative or both be non-positive. If $pz_i^+-z_{i-1}^+=0$, it is both non-positive and non-negative, and there is nothing to prove. If $pz_i^+ - z_{i-1}^+$ is positive,
        then $z_i^+>0$, which implies $z_i^-=0$. As $z_{i-1}^-\leq 0$ we see that $pz_i^+-z_{i-1}^+\geq 0$. One argues similarly when $pz_i^+ - z_{i-1}^+$ is negative.
        
        Using \eqref{eq: facts about digit differences 3} we compute:
        \begin{align*}
            R(\tau_n)   &=  R(\sigma_n^*(\vec{m})+\vec{z})  \\
                &=  R(\sigma_n^*(\vec{m})) + \sum_{i = 1}^b y_i(\sigma_n^*(\vec{m}),\vec{z}^++\vec{z}^-)  \\
                &\geq  R(\sigma_n^*(\vec{m})) + \sum_{i = 1}^b \left( y_i(\sigma_n^*(\vec{m}),\vec{z}^+)+y_i(\sigma_n^*(\vec{m}),\vec{z}^-) \right) \\
                &=  R(\sigma_n^*(\vec{m})) + [R(\tau_n^+) - R(\sigma_n^*(\vec{m}))] + [R(\tau_n^-)-R(\sigma_n^*(\vec{m}))]  \\
                &=  R(\tau_n^+) + R(\tau_n^-) - R(\sigma_n^*(\vec{m})) \geq R(\tau_n^+).
        \end{align*}
        This gives the desired inequality. Next, observe that $\tau_n$ and $\tau_n^+$ are $p$-bounded. To see that $\tau_n$ is $p$-bounded, use Proposition \ref{p: minimal permutations have all the nice properties} together with the fact that $\tau$ is an element of $A_n(\vec{m})$ with minimal $R$-value. 
        To see that $\tau_n^+$ is $p$-bounded, note that $\tau_n^+=\max\{\tau_n,\sigma_n^*(\vec{m})\}$ and that the maximum of
        two $p$-bounded permutations is again $p$-bounded. Then from Lemma \ref{lemma: a-cycles have different valuations} we see that $R(\tau_n^+)=R(\tau_n)$ if and only if $\tau_n^+=\tau_n$.

	\end{proof}

    \begin{corollary}\label{c: tau_n is tau_n^+}
        We have $\tau_n=\tau_n^+$, or equivalently $\tau_n \geq \sigma_n^*(\vec{m})$. 
    \end{corollary}
    \begin{proof}
        Note that $\tau_n^+ \geq \tau_n$, so that $\tau_n^+$ is disjoint
        from $\Sigma_{n-1}^*(\tau_n-\vec{1})$. In particular, we see
        that $\tau_n^+\Sigma_{n-1}^*(\tau_n-\vec{1})$ is in $A_n(\vec{m})$. The corollary then follows from our minimal $R$-value assumption on $\tau$ and Lemma \ref{claim: only need to increase indices}.
    \end{proof}

    \subsubsection{Reduction to two claims}
    For a $\Z/b\Z$-tuple of non-negative integers $\vec{r} = (r_1,\dots,r_b)$ we partition $\Z/b\Z$ as follows:
	\begin{align*}
		N^0(\vec{r}) &:= \{ i\in \Z/b\Z ~|~ r_i=r_{i-1}=0\}, \\
		N^-(\vec{r}) &:= \{ i\in \Z/b\Z ~|~ pr_{i}-r_{i-1}<0\}, \\
		N^+(\vec{r}) &:= \{ i\in \Z/b\Z ~|~i \not\in N^0 \text{ and }pr_{i}-r_{i-1}\geq 0\}.
	\end{align*}
	In addition we set $N^\pm(\vec{r})= N^-(\vec{r}) \cup N^{+}(\vec{r})$. Note that for $i \in N^0(\vec{z})$ we have $y_i(\tau_n)=  y_i (\sigma_n^*(\vec{m}))$.
	We also remark that for $i \in N^-(\vec{z})$ we have $y_i(\sigma_n^*(\vec{m}),\vec{z})>-y_i(\sigma_n^*(\vec{m}))$. In particular,
	\[
		R(\tau_n) - R(\sigma_n^*(\vec{m})) = \sum_{i=1}^b y_i(\sigma_n^*(\vec{m}),\vec{z}) \geq  \sum_{i \in N^+(\vec{z})} y_i(\sigma_n^*(\vec{m}),\vec{z}) - \sum_{i \in N^-(\vec{z})} y_i(\sigma_n^*(\vec{m})).\]
	Combining this with \eqref{eq: the fundamental equality 2}, we are reduced to proving the following inequality:
	\begin{align}\label{eq: the fundamental equality 4}
		\sum_{i \in N^+(\vec{z})} y_i(\sigma_n^*(\vec{m}),\vec{z}) &> \sum_{j=1}^{n-1} \Big[ R(\sigma_j) - R(\tau_j) \Big] + \sum_{i \in N^-(\vec{z})} y_i(\sigma_n^*(\vec{m})).
	\end{align}
	The inequality \eqref{eq: the fundamental equality 4} will follow from the next two claims.
	
	\begin{claim} \label{claim: first}
		We have
		\begin{align} \label{eq: claim 2 inequality}
			\sum_{i \in N^-(\vec{z})} pd_i(\sigma_n^*(\vec{m}))+\sum_{i \in N^+(\vec{z})} pd_i(\sigma_n^*(\vec{m})-\vec{1}) > \sum_{j=1}^{n-1} \Big[ R(\sigma_j) - R(\tau_j) \Big] + \sum_{i \in N^-(\vec{z})} y_i(\sigma_n^*(\vec{m})) .
		\end{align}
	\end{claim}

    \begin{claim} \label{claim: second}
		We have
		\begin{align*}
			\sum_{i \in N^+(\vec{z})} y_i(\sigma_n^*(\vec{m}),\vec{z}) > \sum_{i \in N^-(\vec{z})} pd_i(\sigma_n^*(\vec{m}))+\sum_{i \in N^+(\vec{z})} pd_i(\sigma_n^*(\vec{m})-\vec{1}).
		\end{align*}
	\end{claim}

    \subsubsection{Proof of Claim \ref{claim: first}}
    We first prove two easy lemmas. 
    \begin{lemma}\label{lemma: bounding sigmaj minus tauj}
		For $j=1,\dots,n$ we have 
		\begin{align*}
			y_i(\sigma_n^*(\vec{m})-(n-j)\vec{1})&\geq y_i(\sigma_j)-y_i(\tau_j) .
		\end{align*}
	\end{lemma}
	\begin{proof}
        If $c_{j,i-1}=k_{j,i-1}$, then since $c_{j,i}\leq k_{j,i}$ we have $y_i(\sigma_j)\leq y_i(\tau_j)$, and the lemma follows. If $c_{j,i-1}<k_{j,i-1}$, we have
		$c_{j,i-1}=c_{i-1}-(n-j)$. Since $c_{j,i}\leq c_{i} - (n-j)$, we have
		\begin{align*}
			pc_{j,i} - c_{j,i-1}&\leq pc_{i} - c_{i-1} - (n-j)(p-1).
		\end{align*}
		Thus, $y_i(\sigma_j)\leq y_i(\vec{c}-(n-j)\vec{1})$, which
		gives the lemma.
	\end{proof}

    \begin{lemma} \label{lemma: smaller coordinates}
		If $z_i= 0$, then $c_{j,i}=k_{j,i}$ for all $j$, i.e., the $i$-coordinate of $\tau_j$ and $\sigma_j$ agree.
	\end{lemma}
    \begin{proof}
        This follows from the definition of $\tau_j$.
    \end{proof}

    We now prove Claim \ref{claim: first}. If $i \in N^0(\vec{z})$, then by Lemma \ref{lemma: smaller coordinates} we have $y_i(\sigma_j)=y_i(\tau_j)$ for all $j$.
		In particular, we see that the right side of \eqref{eq: claim 2 inequality} is equal to 
		\[ \sum_{i\in N^\pm(\vec{z})} \sum_{j=1}^{n-1} \Big[	y_i(\sigma_j)-y_i(\tau_j) \Big ]  + \sum_{i \in N^-(\vec{z})} y_i(\sigma_n^*(\vec{m})). \]
		Then from Lemma \ref{lemma: key growth bound 2} and Lemma \ref{lemma: bounding sigmaj minus tauj} we see that 
		\begin{align*}
			&\sum_{i\in N^\pm(\vec{z})} \sum_{j=1}^{n-1} \Big[	y_i(\sigma_j)-y_i(\tau_j) \Big ]  + \sum_{i \in N^-(\vec{z})} y_i(\sigma_n^*(\vec{m})) \\
            \leq& 
            \sum_{i \in N^-(\vec{z})} \sum_{j=0}^{n-1} y_i(\sigma_n^*(\vec{m}) - j\vec{1}) + \sum_{i \in N^+(\vec{z})} \sum_{j=1}^{n-1} y_i(\sigma_n^*(\vec{m}) - j\vec{1})\\
            <&\sum_{i \in N^-(\vec{z})} pd_i(\sigma_n^*(\vec{m}))+\sum_{i \in N^+(\vec{z})} pd_i(\sigma_n^*(\vec{m})-\vec{1}). 
		\end{align*}\qed

    \subsubsection{Proof of Claim \ref{claim: second}}
    To prove the second claim we first need a quick lemma.
    \begin{lemma}\label{lemma: inequalities with different digits}
    	Let $\vec{r}\neq \vec{0}$ be a $b$-tuple of non-negative integers such that
        $\sigma_n^*(\vec{m})+\vec{r} \in S^{\geq n}(\vec{m})$. There exists $i_0 \in N^+(\vec{r})$
        such that
        \begin{align*}
            d_{i_0}(\sigma_n^*(\vec{m})+\vec{r}) \geq \max_{i \in N^{-}(\vec{r})} pd_i(\sigma_n^*(\vec{m})).
        \end{align*}
    \end{lemma}
    \begin{proof}
        Since $\sigma_n^*(\vec{m})$ has minimal $R$-value in $S^{\geq n}(\vec{m})$ we know that
        \begin{align*}
            \sum_{i \in N^{\pm}(\vec{r})} y_i(\sigma_n^*(\vec{m}),\vec{r}) &> 0,     
        \end{align*}
        or equivalently 
        \begin{align}\label{eq: inequality with digits}
            \sum_{i \in N^+(\vec{r})} y_i(\sigma_n^*(\vec{m}),\vec{r}) &> -\sum_{i \in N^-(\vec{r})} y_i(\sigma_n^*(\vec{m}),\vec{r}).
        \end{align}
        The $p$-adic digits present in the left side are different from those present in the right side. The largest digit on the left side of \eqref{eq: inequality with digits} will be $d_{i_0}(\sigma_n^*(\vec{m})+\vec{r})$ for some $i_0 \in N^{+}(\vec{r})$. 
        Then $d_{i_0}(\sigma_n^*(\vec{m})+\vec{r})$ will be larger than any $p$-adic digit occuring in the right side of \eqref{eq: inequality with digits}. 
        This is because the $p$-adic expansion of the right side of $\eqref{eq: inequality with digits}$ cannot have any terms of the form $p^{i_0-1}q^k$. 
        In particular, we have $d_{i_0}(\sigma_n^*(\vec{m})+\vec{r})> d_{i}(\sigma_n^*(\vec{m}))$ for all $i \in N^-(\vec{r})$. As both $d_{i_0}(\sigma_n^*(\vec{m})+\vec{r})$ and $d_{i}(\sigma_n^*(\vec{m}))$ are powers of $p$, we must have $d_{i_0}(\sigma_n^*(\vec{m})+\vec{r})\geq pd_{i}(\sigma_n^*(\vec{m}))$.
    \end{proof}
    
			For any tuple $\vec{r}=(r_1,\dots,r_b)$ of non-negative integers, define
            \begin{align*}
                \alpha(\vec{r}) &=  \sum_{i \in N^+(\vec{r})} y_i(\sigma_n^*(\vec{m}),\vec{r}),  \\   \beta(\vec{r})  &=  \sum_{i \in N^-(\vec{r})} pd_i(\sigma_n^*(\vec{m}))+\sum_{i \in N^+(\vec{r})} pd_i(\sigma_n^*(\vec{m})-\vec{1}).
            \end{align*}
            To prove Claim \ref{claim: second}, we will prove the more general result that
			\begin{align}\label{eq: claim 3 more general statement}
				\alpha(\vec{r}) > \beta({\vec{r}}).
			\end{align}
			Assume that there exists $k'>k+1$ with $r_k=r_{k'}=0$. We define $\vec{r}_1=(r_{1,1},\dots, r_{1,b})$ by
			\begin{align*}
				r_{1,j} &= \begin{cases}
					r_j & k<j<k' \\
					0 & \text{otherwise} 
				\end{cases}
			\end{align*}
			and define $\vec{r}_2$ so that $\vec{r}=\vec{r}_1 + \vec{r}_2$. By definition we have 
			\begin{align*}
				\alpha(\vec{r})&=\alpha(\vec{r}_1)+\alpha(\vec{r}_2) \\
				\beta (\vec{r})&=\beta (\vec{r}_1)+\beta (\vec{r}_2).
			\end{align*}
			This allows us to reduce to the case where $r_j$ is positive for $k<j<k'$ and zero otherwise (i.e. $\vec{r}$
			has one contiguous sequence of positive numbers surrounded by all zeros). The proof will proceed by induction on $m=\max\{r_i\}$. We will only consider the case where $\vec{r}=(r_1,\dots,r_k,0,\dots,0)$ and each $r_i>0$ for $i=1,\dots,k$. The proof of the general case is identical but involves cumbersome notation.

			Our base case 
            is when $m=1$, so that \[\vec{r}=\vec{1}_k:=(r_1,\dots,r_k,0,\dots,0),\]
            and each $r_i=1$. 
            We prove this by induction on $k$. When $k=1$ we have
            $N^+(\vec{r})=\{1\}$ and $N^-(\vec{r})=\{2\}$. Then by
            Lemma \ref{lemma: inequalities with different digits} we have
            \begin{equation*}
                d_1(pc_1-c_0 + p) = d_1(\vec{c} + \vec{r}) \geq pd_2(\vec{c}).
            \end{equation*}
            By definition of $d_1$ we have 
            \begin{align*}
                d_1(pc_1-c_0 + p-1) & \geq qd_1(pc_1 - c_0) = qd_1(\vec{c}).
            \end{align*}
            This implies \eqref{eq: claim 3 more general statement}. 
            Now consider $k>1$. Then we have $N^{-}(\vec{r})=\{k+1\}$ and
            $N^+(\vec{r})=\{1,\dots,k\}$. By Lemma \ref{lemma: inequalities with different digits} there exists $i_0 \in \{1,\dots,k\}$ such that 
            $d_{i_0}(\vec{c}+\vec{r})\geq pd_{k+1}(\vec{c})$. If $i_0=1$, since
            $y_i(\vec{c},p-1)>qd_i(\vec{c})$ for $i=1,\dots,k$ we have
            \begin{equation*}
                \sum_{i=1}^{k} y_i(\vec{c},\vec{r}) =d_1(pc_1-c_0+p) + \sum_{i=1}^{k} y_i(\vec{c},p-1) > \sum_{i=1}^{k+1} pd_i(\vec{c}),
            \end{equation*}
            which implies \eqref{eq: claim 3 more general statement}.
            Finally, assume $1<i_0\leq k$. By our inductive hypothesis,
            we know that \eqref{eq: claim 3 more general statement} holds
            for $\vec{1}_{i_0-1}$, so that
            \begin{equation*}
                \sum_{i=1}^{i_0-1} y_{i}(\vec{c},\vec{r}) = \sum_{i=1}^{i_0-1} y_{i}(\vec{c},\vec{1}_{i_0}) > \sum_{i=1}^{i_0-1} pd_i(\vec{c}-\vec{1}) + pd_{i_0}(\vec{c})> \sum_{i=1}^{i_0} pd_i(\vec{c}-\vec{1}).
            \end{equation*} 
            For $i_0<i\leq k$ we have $y_{i}(\vec{c},\vec{r})=y_i(\vec{c},p-1)$,
            which gives
            \begin{equation*}
                \sum_{i=i_0+1}^{k} y_{i}(\vec{c},\vec{r}) = \sum_{i=i_0+1}^{k} y_{i}(\vec{c},p-1)>\sum_{i=i_0+1}^{k} pd_i(\vec{c}-\vec{1}).
            \end{equation*}
            Putting this all together we have
            \begin{equation*}
                \sum_{i=1}^{k} y_i(\vec{c},\vec{r}) = \sum_{i=1}^{i_0-1} y_i(\vec{c},\vec{r}) + y_{i_0}(\vec{c},\vec{r}) + \sum_{i=i_0+1}^{k} y_i(\vec{c},\vec{r}) >\sum_{i=1}^{k} pd_i(\vec{c}-\vec{1}) + pd_{k+1}(\vec{c}),  
            \end{equation*}
            which gives \eqref{eq: claim 3 more general statement}.

            Consider a general $\vec{r}=(r_1,\dots,r_k,0,\dots,0)$ with $r_i>0$. Set $m=\max\{r_i\}$. Assume the claim holds for any $b$-tuple with smaller maximum. Let $\vec{s}=\vec{r}-\vec{1}_k$.
            Note that $N^+(\vec{s}) \subset N^+(\vec{r})$. In particular,        for $i \in N^+(\vec{s})$ we have
            \begin{equation}
            \begin{split}\label{eq: final claim inductive eq3}
                y_i(\vec{c},\vec{r}) = y_i(\vec{c},\vec{s}) + y_i(\vec{c}+\vec{s},\vec{1}_k) \geq y_i(\vec{c},\vec{s}) + d_i(\vec{c}+\vec{r}).
            \end{split}
            \end{equation}
            If $1\leq i \leq k$ and $i \not\in N^{\pm}(\vec{s})$, we have $i \in N^{+}(\vec{r})$. More precisely, we have $y_i(\vec{c},\vec{r})=y_i(\vec{c},p-1)$ or $y_i(\vec{c},\vec{r})=y_i(\vec{c},p)$, the latter occurring when $i=1$. 
            Thus, we compute
            \begin{equation}
                \begin{split}
                    \label{eq: final claim inductive eq4}
                    y_i(\vec{c},\vec{r}) - d_i(\vec{c}+\vec{r}) \geq d_i(\vec{c}) \geq  qd_i(\vec{c}-\vec{1}). 
                \end{split}
            \end{equation}
            Again, by Lemma \ref{lemma: inequalities with different digits} we know that for some $i_0\in N^+(\vec{r})$ we have
            \begin{align}\label{eq: final claim inductive eq5}
                d_{i_0}(\vec{c}+\vec{r}) &\geq pd_{k+1}(\vec{c}).
            \end{align}
            From our inductive hypotheses, \eqref{eq: final claim inductive eq3}, \eqref{eq: final claim inductive eq4}, and \eqref{eq: final claim inductive eq5} we obtain
            \begin{align*}
                \sum_{i \in N^+(\vec{r})} y_i(\vec{c},\vec{r}) &= \sum_{i \in N^+(\vec{r})} d_i(\vec{c}+\vec{r}) + \sum_{i \in N^+(\vec{r})} \Big[ y_i(\vec{c},\vec{r}) - d_i(\vec{c}+\vec{r}) \Big]  \\
                &\geq \sum_{i \in N^+(\vec{r})} d_i(\vec{c}+\vec{r}) + \sum_{i \in N^+(\vec{s})}\Big[ y_i(\vec{c},\vec{r}) - d_i(\vec{c}+\vec{r}) \Big]    + \sum_{\stackrel{1\leq i \leq k}{i \not\in N^{\pm}(\vec{s})}}\Big[ y_i(\vec{c},\vec{r}) - d_i(\vec{c}+\vec{r}) \Big] \\
                &>\sum_{i \in N^+(\vec{r})} d_i(\vec{c}+\vec{r}) +\sum_{i \in N^+(\vec{s})} y_i(\vec{c},\vec{s}) + \sum_{\stackrel{1\leq i \leq k}{i \not\in N^{\pm}(\vec{s})}} pd_i(\vec{c}-\vec{1}) \\
                &> pd_{k+1}(\vec{c}) + \sum_{i \in N^-(\vec{s})} pd_i(\vec{c})+ \sum_{i \in N^+(\vec{s})} pd_i(\vec{c}-\vec{1})  + \sum_{\stackrel{1\leq i \leq k}{i \not\in N^{\pm}(\vec{s})}} pd_i(\vec{c}-\vec{1}) \\
                & \geq \sum_{i \in N^-(\vec{s}) \cup \{k+1\}} pd_i(\vec{c})+ \sum_{i \in N^+(\vec{s})} pd_i(\vec{c}-\vec{1})  + \sum_{\stackrel{1\leq i \leq k}{i \not\in N^{\pm}(\vec{s})}} pd_i(\vec{c}-\vec{1}) 
            \end{align*}
            The claim follows by observing that $N^-(\vec{r}) \subset N^{-}(\vec{s}) \cup \{k+1\}$ and because $d_i(\vec{c})>d_i(\vec{c}-\vec{1})$.\qed

		\subsection{Finishing the proof of Theorem \ref{theorem: unique minimal term}}\label{ss: finishing proof of affine line}
		We can now prove our main results about the affine line.
		\begin{proof}
			(Of Theorem \ref{theorem: unique minimal term}) Let $\Sigma_n$ and $\Sigma_n'$ be $R$-minimal
			permutations of $J_1$ of size $n$. By Proposition \ref{p: minimal permutations have all the nice properties}
			we know that $\Sigma_n$ and $\Sigma_n'$ are lexicographical. Let $\vec{m}$ be a $b$-tuple large enough so that
			$\Sigma_n,\Sigma_n' \leq \vec{m}$. Then $\Sigma_n,\Sigma_n' \in A_n(\vec{m})$ and by Proposition \ref{proposition: key technical proposition} we have $\Sigma_n=\Sigma_n'=\Sigma_n^*(\vec{m})$. Thus there is a unique $R$-minimal
			permutation of $J_1$ of size $n$. From Proposition \ref{p: minimal permutations have all the nice properties}
			we also know that $\Sigma_n$ is $p$-bounded. Then by Lemma \ref{l: compare R and valuation affine line} we see that
			$\Sigma_n$ is the unique $v$-minimal permutation of size $n$. 
			
			By the previous paragraph and Proposition \ref{proposition: key technical proposition} we know $\Sigma_n = \Sigma_{n-1}^*(\vec{c}-1)\vec{c}$, where $\vec{c}=\sigma_n^*(\vec{m})$. We know  $R(\Sigma_{n-1}^*(\vec{c}-1)) \geq R(\Sigma_{n-1})$, which gives
			\begin{align*}
				R(\Sigma_n) - R(\Sigma_{n-1}) &\geq R(\vec{c}).
			\end{align*}
			Let $\tau_k=\vec{c}-k\vec{1}$ and consider $\Sigma_{n-1}'=\tau_1\dots\tau_{n-1}$. By applying Lemma \ref{lemma: key growth bound 2} to $\displaystyle\sum_{i=1}^n y_i(\vec{c}-i\vec{1})$ we have
			\[
				R(\vec{c})  = \sum_{i=1}^b y_i(\vec{c}) > \sum_{i=1}^b pd_i(\vec{c}-\vec{1}) > R(\Sigma_{n-1}').\]
			We have $R(\Sigma_{n-1}')\geq R(\Sigma_{n-1})$, since $\Sigma_{n-1}$ is $R$-minimal. Thus, 
			\begin{align*}
				R(\Sigma_n) - R(\Sigma_{n-1}) &> R(\Sigma_{n-1}) - R(\Sigma_{n-2}).
			\end{align*}
			Since the $\Sigma_k$ are $p$-bounded, the theorem follows from Lemma \ref{l: compare R and valuation affine line}.
		\end{proof}
	
		\begin{corollary} \label{c: theorem for affine line}
			Fix $y \in \Z_p$. Then every slope in the Newton polygon of $\zeta_{\F_r[\theta],\pi}(x,y)$ occurs with multiplicity
			one and is a integer multiple of $r-1$.
		\end{corollary}
		
		\begin{proof}
			First, consider the case where $y$ is $q$-full. Write
            \begin{align*}
                \det(1-x\Psi^\circ) &= 1 + c_1x + c_2x^2 + \dots 
            \end{align*}
            Using the definition of the Fredholm determinant of a matrix, we know by Theorem \ref{theorem: unique minimal term} and Lemma \ref{lemma: the nonzero coefficients} that $c_i=0$ if $b \nmid i$ and that 
            $v(c_{bn}) = v(\Sigma_n)$. Also, if we set $\nu_n=v(\Sigma_n)-v(\Sigma_{n-1})$, we know $\nu_{n+1} > \nu_n$ by Theorem \ref{theorem: unique minimal term}. Thus, the slopes of the Newton polygon of $\det(1-x\Psi^\circ)$ are
			\[ \frac{\nu_1}{b}, \frac{\nu_2}{b}, \frac{\nu_3}{b}, \dots \]
			and each slope occurs with multiplicity $b$. By \eqref{eq: matrix calculation of Newton polygon} we know the slopes of the Newton polygon of $\zeta_{\F_r[\theta],\pi}(x,y)$ are 
			\[ \nu_1,\nu_2, \nu_3, \dots \]
			and that each slope occurs with multiplicity one. 
			
			For the general case, let $y \in \Z_p$. The subset of $q$-full numbers is dense in $\Z_p$. Let $y_1,y_2, \dots$
			be a sequence of $q$-full $p$-adic numbers that converge to $y$ such that 
			\begin{align*}
				y-y_i & \in p^{i}\Z_p
			\end{align*} 
			This implies $\zeta_{\F_r[\theta],\pi}(x,y) \equiv \zeta_{\F_r[\theta],\pi}(x,y_i) \mod \pi^{p^i}$. To see this, note that $\rho_{\F_r[\theta],\pi}^{\otimes y}$ and $\rho_{\F_r[\theta],\pi}^{\otimes y_i}$ are congruent modulo $\pi^{p^i}$. 
            In particular, every vertex of $\NP_\pi(\zeta_{\F_r[\theta],\pi}(x,y_i))$ whose $y$-coordinate is less than $p^i$ is also a vertex of $\NP_\pi(\zeta_{\F_r[\theta],\pi}(x,y))$. By taking $i$ to be large, we see that every slope of $\NP_\pi(\zeta_{\F_r[\theta],\pi}(x,y))$ occurs with multiplicity one.

            To prove divisibility by $r-1$, it suffices to show $r-1|\nu_n$. This will follow from the much stronger statement: For any rotational $b$-cycle $\sigma$ of $J_1$ we have $r-1|v(\sigma,\Psi^\circ)$. 
            Write $\sigma=(m_1,\dots,m_b)$, so that
            \begin{align*}
                v(\sigma,\Psi^\circ) &= \sum_{i=1}^{b} y_i(pm_{i} - m_{i-1}).
            \end{align*}
            By the definition of $y_i$ we know $y_i(k) \equiv p^{i-1}k \mod r-1$.
            Thus, we compute
            \begin{align*}
                v(\sigma,\Psi^\circ) & \equiv \sum_{i=1}^b p^{i-1}(pm_{i} - m_{i-1}) \mod r-1 \\
                &\equiv (p^b-1)m_b \equiv 0 \mod r-1.
            \end{align*}
		\end{proof}

    \section{General ordinary curves} \label{s: general ordinary curves}
    We now return to the general case where $X$ is a curve over $\F_q$ and $\infty \in X$ is a closed point of degree $d$. 
    We adopt the notation used in \S \ref{s:zeta} and \S \ref{s: tau crystals and the trace formula}.

    \subsection{The initial setup} \label{ss: initial setup}
    Let $\rho_{A,\pi}$ be the representation defined in \S \ref{ss:zeta} and let $\rho_{A_{\F_r}}$ denote
    the restriction of $\rho_{A,\pi}$ to $\pi_1(\Spec(A_{\F_r}))$. 
    Let $\mathscr{G}_{A_{\F_r}}$ denote the $\tau$-module associated to $\rho_{A,\pi}$ using
    Proposition \ref{p:Katz-correspondence}. Let
    $\mathscr{G}_{A_{\F_r},\infty}$ denote the localization of $\mathscr{G}_{A_{\F_r}}$ at $\infty$.
    Note that $\mathscr{G}_{A_{\F_r},\infty}$ is the $\tau$-module associated
    to the restriction of $\rho_{A,\pi}$ along the map $\Spec(K_{\F_r}) \to \Spec(A)$.
    
    \begin{proposition}\label{proposition: tau module at infinity comes from affine line case}
        The element $1 + \pi\theta$ is a Frobenius matrix of $\mathscr{G}_{A_{\F_r},\infty}$. 
    \end{proposition}
    \begin{proof}
    Consider the diagram
    \begin{equation*}
    \begin{tikzcd}
\Spec(K_{\F_r}) \arrow[r, "h"] & \Spec(K) \arrow[r, "g"] \arrow[rd, "f"] & \Spec(\F_r[\theta]) \\ 
& & A
\end{tikzcd}.
    \end{equation*}
    By Corollary \ref{c: localization does not depend on curve} we know that
    $g^*\rho_{\F_r[\theta],\pi}$ and $f^*\rho_{A,\pi}$ are isomorphic representations
    of $\pi_1^{et}(\Spec(K))$.  In particular, we have $h^*g^*(\rho_{\F_r[\theta],\pi})\cong h^*f^*(\rho_{A,\pi})$.
    From Proposition \ref{p:A1-Frob} we know that $1 + \pi\theta$ is
    a Frobenius matrix for $\mathscr{F}_{\F_r[\theta]}$. It follows
    that $1 + \pi\theta$ is a Frobenius matrix for the $\tau$-module
    associated to $h^*g^*(\rho_{\F_r[\theta],\pi})$. As $h^*g^*(\rho_{\F_r[\theta],\pi})\cong h^*f^*(\rho_{A,\pi})$, we see that $1+\pi\theta$ is a Frobenius matrix of $\mathscr{G}_{A_{\F_r},\infty}$.

    \end{proof}
    
    \begin{proposition}\label{proposition: Goss crystal is free, OC and 1 mod pi}
    The underlying $A_{\F_r}\widehat{\otimes}_{\F_p} \bfV^\circ$-module of $\mathscr{G}_{\F_r}$ is free.
    Furthermore, there exists a Frobenius matrix $\alpha$ of $\mathscr{G}_{\F_r}$
    with $\alpha\equiv 1 \mod \pi$.
    \end{proposition}
    \begin{proof}
        First, note that $\mathscr{G}_{\F_r}\otimes_{\F_r} \bfV/\pi^{1/p^h}$ is
        free of rank one, since the residual representation of $\rho_{A_{\F_r}}$ is trivial. For any $n$, the deformation problem
        of lifting an invertible sheaf over $\Spec(A_{\F_r})$ to $\Spec(A_{\F_r}\otimes_{\F_r} \bfV/\pi^{n/p^h})$ is classified by
        $H^1$ of a coherent sheaf on $\Spec(A_{\F_r})$ (see \cite{Hartshorne-deformation_theory}*{Theorem 6.4}.) Since $\Spec(A_{\F_r})$ is affine, we know that $H^1$ of any coherent sheaf vanishes. Thus, the only deformation of $\mathcal{O}_{\Spec(A_{\F_r})}$ from $\Spec(A_{\F_r})$ to $\Spec(A_{\F_r}\otimes_{\F_r} \bfV/\pi^{n/p^h})$ is
        $\mathcal{O}_{\Spec(A_{\F_r}\otimes_{\F_r} \bfV/\pi^{n/p^h})}$.
        We then deduce that the only deformation of $\mathcal{O}_{\Spec(A_{\F_r})}$ from $\Spec(A_{\F_r})$ to $A_{\F_r}\widehat{\otimes}_{\F_r} \bfV$ is $\mathcal{O}_{A_{\F_r}\widehat{\otimes}_{\F_r} \bfV}$. 
    \end{proof}
    
    Fix $\alpha$ as in Proposition \ref{proposition: Goss crystal is free, OC and 1 mod pi}. Let $y \in \Z_p$ and
    write $y=\displaystyle\sum_{i=1}^{b}y_ip^{i-1}$ as in \eqref{eq:y-decomposition}. Recall from \S \ref{sss: F and tau basic defs} that $W$ is the $p$-th Frobenius map relative to $\bfV$. For $i=1,\dots,b$ we define: 
    \begin{align*}
        \alpha_i&= \alpha^{y_{i}W^{i-1}}.
    \end{align*}
    \begin{lemma}
        \label{l: compare affine line frob to general case}
        Let $\beta_i$ be as defined in \eqref{eq: product description of beta}. Then $\beta_i$ is $\tau$-equivalent to $\alpha_i$ over $K_{\F_r} \widehat{\otimes}_{\F_p} \bfV^\circ$. In particular,
        there exists $d_i \in K_{\F_r} \widehat{\otimes}_{\F_p} \bfV^\circ$ such that $\beta_i d_i^{1-\tau}= \alpha_i$. Furthermore, we may take $d_i$ to satisfy
            $d_i \equiv 1 \mod \pi^{1/p^h} K_{\F_r} \widehat{\otimes}_{\F_p} \bfV^\circ$.
    \end{lemma}
    \begin{proof}
        From Proposition \ref{proposition: tau module at infinity comes from affine line case} we see that $\alpha^{y_i}$ is $\tau$-equivalent to 
        $(1+\pi\theta)^{y_{i}}$ over $K_{\F_r} \otimes_{\F_p} \bfV^\circ$. This means $\alpha_i$ is $\tau$-equivalent to $(1+\pi^{p^{i-1}}\theta)^{y_i}$ over $K_{\F_r} \widehat{\otimes}_{\F_p} \bfV^\circ$. From \eqref{eq: product description of beta} we see that $\beta_i$ is $\tau$-equivalent to $(1+\pi^{p^{i-1}}\theta)^{y_i}$
        over $K_{\F_r} \widehat{\otimes}_{\F_p} \bfV^\circ$.
        This gives the lemma.
    \end{proof}
    Set $\mathscr{F}_{A_{\F_r}}= \Res(\mathscr{G}_{A_{\F_r}}^{\otimes y})$. Since $\alpha^{y} = 
    \alpha_1 \alpha_2^F\dots \alpha_{b}^{F^{b-1}}$
    we know from Proposition \ref{prop: shape of restriction of scalars Frobenius structure} that there exists a basis $\{f_1,\dots,f_b\}$
    of $\mathscr{F}_{A_{\F_r}}$ whose corresponding Frobenius matrix is
    \begin{align*}
        E_{\mathscr{F}_{A_{\F_r}}} &= \begin{bmatrix} 0 & 0 & \dots & 0 &  \alpha_1 \\ 
	\alpha_2  & 0 & \dots & 0 & 0 \\
	0 & \ddots & \ddots & \vdots & \vdots \\
	\vdots & \ddots & \ddots & 0 & 0 \\
	0 & \dots & 0 &\alpha_b& 0
	\end{bmatrix}.
    \end{align*}
    Let $d_i$ be as in Lemma \ref{l: compare affine line frob to general case}. Define the following elements of $K_{\F_r} \otimes_{\F_p} \bfV^\circ $:
    \begin{align*}
        \delta_i &:= \prod_{j=0}^{b-1}d_{i+j}^{F^{j}}.
    \end{align*}
    We then define the matrix:
    \begin{align*}
        D&:=\begin{bmatrix} \delta_1 & 0 & \dots & 0  \\ 
	0 & \delta_2 & \dots & 0  \\
    \vdots & \ddots & \ddots & \vdots  \\
	0 & \dots & \dots & \delta_b 
	\end{bmatrix}.
    \end{align*}
    A quick calculation shows
    \begin{align} \label{eq: local twisting at infinity}
        E_{\mathscr{F}_{\F_r[\theta]}} &= D E_{\mathscr{F}_{A_{\F_r}}}D^{-F},
    \end{align}
    where $E_{\mathscr{F}_{\F_r[\theta]}}$ is the matrix defined in \S \ref{ss: the initial setup for the affine line}.
    \begin{proposition}\label{proposition: computing L function via A0}
        Let $M_{\mathscr{F}_{A_{\F_r}}}$ and $V_{\mathscr{F}_{A_{\F_r}}}$ be defined as in \S \ref{ss: the trace formula}. Then
        \begin{align*}
            \NP_\pi(L(\rho_{A_{\F_r}}^{\otimes y}, x^b)^b) &= \NP_\pi(\det_{\bfV^\circ}(1-xU_p\circ E_{\mathscr{F}_{\F_r[\theta]}}|D^FM_{\mathscr{F}_{A_{\F_r}}}).
        \end{align*}
    \end{proposition}

    \begin{proof}
        This follows from  Corollary \ref{cor: a-th root trick} and \eqref{eq: local twisting at infinity}. 
    \end{proof}

    \begin{remark}
        By Proposition \ref{proposition: computing L function via A0},
        the operator that we use to compute $\NP_\pi(L(\rho_{A_{\F_r}}^{\otimes y}, x))$ is the same operator we use to compute $\NP_\pi(L(\rho_{\F_r[\theta],\pi}^{\otimes y}, x))$. The difference is that the same operator is acting on different spaces: $D^F M_{\mathscr{F}_{A_{\F_r}}}$ in the former case and $M_{\mathscr{F}_{\F_r[\theta]}}$ in the latter case. If we make the further assumption that $d=1$ (i.e. so $A_{\F_r}=A$),
        then we can identify $V_{\mathscr{F}_{A_{\F_r}}}$ with $V_{\mathscr{F}_{\F_r[\theta]}}$. This follows from the observation that the construction $\mathscr{F} \to V_{\mathscr{F}}$ only depends on the the localization of the $F$-module at $\infty$. In this special scenario, we can view both $D^F M_{\mathscr{F}_{A_{\F_r}}}$ and $M_{\mathscr{F}_{\F_r[\theta]}}$ as subsets of the same space: $V_{\mathscr{F}_{A_{\F_r}}}$. The two $L$-functions
        are computed by the Fredholm determinants of $\Theta_{\F_r[\theta]}=U_p\circ E_{\mathscr{F}_{\F_r[\theta]}}$ acting on different subspaces of $V_{\mathscr{F}_{A_{\F_r}}}$. The key idea is that the matrices used to compute the two Fredholm determinants are very similar. 
    \end{remark}

    \subsection{A matrix representation\texorpdfstring{ of $U_p \circ E_{\mathscr{F}_{\F_r[\theta]}}$ on $ D^F M_{\mathscr{F}_{A_{\F_r}}}$}{}} \label{ss: comparing affine line case}
    In this subsection we choose a basis for $ D^F M_{\mathscr{F}_{A_{\F_r}}}$
    and describe the matrix of $U_p \circ E_{\mathscr{F}_{\F_r[\theta]}}$
    in terms of this basis.

    \subsubsection{Choice of basis} \label{ss: choice of basis}
    In
    \S \ref{ss: initial setup} we introduced a basis $\{f_1,\dots,f_b\}$ of $\mathscr{F}_{A_{\F_r}}$ whose corresponding Frobenius matrix is $E_{\mathscr{F}_{A_{\F_r}}}$. Then $V_{\mathscr{F}_{A_{\F_r}}}$
    is a free $K_{\F_r} \widehat{\otimes}_{\F_p} \bfV^\circ$-module with basis $\{e_1,\dots,e_b\}$, where $e_i=f_i \otimes \frac{d\theta}{\theta}$. The maps $\pr$ and
    $U_p$ are defined with respect to this basis as in \S \ref{ss: the trace formula}. Using this basis we have
    \begin{align*}
        D^F M_{\mathscr{F}_{A_{\F_r}}} = \bigoplus_{i=1}^b \delta_i^F\Omega_{A_{\F_r} \widehat{\otimes} \bfV^\circ} f_i.
    \end{align*}
    We let $\pr_i$ denote the restriction of $\pr$ to the $i$-th summand.
    \begin{lemma} \label{l: the projection is surjective}
    	The projection map \[\pr_i: \delta_i^F\Omega_{A_{\F_r} \widehat{\otimes} \bfV^\circ} f_i \to \theta \bfV_{\F_r}\langle \theta \rangle e_i \] is surjective
    	and the kernel is free over $\bfV$ of rank $g+d-1$. 
    \end{lemma}
	\begin{proof}
		The proof is almost identical to that of Lemma \ref{p: proj is surjective}.
	\end{proof}

    Let $\xi_1,\dots,\xi_{bd}$ be a basis of $\F_r \otimes_{\F_q} \F_r$ over
    $\F_p$.
    Recall that $J_{bd}=\Z/b\Z \times \{1,\dots,bd\} \times \Z_{>0}$. For $k=(i,j,m) \in J_{bd}$ we choose $g_k$ in $\delta_i^F\Omega_{A_{\F_r} \widehat{\otimes} \bfV^\circ} f_i$ such that
	\begin{align*}
		\pr(g_k) &= \xi_j^{p^{-i}}\theta^{m}e_i.
	\end{align*}
	The existence of $g_k$ is guaranteed by Lemma \ref{l: the projection is surjective}. We define the set
	\begin{align*}
		B_i^{mero} &= \{g_k e_i\}_{k \in \{i\} \times \{1,\dots,d\}\times \Z_{>0}}.
	\end{align*}
    Let $I = \Z/b\Z \times \{1,\dots, b(g+d-1)\}$. Let $B_i^{hol}$ be a basis of $\ker(\pr_i)$ over $\bfV^\circ$. Then by Lemma \ref{l: the projection is surjective} the set $B_i^{hol}$ has $b(g+d-1)$ elements. Thus, we can regard $B_i^{hol}$ as a set indexed by $\{i\} \times \{1,\dots,b(g+d-1)\}\subset I$. From Lemma \ref{l: the projection is surjective}
    we see that
    \begin{align*}
	    B := \bigcup_{i=1}^b (B_i^{mero} \cup B_i^{hol})
    \end{align*} 
    is a pole order basis of $D^F M_{\mathscr{F}_{A_{\F_r}}}$ over $\bfV^\circ$
    indexed by $I\sqcup J_{bd}$. 

    \begin{proposition}
        \label{p: we can compute using matrix}
        Let $\Delta$ be the matrix of $U_p \circ E_{\mathscr{F}_{\F_r[\theta]}}$ in terms of the basis $B$. Then
        \begin{align*}
            \NP_\pi(L(\rho_{A_{\F_r}}^{\otimes y}, x^b)^b)&= \NP_\pi(\det(1-x\Delta)).
        \end{align*}
    \end{proposition}
    \begin{proof}
        This follows from Proposition \ref{proposition: computing L function via A0}, Proposition \ref{prop: Fredholm can be computed using pole-ordered bases}, and the fact that $B$ is a pole ordered basis of $D^F M_{\mathscr{F}_{A_{\F_r}}}$ over $\bfV^\circ$.
    \end{proof}
    \subsubsection{Estimates for the matrix \texorpdfstring{$\Delta$}{Delta}}
    The matrix $\Delta$ is a square matrix indexed by $I \sqcup J_{bd}$.
    The following result is a generalization of Proposition \ref{p: matrix entries}.
    \begin{lemma} \label{l: esitames of higher genus matrix}
    	Fix $j$ with $1 \leq j \leq bd$. Let $k_1 \in J_{bd}$ and
    	let $k_2 \in I \sqcup J_{bd}^{(j)}$. 
    	\begin{enumerate}
    		\item If $\bfi{k_1} -1 \neq \bfi{k_2}$, then $\Delta_{k_1,k_2}=0$.
    		\item If $\bfi{k_1} -1 = \bfi{k_2} $, then 
    		\begin{align*}
	    		\Delta_{k_1,k_2} &\equiv \begin{cases}
	    			a_{\bfi{k_1},p|k_1|-|k_2|} \mod \pi^{y_{\bfi{k_1}}(p|k_1|)} & \text{ if }k_1,k_2 \in J_{bd}^{(j)} \\
	    			0 \mod \pi^{y_{\bfi{k_1}}(p|k_1|)} & \text{ otherwise}
	    		\end{cases}.
    		\end{align*}
    	\end{enumerate}
    \end{lemma}
	\begin{proof}
		The first statement follows from the cyclic shape of $E_{\mathscr{F}_{\F_r[\theta]}}$. For the second statement,
		we first assume $k_2 \in J_{bd}^{(j)}$ and write $k_2=(i-1,j,m_2)$. Since $k_1 \in J_{bd}$ we may write
		$k_1 =(i,j',m_1)$. By the definition of $g_{k_2}$ we know 
		\begin{align*}
			g_{k_2} &= \Big[ \xi_{j}^{p^{-(i-1)}} \theta^{m_2} + \sum_{n=0}^\infty c_n \theta^{-n}\Big]  e_{i-1},
		\end{align*}
		where $c_n \in \bfV_{\F_r}$. By the definition of $U_p \circ E_{\mathscr{F}_{\F_r[\theta]}}$ we have 
		\begin{align}\label{eq: computation of operator higher genus}
			U_p \circ E_{\mathscr{F}_{\F_r[\theta]}}(g_{k_2}) &= U_p\Bigg ( \beta_i \Big[ \xi_{j}^{p^{-(i-1)}} \theta^{m_2}  + \sum_{n=0}^\infty c_n \theta^{-n}\Big]\Bigg )e_{i}.
		\end{align}
		By expanding $\beta_i=\sum a_{i,n}$ and applying $U_p$, we see that the coefficient of $\theta^{m_1}e_{i}$ in the right hand side of \eqref{eq: computation of operator higher genus} is 
		\begin{align*}
			a_{i,pm_1- m_2} \xi_{j}^{p^{-i}} + \sum_{n=1}^\infty  a_{i,pm_1 + n}c_{n} & \equiv a_{i,pm_1 - m_2} \xi_{j}^{p^{-i}} \mod \pi^{y_{\bfi{k_1}}(p|k_1|)}.
		\end{align*}
		This proves the second statement for $k_2 \in J_{bd}^{(j)}$. The proof of the second statement for $k_2 \in I$ is almost identical.
	\end{proof}

	\begin{corollary} \label{c: valuations of matrix for higher genus}
		Let $k_1 \in J_{bd}$ and
    	let $k_2 \in I \sqcup J_{bd}^{(j)}$. Assume that $\bfi{k_2} = \bfi{k_1}-1$.
		Then
		\begin{align*}
			v(\Delta_{k_1,k_2}) &= y_{\bfi{k_1}}(p|k_1|-|k_2|)~~ \text{ if }k_1,k_2 \in J_{bd}^{(j)} \text{ and }p|k_1|\geq |k_2|, \\
			v(\Delta_{k_1,k_2}) &\geq  y_{\bfi{k_1}}(p|k_1|) ~~\text{otherwise}.
		\end{align*}
	\end{corollary}
	\begin{proof}
		This follows from Lemma \ref{l: esitames of higher genus matrix} by taking $\pi$-adic valuations. Keep in mind
		that $a_{i,n}$ is defined to be zero for $n<0$.
	\end{proof}
    For this next corollary, we recall the matrix $\Psi^\circ$ defined in \eqref{eq: matrix entries one copy}. We also recall from \S \ref{ss: standard indexing sets} that $\iota:J_{bd} \to J_1$ is the map sending
    $(i,j,m)$ to $(i,m)$. 
    \begin{corollary}
        \label{c: valuations of matrix higher genus permutation}
        Let $(S,\sigma)$ be an enriched permutation of $J_{bd}^{(j)}$. Then we have $
            v(\sigma, \Delta) = v(\iota_*\sigma, \Psi^{\circ})$.
            More precisely,
        \begin{align*}
            \prod_{k \in \iota(S)} \Psi^{\circ}_{k,\iota_*\sigma(k)} \equiv \prod_{k \in S} \Delta_{k,\sigma(k)} \mod \pi^{v(\iota_*\sigma, \Psi^{\circ}) + 1}.
        \end{align*}
    \end{corollary}

    \begin{corollary}\label{c: ordinary determinant claim}
        Assume that $X$ is ordinary. 
        Then $\det(\Delta_{I \times I}) \in (\bfV^{\circ})^\times$
        and $\det(1-s\Delta)$ reduces modulo $\pi$ to a polynomial of 
        degree $b^2(g+d-1)$. 
    \end{corollary}
    \begin{proof}
        By Corollary \ref{c: valuations of matrix for higher genus} we know that for $k_1 \in J_{bd}$ 
        and $k_2 \in I$ we have $v(\Delta_{k_1,k_2})>0$. In particular, 
       \[\Delta\equiv  \begin{bmatrix}
        	\Delta_{I \times I} & * \\ 0 & \Delta_{J_{bd} \times J_{bd}} 
        \end{bmatrix}\mod \pi.\]
        Similarly, using Lemma \ref{l: esitames of higher genus matrix} we see that
        \[\Delta_{J_{bd} \times J_{bd}} \equiv \begin{bmatrix}
        	\Psi^\circ  & \dots & 0 \\ \vdots & \ddots & \vdots \\ 0 & \dots & \Psi^\circ
        \end{bmatrix} \mod \pi.\]
        As $\det(1-x\Psi^\circ) \equiv 1 \mod \pi$ we see that
        \begin{align}\label{eq: reducing matrix mod p 1}
	        \det(1-x\Delta) &\equiv \det(1-\Delta_{I \times I}) \mod \pi.
        \end{align}
        From Proposition \ref{proposition: computing L function via A0} we know
        $L(\rho_{A_{\F_r}}^{\otimes y},x^{b})^b = \det(1 - x\Delta)$. The reduction of $\rho_{A_{\F_r}}$ modulo $\pi$ is the trivial representation. In particular, we have
        \begin{align}\label{eq: reducing matrix mod p 2}
            L(\rho_{A_{\F_r}}^{\otimes y},x) \mod \pi &=  \zeta(\Spec(A),x) \mod p,
        \end{align}
        where $\zeta(\Spec(A),x)$ is the local Weil zeta function of $\Spec(A)$. 
        As $\Spec(A)=X-\infty$ and $X$ has genus $g$, we know $\zeta(\Spec(A),x) = \frac{P(x)}{1-px}$, where $P(x)$ is a polynomial with integer coefficients and degree $2g + d-1$. The ordinary assumption on $X$ implies that
        exactly $g + d-1$ of the roots of $P$ are $p$-adic units. In particular, the zeta function $\zeta(\Spec(A),x)$ reduces to a polynomial of degree $g+d-1$ modulo $p$. 
        Combining \eqref{eq: reducing matrix mod p 1} and \eqref{eq: reducing matrix mod p 2} gives 
        \begin{align*}
            \det(1 - x\Delta_{I \times I}) \mod \pi &= \zeta(\Spec(A),x^b)^b \mod p.
        \end{align*}
        We know the coefficient of $x^{b^2(g+d-1)}$ on the right side
        is a unit, by the previous remarks. Thus the coefficient of $x^{b^2(g+d-1)}$ in $\det(1 - x\Delta_{I \times I})$ is also a unit. The cardinality of $I$ is $b^2(g+d-1)$, so the coefficient of $x^{b^2(g+d-1)}$ in $\det(1 - x\Delta_{I \times I})$ is $\det(\Delta_{I \times I})$. This proves the corollary. 
        
    \end{proof}

    \subsection{\texorpdfstring{$v$}{v}-minimal enriched permutations of \texorpdfstring{$\Delta$}{Delta}} \label{ss: estimating the fredholm determinant}

    In this section we use Theorem \ref{theorem: unique minimal term} together with Lemma \ref{l: esitames of higher genus matrix}
    and its corollaries to
    determine the $v$-minimal enriched permutations of $\Delta$. Recall from Theorem \ref{theorem: unique minimal term} that $\Sigma_n$ is the unique enriched permutation
    of $J_1$ with $\bfs{\Sigma_n}=n$ that is $v$-minimal (it is also the unique $R$-minimal element) for $\Psi^\circ$. Let $\Sigma_n^{(j)}$ be the unique enriched permutation of $J_{bd}^{(j)}$ such that $\iota_*\Sigma_n^{(j)} = \Sigma_n$. 
    
    \begin{definition}
        Let $n_0\geq 0$, let $n=n_0+b^2(g+d-1)$, and write $n_0=c\cdot bd+r$ with $r<bd$. Let $\sigma$
        be an enriched permutation of $I\sqcup J_{bd}$ of size $n$.
        We say $\sigma$ is \emph{distinguished} if there exists
        $H \subset \{1,\dots,bd\}$ with $|H|=r$ such that
        \begin{align*}
            \sigma &= \tau \cdot \prod_{j \in H} \Sigma_{c+1}^{(j)} \cdot \prod_{j \not \in H} \Sigma_c^{(j)},
        \end{align*}
        where $\tau$ is a permutation of $I$ satisfying $v(\tau,\Delta)=0$.
    \end{definition}
    Our main result
     is the following:
    \begin{proposition}\label{p: minimal permutations of N}
        Let $n= n_0 + b^2(g + d -1)$ with $n_0\geq 0$. An enriched permutation
        of $I\sqcup J_{bd}$ with size $n$ is $v$-minimal for $\Delta$ if and only if it is distinguished.
    \end{proposition}
    We break up the proof of Proposition \ref{p: minimal permutations of N} into several smaller lemmas. 
    The first step is to extend the definition of $R(-,-)$ to $I\sqcup J_{bd}$.
    \begin{definition}
        Let $k_1,k_2 \in I\sqcup J_{bd}$. We define
        \begin{align*}
        R(k_1,k_2) &:= \begin{cases}
                    0 & k_1 \in I\\
                    y_{\bfi{k_1}}(p|k_1|) & k_1 \in J_{bd} \text{ and } k_2 \in I \\
                    y_{\bfi{k_1}}(p|k_1|-|k_2|) & k_1,k_2 \in J_{bd} 
                    \end{cases}    
        \end{align*}
        For an enriched permutation $(S, \sigma)$ of $I \sqcup J_{bd}$ we define
        \begin{align*}
            R(\sigma) &= \sum_{k \in S} R(k,\sigma(k)).
        \end{align*}
    \end{definition}
    The following definitions are useful for relating $R$-values and $v$-values. 

    \begin{definition}
        Let $\sigma$ be a rotational enriched permutation of 
        $I\sqcup J_{bd}$. We say that $\sigma$ is \emph{split} if we have 
        $\sigma=\tau\omega$ where $\tau$ is a permutation of $I$ (all of $I$, not just a subset) and
        $\omega$ is an enriched permutation of $J_{bd}$. We say that $\sigma$ is \emph{strongly split} if in addition we have
        $v(\tau,\Delta)=0$, $\omega$ is decomposible, and $\omega$ is $p$-bounded (recall these definitions in \S \ref{ss: some preliminaries on permutations}).
    \end{definition}
     \begin{lemma} \label{l: valuation equals R if sigma is decomposable}
        Let $(S, \sigma)$ be an enriched rotational permutation of $I\sqcup J_{bd}$. Then 
        \begin{align}
            \label{eq: compare R and valuation}
            v(\sigma, \Delta) &\geq R(\sigma).
        \end{align}
        If $\sigma$ is a permutation of $J_{bd}$, then $v(\sigma,\Delta) = R(\sigma)$ if and only if
        $\sigma$ is decomposible and $p$-bounded. 
    \end{lemma}
    \begin{proof}
        We obtain $v(\Delta_{k,\sigma(k))}) \geq R(k,\sigma(k))$ from Corollary
        \ref{c: valuations of matrix for higher genus} and the definition of $R$. This implies \eqref{eq: compare R and valuation}. Next, assume $\sigma$ is a permutation of $J_{bd}$ that is decomposible and $p$-bounded.
        From Corollary \ref{c: valuations of matrix for higher genus} and the definition of $R$ we see 
        that $v(\sigma,\Delta)$ equals $R(\sigma)$. For the converse direction, first assume $\sigma$ is $p$-bounded but not decomposible. Then
        there exists $k \in S$ such that $k \in J_{bd}^{(j)}$ and $\sigma(k) \in J_{bd}^{(j')}$ where $j\neq j'$. From
        Corollary \ref{c: valuations of matrix for higher genus} we know $v(\Delta_{k,\sigma(k)})\geq y_{\bfi{k}}(p|k|)$.
        On the other hand $R(k,\sigma(k)) = y_{\bfi{k}}(p|k|-|\sigma(k)|)$. This implies $v(\Delta_{k,\sigma(k)})>R(k,\sigma(k))$, so that $v(\sigma,\Delta)>R(\sigma)$. Next, assume
        $\sigma$ is not $p$-bounded. Then there exists $k \in S$ such that $|\sigma(k)|>p|k|$. Then $R(k,\sigma(k)) = y_{\bfi{k}}(p|k|-|\sigma(k)|)$ is equal to zero. From Corollary \ref{c: valuations of matrix for higher genus}
        we have $v(\Delta_{k,\sigma(k)})\geq y_{\bfi{k}}(p|k|)$, which is larger than zero. This proves the lemma.
        
    \end{proof}

    \begin{corollary}
        \label{c: split when v is R}
        Assume $\sigma$ is a split enriched permutation of $I\sqcup J_{bd}$. Then $v(\sigma, \Delta)=R(\sigma)$ if and only if
        $\sigma$ is strongly split. In this case, if we write
        $\sigma=\tau\omega^{(1)}\dots\omega^{(bd)}$ where $\tau$ is
        a permutation of $I$ and $\omega^{(j)}$ is an enriched permutation of $J_{bd}^{(j)}$, then \[R(\sigma)=\sum_{i=1}^{bd} v(\iota_*\omega^{(j)}, \Psi^\circ).\]
    \end{corollary}
    \begin{proof}
        Write $\sigma=\tau\omega$ where $\tau$ is a permutation of $I$
        and $\omega$ is an enriched permutation of $J_{bd}$. We have $R(\sigma)=v(\sigma,\Delta)$ if and only if $v(\tau,\Delta)=R(\tau)=0$ and $R(\omega)=v(\omega,\Delta)$. The first statement then follows from Lemma \ref{l: valuation equals R if sigma is decomposable}. The equation for $R(\sigma)$ follows from Corollary \ref{c: valuations of matrix higher genus permutation}.
    \end{proof}

        \begin{lemma}
            \label{l: R-minimal is split}
            Let $(S,\sigma)$ be an $R$-minimal rotational enriched permutation of $I\sqcup J_{bd}$. Then $\sigma$ is split.
        \end{lemma}

    \begin{proof}
        For any $(S,\sigma)$ we define
        \begin{align*}
        K(\sigma)&:=S \cap I, \\
            H(\sigma)&:=S\cap J_{bd}, \\
            L(\sigma) &:=  \{ k \in K(\sigma) ~|~ \sigma(k) \in H(\sigma)\}.
        \end{align*}
        For $i \in \Z/b\Z$ we define
        \begin{align*}
            A_i(\sigma) &:= \{k \in A(\sigma) ~|~\bfi{k}=i \} \text{ where $A$ can be $K$, $H$, or $L$,} \\
            I_i&:= \{k \in I ~|~ \bfi{k}=i\}. 
        \end{align*}
        Finally, taking $*$ to be $i \in \Z/b\Z$ or to be empty, we define
    	\begin{align*}
    		n_*(\sigma) &:= \# \{ k \in K_*(\sigma) ~|~ \sigma(k) \in H(\sigma)\}, \\
    		m_*(\sigma) &:= \# (I_* \backslash K_*(\sigma)).
    	\end{align*}
    	We note that $\sigma$ being split is equivalent to $m(\sigma)+n(\sigma)=0$.

        First, assume $\sigma$ is an $R$-minmal enriched permutation satisfying $n(\sigma)=0$ and $m(\sigma)>0$. Then $\sigma=\tau\omega$, where $\tau$ is an enriched permutation of $I$ and $\omega$ is an enriched permutation of $J_{bd}$. Since $m(\sigma)>0$ we know that $\bfs{\omega}>n_0$. We can find an enriched permutation $\omega_0$ of $J_{bd}$ with size $n_0$ such that $R(\omega_0)<R(\omega)$. Let $\tau_0$ be any permutation of $I$. Then we have \[R(\tau_0\omega_0)=R(\omega_0)< R(\omega)=R(\sigma),\]
        which contradicts the $R$-minimality of $\sigma$. 
     
        Next, assume there exists an enriched permutation $\sigma$ that is $R$-minimal and satisfies $n(\sigma)>0$. 
        Since $n(\sigma)=\sum n_i(\sigma)$, we can always find $i$ such that $n_{i}(\sigma)>0$. In the following three steps we will come to a contradiction.
        
		\begin{step}
			Assume $m_{i-1}(\sigma)>0$ and $n_{i}(\sigma)>0$. In this step we construct $\sigma'$ that is $R$-minimal such that $m(\sigma')=m(\sigma)-1$ and $n(\sigma')=n(\sigma)$. There exists $k_i \in K_i(\sigma)$ such that $k_{i-1}=\sigma(k_i)$ is contained in $H_{i-1}(\sigma)$. Since $m_{i-1}(\sigma)>0$, there exists Also, there exists $k'_{i-1}\in I_i \backslash K_{i-1}(\sigma)$. Define the permutation $(\sigma',S')$  be replacing $k_{i-1}$ with $k'_{i-1}$. More precisely we have $S' = (S\backslash \{k_{i-1}\})\cup \{k'_{i-1}\}$ and
			\begin{align*}
				\sigma'(k) &= \begin{cases}
					k'_{i-1} & k=k_i \\
					\sigma(k_{i-1}) &k=k'_{i-1} \\
					\sigma(k) & k \in S \backslash \{k_{i-1},k_i\}
				\end{cases}.
			\end{align*}
			For $k \in  S \backslash \{k_{i-1},k_i\}$ we have $R(k,\sigma(k))=R(k,\sigma'(k))$. Also,
			$R(k_i,k_{i-1})=0$ and $R(k_i,k'_{i-1})=0$, since $k_i \in I$. 
            As $k'_{i-1} \in I$ we have $R(k'_{i-1},\sigma(k_{i-1}))=0\leq R(k_{i-1},\sigma(k_{i-1}))$. Putting this together, we obtain $R(\sigma')\leq R(\sigma)$, so $\sigma'$ is $R$-minimal. 
            
            By construction we have $m_{i-1}(\sigma')=m_{i-1}(\sigma)-1$ and $m(\sigma')=m(\sigma)-1$. We claim that 
            $n(\sigma')=n(\sigma)$. First, note that $L_i(\sigma')=L_i(\sigma)\backslash \{k_i\}$, so $n_{i}(\sigma')=n_{i}(\sigma)-1$. Also $L_{i'}(\sigma)=L_{i'}(\sigma')$ for all $i'\neq i,i-1$ from the definition of $\sigma'$, so that
            $n_{i'}(\sigma)=n_{i'}(\sigma')$. 
            Thus, we are reduced to showing $n_{i-1}(\sigma') = n_{i-1}(\sigma)+1$. If
            $\sigma(k_{i-1}) \in K_{i-2}(\sigma)$, then $R(k_{i-1},\sigma(k_{i-1}))=y_i(p|k_{i-1}|) > 0$. Since $R(k_{i-1}',\sigma(k_{i-1}))=0$, we must have
            $R(\sigma')<R(\sigma)$, which contradicts the $R$-minimality of $\sigma$. Thus, we have $\sigma(k_{i-1}) \in H_{i-2}(\sigma)$. 
            In particular, we see that $L_{i-1}(\sigma')=L_{i-1}(\sigma) \cup \{k_{i-1}'\}$
            and so $n_{i-1}(\sigma') = n_{i-1}(\sigma)+1$.

		\end{step}
        \begin{step}
            By repeating the first step we may replace $\sigma$ with an $R$-minimal enriched permutation where for some $i \in \Z/b\Z$ with $m_{i-1}(\sigma)=0$ and $n_{i}(\sigma)>0$.
        \end{step}
        
		\begin{step}
			We are now reduced to the case where $m_{i-1}(\sigma)=0$ and $n_{i}(\sigma)>0$. There exists $k_i \in K_i(\sigma)$ with $k_{i-1}=\sigma(k_i) \in H_{i-1}(\sigma)$, as $n_i(\sigma)>0$. Also $K_{i-1}(\sigma)=\{i-1\}\times\{1,\dots,g+d-1\}$, since $m_{i-1}(\sigma)=0$. This means there exists $k_i' \in H_i(\sigma)$
			such that $k_{i-1}' = \sigma(k_i') \in K_{i-1}(\sigma)$. If this were not true, then $\sigma$ would map $K_i(\sigma)$ bijectively to $K_{i-1}(\sigma)$, which we know is not the case. We define $(\sigma',S)$ by swapping $k_{i-1}$ and $k_{i-1}'$. More precisely, 
			\begin{align*}
				\sigma'(k) &= \begin{cases}
					k_{i-1}' & k=k_i \\
					k_{i-1} & k=k_i' \\
					\sigma(k) &k \neq k_i,k_i'
				\end{cases}.
			\end{align*}
            We now compare $R(\sigma)$ and $R(\sigma')$. First note that $R(k_i,\sigma(k_i))=0$ and $R(k_i,\sigma'(k_i))=0$, since $k_i \in I$. Next, note that
            $R(k_i',\sigma(k_i'))=y_{i}(p|k_i'|)$ and $R(k_i',\sigma'(k_i'))=y_{i}(p|k_i'|-|k_{i-1}|)$. In particular, we have $R(k_i',\sigma(k_i'))>R(k_i',\sigma'(k_i'))$. As $R(k,\sigma(k))=R(k,\sigma'(k))$ for $k \neq k_i,k_i'$, we see that $R(\sigma')<R(\sigma)$, which contradicts the fact that $\sigma$ is $R$-minimal.  
		\end{step}
        
    \end{proof}

    \begin{corollary}
        \label{c: there are strongly split R-minimal permutaitons}
        There exists an $R$-minimal permutation of $I\sqcup J_{bd}$
        that is strongly split.
    \end{corollary}
    \begin{proof}
        Let $\sigma$ be an $R$-minimal permutation. By Lemma \ref{l: R-minimal is split} we know that $\sigma$ is split. Write
        $\sigma=\tau\omega$ where $\tau$ is a permutation of $I$ and $\omega$ is an enriched permutation of $J_{bd}$. By Proposition \ref{p: minimal permutations have all the nice properties} there
        exists $\omega'$ that is decomposible and $p$-bounded with
        the same size as $\omega$ satisfying $R(\omega')\leq R(\omega)$.
        We also know from Corollary \ref{c: ordinary determinant claim} that there exists a permutation $\tau'$ of $I$ with $v(\tau',\Delta)=0$. Then $R(\tau'\omega')\leq R(\tau\omega)$, so that $\tau'\omega'$ is $R$-minimal. We also see that $\tau'\omega'$ is strongly split by construction.
    \end{proof}

	\begin{lemma}\label{l: R-minimal strongly split is distinguished}
		Let $n=n_0+b^2(g+d-1)$ with $n_0\geq 0$ and let $\sigma$ be an enriched permutation of $I\sqcup J_{bd}$. Then $\sigma$ is distinguished if and only if it is $R$-minimal and strongly split.
	\end{lemma} 
	\begin{proof}
		We first do the `if' direction. Assume $\sigma=\tau\omega$ is $R$-minimal and strongly split. We have
		$\omega = \omega^{(1)} \dots \omega^{(bd)}$,
		were each $\omega^{(j)}$ is a $p$-bounded permutation of $J_{bd}^{(j)}$. From Theorem \ref{theorem: unique minimal term} and Corollary \ref{c: split when v is R} we must have $\sigma^{(j)}=\Sigma^{(j)}_{n_j}$, where $\sum n_j = \bfs{\omega}$. If $n_{j_1} > n_{j_2}+1$, then the inequality in Theorem \ref{theorem: unique minimal term} 
        tells us that \[ v(\Sigma_{n_{j_1}-1}, \Psi^\circ )+  v(\Sigma_{n_{j_2}+1}, \Psi^\circ) < v(\Sigma_{n_{j_1}}, \Psi^\circ)+  v(\Sigma_{n_{j_2}}, \Psi^\circ).\]
        This combined with Corollary \ref{c: split when v is R} gives
        \[ R(\Sigma^{(j_1)}_{n_{j_1}-1})+  R(\Sigma^{(j_2)}_{n_{j_2}+1}) < R(\Sigma^{(j_1)}_{n_{j_1}})+  R(\Sigma^{(j_2)}_{n_{j_2}}). \]
        This contradicts the $R$-minimality of $\sigma$. It
        follows that $\sigma$ is distinguished. For the `only if' direction, note that all distinguished permutations are strongly split by definition. From Corollary \ref{c: there are strongly split R-minimal permutaitons} and the `if' direction we know that there exists a distinguished permutation that is $R$-minimal. However, all distinguished permutations have the same $R$-value. Thus, they all are $R$-minimal.
	\end{proof}

    \begin{proof}
        (Of Proposition \ref{p: minimal permutations of N})
        Let $\sigma_0$ be distinguished, and thus $R$-minimal and strongly split by Lemma \ref{l: R-minimal strongly split is distinguished}. In particular, we have $R(\sigma_0)=v(\sigma_0, \Delta)$ by Corollary \ref{c: split when v is R}. From Lemma \ref{l: valuation equals R if sigma is decomposable} we see that $\sigma_0$ is $v$-minimal as well. Thus, distinguished implies $v$-minimal. Conversely, let $\sigma$ be $v$-minimal. Using the $R$-minimality of $\sigma_0$ we have
        \[ v(\sigma, \Delta) \geq  R(\sigma)\geq R(\sigma_0) = v(\sigma_0, \Delta).\]
        However we know $v(\sigma,\Delta)=v(\sigma_0,\Delta)$, since both are $v$-minimal. This implies $\sigma$ is $R$-minimal and $v(\sigma, \Delta) =  R(\sigma)$. By Lemma \ref{l: R-minimal is split} we know $\sigma$ is split. Then Corollary \ref{c: split when v is R} tells us that $\sigma$ is strongly split. We see that $\sigma$ is distinguished by Lemma \ref{l: R-minimal strongly split is distinguished}.
    \end{proof}

	\subsection{Finishing the proof of Theorem \ref{t:Riemann-hypothesis}}
	
	We can now finish the proof of Theorem \ref{t:Riemann-hypothesis}. Write
    \begin{align*}
        \det(1-x\Delta) &= 1 + a_1x + a_2x^2 + \dots.
    \end{align*}
    By the cyclic shape of $\Delta$ we know $a_i=0$ for $i$ with $b\nmid i$.
    Also, by Corollary \ref{c: ordinary determinant claim} we know
    $v(a_{b^2(g+d-1)})=0$. 
	Let $n=n_0 + b^2(g + d -1) $ with $n_0\geq 1$ and write $n_0= c\cdot bd + r$ where $r<bd$. From Proposition \ref{p: minimal permutations of N}
	we know
	\begin{align*}
		v(c_{bn}) & \geq (bd-r)v\Big(\Sigma_c, \Psi^\circ \Big ) + rv\Big(\Sigma_{c+1}, \Psi^\circ\Big).
	\end{align*}
    Now assume $bd|n_0$ and let $S$ be the set underlying the enriched permutation $\Sigma_c$ of $J_1$. Then from Corollary \ref{c: valuations of matrix higher genus permutation} and Proposition \ref{p: minimal permutations of N} we have 
	\begin{align*}
		c_{bn} &\equiv \det(\Delta_{I \times I}) \Big( \prod_{k \in S}\Psi^\circ_{k,\Sigma_{c}(k)} \Big)^{bd} \mod \pi^{bd\cdot v(\Sigma_{c}, \Psi^\circ )+1}.
	\end{align*}
	As $\det(\Delta_{I \times I})$ is in $\bfV^\times $ by Corollary \ref{c: ordinary determinant claim} we see that
	\begin{align*}
		v(c_{bn}) &= bd \cdot v\Big(\Sigma_c, \Psi^\circ \Big ).
	\end{align*}
	Recall from Theorem \ref{theorem: unique minimal term} that $\nu_n=v(\Sigma_n, \Psi^\circ) - v(\Sigma_{n-1}, \Psi^\circ)$ satisfies
	$\nu_{n+1}> \nu_n$ for all $n\geq 1$. 
	Thus, we determine the Newton polygon of $\det(1-x\Delta)$ to have slopes
	\[
	        \underbrace{\{0, \dots, 0\}}_{b^2(g - 1 + d)} \sqcup \bigsqcup_{n\geq 1} \underbrace{\Big \{ \frac{\nu_n}{b},\dots,\frac{\nu_n}{b} \Big \}}_{b^2d \text{ times}}.\]
    By Proposition \ref{p: we can compute using matrix} we know that
    $L(\rho^{\otimes y}_{A_{\F_q}}, x^b)^b$ and $\det(1-x\Delta)$ have
    the same Newton polygon. Therefore, the Newton polygon of
    $L(\rho^{\otimes y}_{A_{\F_q}}, x)$ has slopes
    \[
	        \underbrace{\{0, \dots, 0\}}_{g - 1 + d} \sqcup \bigsqcup_{n\geq 1} \underbrace{\Big \{ \nu_n,\dots,\nu_n \Big \}}_{d \text{ times}}.\]
	The zeros of $L(\rho^{\otimes y}_{A}, x)$ are $d$-th roots of the zeros of $L(\rho^{\otimes y}_{A_{\F_q}}, x)$.
	Thus, the Newton polygon of $L(\rho^{\otimes y}_{A}, x)$ has slopes
	\[
	\underbrace{\{0, \dots, 0\}}_{g - 1 + d} \sqcup \bigsqcup_{n\geq 1} \underbrace{\Big \{ \frac{\nu_n}{d}, \dots, \frac{\nu_n}{d} \Big \}}_{d}.
	\]
    From Corollary \ref{c: theorem for affine line} we may write $\nu_n=\alpha_n(r-1)$ where $\alpha_n$ is a positive integer. This completes the proof of Theorem \ref{t:Riemann-hypothesis}.

    \bibliographystyle{refs}
	\bibliography{Goss}

	\Addresses
\end{document}